\documentclass[a4,11pt, pdftex]{amsart}
\usepackage[top=3cm, bottom=3cm, left=2.5cm, right=2.5cm]{geometry}

\usepackage{amssymb,amsmath,amsthm,txfonts}
\usepackage{cite}
\usepackage{nicefrac}

\usepackage[pdftex]{graphicx}
\usepackage{lscape}
\usepackage[hang,small,bf]{caption}
\usepackage[subrefformat=parens]{subcaption}
\newtheorem{thm}{Theorem}[section]
\newtheorem{lem}[thm]{Lemma}

\newtheorem{prop}[thm]{Proposition}
\theoremstyle{definition}
\newtheorem{defn}[thm]{Definition}
\theoremstyle{remark}
\newtheorem{rem}[thm]{\textbf{Remark}}

 \makeatletter
    
    \@addtoreset{equation}{section}
  \makeatother

\makeatletter
      \def\@makefnmark{%
         \leavevmode
            \raise.9ex\hbox{\check@mathfonts
                \fontsize\sf@size\z@\normalfont%
                            \@thefnmark}%
       }
      \makeatother

\makeatletter
\@namedef{subjclassname@2020}{%
  \textup{2020} Mathematics Subject Classification}
\makeatother

\newcommand{\D}{\textrm{div}}

\newcommand{\dd}{\textrm{d}}

\begin{document}

\title[]{Stability of Chandrasekhar's nonlinear force-free fields}
\author[]{Ken Abe}
\date{}
\address[K. ABE]{Department of Mathematics, Graduate School of Science, Osaka Metropolitan University, 3-3-138 Sugimoto, Sumiyoshi-ku Osaka, 558-8585, Japan}
\email{kabe@omu.ac.jp}

\date{\today}

\begin{abstract}
Force-free fields described by $\textrm{curl}\ U=fU$ and $\textrm{div}\ U=0$ appear as coherent structures of the magnetic field $U$ in MHD turbulence as $t\to\infty$. The essential subjects to these equations are \textit{Taylor states}, which are the lowest energy \textit{linear} force-free fields with a constant factor $f\equiv \textrm{const.}$ We establish Taylor state stability of ideal MHD in a bounded and simply connected domain in terms of the existence of weak ideal limits of Leray--Hopf solutions based on Woltjer's principle and Faraco and Lindberg's proof for Taylor's helicity conservation conjecture.
 
On the other hand, Pontin et al. recently observed in computer simulations that turbulent flows relax toward \textit{nonlinear} force-free fields $f\nequiv \textrm{const.}$ rather than Taylor states. Chandrasekhar's nonlinear force-free field, discovered in 1956, is a rare example of a nonlinear force-free field describing a large coherent structure of axisymmetric magnetic field lines with swirls in the presence of Alfv\'en waves. We demonstrate that his traveling wave solution is orbitally stable with the energy norm in ideal MHD regarding weak ideal limits of axisymmetric Leray--Hopf solutions. The energy-norm stability incorporates the strong solution blow-up of ideal MHD and the non-uniqueness of weak ideal limits of Leray--Hopf solutions.

The ideal MHD orbital stability of axisymmetric nonlinear force-free fields is based on the Energy-Casimir method using generalized magnetic helicity. Namely, we consider a new magnetic energy minimization in the whole space and Casimir invariant conservation without using Sobolev regular drifts. The novel property of this \textit{vector} minimization is the minimum strict subadditivity for the generalized magnetic helicity. In contrast, Woltjer's minimum is additive, and the strict subadditivity for the axisymmetric vortex ring problem is unknown. We establish the ideal MHD orbital stability of axisymmetric nonlinear force-free fields based on new ideas on the Grad--Shafranov equation and the usage of magnetic flux Sobolev regularity. This study, inspired by Taylor's relaxation theory, establishes the first stability result for MHS equilibria that differ from Taylor states. 
\end{abstract}

\maketitle

\tableofcontents

\section{Introduction}

The magnetohydrodynamics (MHD) 

\begin{equation}
\begin{aligned}
u_t+u\cdot \nabla  u+\nabla p&=B\cdot \nabla  B+\nu \Delta u,\\
B_t+u\cdot \nabla  B&=B\cdot \nabla  u+\mu \Delta B,\\
\nabla \cdot u=\nabla \cdot B&=0,
\end{aligned} 
\end{equation}\\
describes the velocity field $u(x,t)$, the magnetic field $B(x,t)$, and the total pressure $p(x,t)$ of electrically conducting fluids for $x\in \Omega$ and $t>0$ in a domain $\Omega\subset \mathbb{R}^{3}$, where $\nu>0$ and $\mu>0$ denote viscosity and resistivity. In this paper, we look at the problem (1.1) both for a bounded and simply connected domain $\Omega$ with a $C^{1,1}$-boundary as well as for $\Omega=\mathbb{R}^{3}$.

We consider the non-slip and perfect conductivity conditions when $\Omega$ is bounded, 

\begin{align}
u=0,\quad (\nabla \times B)\times n=0,\quad B\cdot n=0,  
\end{align}\\
for $x\in \partial\Omega$ and $t>0$, where $n$ denotes the unit outward normal vector field on $\partial\Omega$. 

The frozen-field equation $(1.1)_2$ describes the topology-preserving diffusion of magnetic field lines at the zero resistivity limit $\mu=0$ (integral curves of $B$). For the frozen-field equation $(1.1)_2$, we apply the normal trace condition 

\begin{align*}
B\cdot n=0.
\end{align*}\\
The cases $\nu=0$ and $\nu>0$ are ideal MHD and non-resistive MHD, respectively. For ideal MHD, we impose the normal trace condition $u\cdot n=0$.

Ideal MHD (and non-resistive MHD) admits steady states $u=0$ and $B=U$ that satisfy the steady Euler flow 

\begin{equation*}
\begin{aligned}
\nabla \Pi=(\nabla \times U)\times U,\quad \nabla \cdot U&=0\quad \textrm{in}\   \Omega, \\
U\cdot n&=0\quad \textrm{on}\  \partial\Omega,
\end{aligned}
\end{equation*}\\
with the Bernoulli function $\Pi$. Those $U$ that appear in turbulence with vanishing Lorentz force $(\nabla \times U)\times U$ are known as \textit{force-free fields}. The equations for such $U$ are expressed using the proportionality factor $f$ as  

\begin{equation}
\begin{aligned}
\nabla \times U=fU,\quad \nabla \cdot U&=0\quad \textrm{in}\ \Omega, \\
U\cdot n&=0\quad \textrm{on}\ \partial\Omega.
\end{aligned}
\end{equation}\\
The simplest solutions to these equations are \textit{linear} force-free fields, which are eigenfunctions of the rotation operator with eigenvalues $f\equiv \textrm{const.}$ In general, the condition $\nabla \cdot U=0$ implies the first-order equation

\begin{align}
U\cdot \nabla f=0,  
\end{align}\\
and the equations (1.3) are a \textit{nonlinear} system for $U$ and $f\nequiv \textrm{const.}$ Magnetic field lines of $U$ are confined on level sets of $f$, according to the equation (1.4). The primary focus of this study is the \textit{stability} of both linear and nonlinear force-free fields (1.3) in ideal MHD.

\subsection{Taylor states}
In turbulent flows, force-free fields appear as coherent structures. Their stability is based on Taylor's relaxation theory \cite{T74}, \cite{T86}, which is based on total energy and magnetic helicity 

 \begin{align*}
{\mathcal{E}}=\frac{1}{2}\int_{\Omega}\left(|u|^{2}+|B|^{2} \right)\dd x,\quad {\mathcal{H}}=\int_{\Omega}A\cdot B \dd x,  
\end{align*}\\
where $A=\textrm{curl}^{-1}B$ is a unique vector potential such that $\nabla \times A=B$, $\nabla \cdot A=0$ in $\Omega$, $A\times n=0$ on $\partial\Omega$ and $\int_{\Gamma_i}A\cdot n\dd H=0$ for connected components $\Gamma_i$ of $\partial\Omega$, $1\leq i\leq I$ as shown in  \cite[Theorem 3.17]{ABDG}. Woltjer \cite{W58} used magnetic helicity to find \textit{linear} force-free fields by minimizing total energy. Moreau \cite{Moreau} and Moffatt \cite{Moffatt} recognized helicity as a topological quantity measuring knots and links of field lines conserved by the frozen-field equation. More specifically, for the solenoidal space $L^{2}_{\sigma}(\Omega)=\{B\in L^{2}(\Omega)\ |\ \textrm{div}\ B=0\ \textrm{in}\ \Omega,\ B\cdot n=0\ \textrm{on}\ \partial\Omega \}$, minimizers of   

\begin{align}
{\mathcal{I}}_h=\inf\left\{ \frac{1}{2}\int_{\Omega}|B|^{2}\dd x\ \middle|\ B\in L^{2}_{\sigma}(\Omega),\ \int_{\Omega}\textrm{curl}^{-1}B\cdot B\dd x=h  \right\},  
\end{align}\\
are linear force-free fields (Taylor states) \cite[p.1246]{Laurence91}. Rotation is a self-adjoint operator on $L^{2}_{\sigma}(\Omega)$ for a simply connected domain and linear force-free fields are countable with eigenvalues $\cdots\leq f^{-}_{2}\leq f^{-}_1<0<f^{+}_{1}\leq f^{+}_{2}\leq \cdots$ \cite[Theorem 1]{YG90}. Taylor states are a finite number of eigenfunctions associated with the least positive eigenvalue or the largest negative eigenvalue (principal eigenvalues), as demonstrated in Lemma 2.3.

Taylor \cite{T74}, \cite{T86} hypothesized that among all sub-helicities \cite{Moffatt}, only magnetic helicity is approximately conserved with low resistivity in turbulence and used Woltjer's principle as a theoretical foundation for turbulence relaxation toward linear force-free fields as $t\to\infty$. Taylor's theory predicted the relaxed state in a reversed field pinch and other devices successfully \cite{OS93}. Taylor's conjecture is treated mathematically as a question about magnetic helicity conservation at the ideal limit $(\nu,\mu)\to (0,0)$ \cite[p.444]{CKS97}. We refer to Buckmaster and Vicol \cite{BV21} and Faraco et al. \cite{FLS22} for gentle reviews on Taylor's (Woltjer--Taylor) relaxation theory.

Faraco and Lindberg \cite{FL20} provide proof for Taylor's conjecture in terms of \textit{weak ideal limits} of Leray--Hopf solutions. A weak ideal (resp. non-resistive) limit is a weak-star limit of Leray--Hopf solutions to viscous and resistive MHD (1.1)--(1.2) in $L^{\infty}_t L^{2}_x$ as $\nu,\mu\to0$ (resp. as $\mu\to0$ with fixed $\nu>0$). They are fragile notions associated with measure-valued solutions, cf. \cite{DM87m}, \cite{BDS11}, and \textit{conserve} magnetic helicity \cite{FL20} despite the scaling gap to the $L^{3}_{t}L^{3}_{x}$ threshold for magnetic helicity conservation of weak solutions to ideal MHD \cite{Aluie}, \cite{KL07}, \cite{FL18}. See also \cite{FL22} on multiply connected domains. Conservation of low regular quantities at ideal limits is also investigated for other hydrodynamical equations; see \cite{CLNS} for the 2D Euler equations and \cite{CIN} for the SQG equations.

The conservation of total energy and magnetic helicity to weak solutions of ideal MHD is first investigated in \cite{CKS97}, cf. \cite{CET}, \cite{Constantin08} for the Euler equations. The helicity conservation for weak solutions to the Euler equations is investigated in \cite{Constantin08}. In ideal MHD, two regularity thresholds appear for total energy and magnetic helicity conservation, resulting in Onsager-type conjectures \cite{BV21}. Based on the convex integration scheme of De Lellis and Sz\'ekelyhidi, Jr. \cite{DS09}, Faraco et al. \cite{FLS} constructed bounded weak solutions to ideal MHD with compact support in space and time that dissipate total energy with identically vanishing magnetic helicity. See also \cite{BLL15} for the first construction of $2\nicefrac{1}{2}$D weak solutions to ideal MHD. 

By contrast, using Buckmaster and Vicol intermitted convex integration scheme \cite{BV19}, Beekie et al. \cite{BBV} constructed weak solutions of ideal MHD in $L^{\infty}_{t}L^{2}_{x}$ that \textit{do not} conserve magnetic helicity. More specifically, weak solutions constructed in \cite{BBV} increase the absolute value of magnetic helicity, i.e., $2|{\mathcal{H}}(0)|\leq {\mathcal{H}}(1)$. See also \cite{Dai} for non-unique weak solutions to the hall MHD system. Li et al. \cite{YZZ} demonstrated the existence of weak solutions to (hyper) viscous and resistive MHD that do not conserve magnetic helicity and their strong convergence to ideal limits. Based on the convex integration through staircase laminates \cite{Faraco03}, \cite{AFS}, Faraco et al. \cite{FLS21} demonstrated the sharpness of the threshold $L^{3}_{t}L^{3}_{x}$ by constructing weak solutions to the Faraday--Maxwell system in $L^{\infty}_{t}L^{3,\infty}_{x}$ which do not conserve magnetic helicity.
   
Taylor's relaxation theory \cite{T74}, \cite{T86} contains two components: Woltjer's principle and Taylor's conjecture. We begin by noting that proof of Taylor's conjecture for weak ideal limits \cite{FL20}, \cite{FL22} implies Taylor state stability \cite{W58}, \cite{Laurence91}.

\begin{thm}[Taylor state stability]
Let ${\mathcal{S}}_{h}$ be a set of eigenfunctions of the rotation operator on $L^{2}_{\sigma}(\Omega)$ associated with the least positive (resp. largest negative) eigenvalue with magnetic helicity $h>0$ (resp. $h<0$). Let ${\mathcal{S}}_0=\emptyset$. 

The set ${\mathcal{S}}_{h}=\{U_j\}_{j=1}^{N}$ is stable in weak ideal limits of Leray--Hopf solutions to (1.1)--(1.2) in the sense that for arbitrary $\varepsilon>0$, there exists $\delta>0$ such that for $u_0,B_0\in L^{2}_{\sigma}(\Omega)$ satisfying 

\begin{align*}
||u_0||_{L^{2}}+\inf_{1\leq j\leq N}||B_0-U_j||_{L^{2}}+\left|\int_{\Omega}\textrm{curl}^{-1}B_0\cdot B_0\dd x-h\right|\leq \delta,  
\end{align*}\\
there exists a weak ideal limit $(u,B)$ of Leray--Hopf solutions to (1.1)--(1.2) for $(u_0,B_0)$ such that 

\begin{align*}
||u||_{L^{2}}+\inf_{1\leq j\leq N}||B-U_j||_{L^{2}}\leq \varepsilon\quad \textrm{for a.e.}\ t\geq 0.
\end{align*}
\end{thm}

\vspace{5pt}

\begin{rem}
For weak non-resistive limits of Leray--Hopf solutions, the same stability result as Theorem 1.1 holds. The stability result also holds for unique strong solutions to ideal MHD and non-resistive MHD up to maximal existence time; see \cite{CMZ} for local well-posedness of ideal MHD. 
\end{rem}

\begin{rem}[Three Taylor states]
When $\Omega$ is an open ball with radius $R>0$, the three Taylor states ${\mathcal{S}}_{h}=\{U_j\}_{j=1}^{3}$ exist for given helicity $h\neq 0$. They are axisymmetric about the $x_j$-axis. This fact can be observed from the toroidal-poloidal decomposition to linear force-free fields in $\Omega$,  

\begin{align*}
U=\nabla \times (\nabla \times (x\varphi))+\nabla \times (xf \varphi),
\end{align*}\\
by eigenfunctions of the Dirichlet Laplacian satisfying the mean-zero condition on the sphere: 

\begin{align*}
-\Delta \varphi&=f^{2} \varphi, \quad \textrm{in}\ \Omega,\\
\varphi&=0,\quad \textrm{on}\ \partial\Omega,\\
\int_{|x|=\rho}\varphi\dd H&=0,\quad 0<\rho<R.
\end{align*}\\
The smallest $f^{2}$ agrees with the second smallest eigenvalue of the Dirichlet Laplacian $\lambda_2$ because the eigenfunction associated with the principal eigenvalue is radially symmetric and has non-zero mean on the sphere, i.e., $f^{+}_{1}=c_{3/2}/R$ for the first zero point $c_{3/2}=4.4934\cdots$ of the $3/2$-th order Bessel function of the first kind $J_{3/2}$. The multiplicity of $\lambda_2$ is three, and the associated eigenfunctions are expressed as

\begin{align*}
\varphi_{j}&=\frac{1}{\sqrt{|x|}}J_{3/2}(f^{+}_{1}|x|)\frac{x_j}{|x|},\quad j=1,2,3.
\end{align*}\\
We give proof of this fact in Appendix B.
\end{rem}

\vspace{5pt}

When $\Omega$ is multiply connected, the rotation operator on $L^{2}_{\sigma}(\Omega)$ is not self-adjoint. Its spectrum is the point spectrum and agrees with all complex numbers \cite[Theorem 2]{YG90}. Neither Woltjer's principle nor Taylor states are known for this case, cf. \cite[p.1245]{Laurence91}.

When $\Omega=\mathbb{R}^{3}$, linear force-free fields do not exist in $L^{2}_{\sigma}(\mathbb{R}^{3})$ since $-\Delta U=f^{2}U$ and the spectrum of the Laplace operator is the essential spectrum. Nevertheless, there exist linear force-free fields with highly nontrivial topology of magnetic field lines decaying by the order $U=O(|x|^{-1})$ as $|x|\to\infty$ \cite{EP12}, \cite{EP15}. It is known \cite{Na14}, \cite{CC15} that the conditions $U\in L^{q}(\mathbb{R}^{3})$ for $2\leq q\leq 3$ or $U=o(|x|^{-1})$ imply the non-existence of force-free fields in $\mathbb{R}^{3}$.

Taylor \cite{T74}, \cite{T86} originally considered his theory application for the periodic cylinder $\Omega=D\times (-\pi d, \pi d)$ for $D=\{(x_1,x_2)\ |\ x_1^{2}+x_2^{2}<R^{2} \}$ and $d>0$ with the periodic boundary condition at $x_3=\pm \pi d$ \cite[9.1.1]{Biskamp93}, \cite[Example 2.3]{Yeates}. This domain is the simplest multiply connected domain in which the harmonic (constant) vector field $e_z={}^{t}(0,0,1)$ exists. All eigenfunctions to (1.3) for $f\in \mathbb{R}$ are expressed by separation variable solutions using the $m$-th order Bessel function of the first kind $J_m$ for $m\in \mathbb{Z}$. Reiman \cite{Reiman} investigated Taylor states by choosing a particular vector potential of $U$ subject to the constant helicity and the constant toroidal flux $\int_{D}U^{z}\dd x'$. It is known that the Taylor state is axisymmetric for $fR< 3.11$ and helical for $fR\geq  3.11$; see also \cite{Yoshida90}.

\subsection{Chandrasekhar's nonlinear force-free fields}

Recent computer simulations \cite{YRH}, \cite{PWHG}, and \cite{PCRH} demonstrated that turbulent flows relax toward \textit{nonlinear} force-free fields rather than linear force-free fields (Taylor states). The observed nonlinear force-free fields have compactly supported current fields with opposite signs in a uniform magnetic field. Such coherent structures are attributed to integrand $A\cdot B$ redistribution, and sub-helicities are still regarded as essential quantities \cite{Yeates}, \cite{FL22}. They could be used as constraints in an extension of Woltjer's principle to nonlinear force-free fields \cite{T74}, \cite{T86}, \cite{Laurence91}, \cite[Section 4]{Yeates} although such a variational principle is unknown. 

The existence of nonlinear force-free fields (1.3) is linked to their stability. It has long been debated whether nonlinear force-free fields exist besides symmetric solutions. Enciso and Peralta-Salas \cite{EP16} demonstrated that force-free fields do not exist if $f\in C^{2,\alpha}$, $0<\alpha<1$, admits a level set diffeomorphic to a \textit{sphere}. This rigidity result proved non-existence for a wide range of $f$, such as radial or having extrema, and contrasts with rigidity results \cite{Na14}, \cite{CC15} based on the decay of the magnetic field at infinity, such as $U=o(|x|^{-1})$ as $|x|\to\infty$. The nonlinear system $(1.3)$ is an overdetermined problem in general \cite{EP16}, \cite{CK20}, cf. \cite{EP12}, \cite{EP15}, and existence results are available only under symmetry, e.g., \cite{Chandra}, \cite{Tu89}, \cite{A8}. In the axisymmetric setting, both the system (1.3) and the steady Euler flow can be reduced to the Grad--Shafranov equation \cite{Grad}, \cite{Shafranov}; see \cite{Gav}, \cite{CLV}, and \cite{DEPS21} for the existence of compactly supported axisymmetric steady Euler flows. Constantin et al. \cite[p.529]{CDG21b} posed Grad's conjecture \cite[p.144]{Grad67}, which states that non-symmetric steady Euler flows do not exist. With small force \cite{CDG21} or piecewise constant Bernoulli functions \cite{BL96}, \cite{ELP21}, the existence of non-symmetric steady Euler flows is known. See also \cite{CDG22} for more information on the flexibility and rigidity of magnetohydrostatic (MHS) equilibria. The recent works of Pasqualotto \cite{Pasqualotto} and Constantin and Pasqualotto \cite{CP23} constructed steady Euler flows both in tori and bounded domains as long-time limits of solutions to the Voigt--MHD system without assuming any symmetries. In particular, constructed steady states in bounded domains belong to $H^{s}$ for $s\geq 1$ and are not force-free fields \cite[Theorem 1.2 (3)]{CP23}. We discuss magnetic relaxation in subsection 1.5.

The explicit solution of Chandrasekhar \cite{Chandra}, a particular case of Hicks--Moffatt solution \cite{Hicks85}, \cite{Moffatt}, which is an axisymmetric solution with a swirl in $\Omega=\mathbb{R}^{3}$ and a uniform field at infinity, is an essential example of nonlinear force-free fields; see also \cite[2.5.1]{Moffatt19}. In terms of the cylindrical coordinate $(r,\theta,z)$ and Clebsch representation, Chandrasekhar's force-free field is expressed as 

\begin{equation}
\begin{aligned}
&U_C=\nabla \times (\Phi_C\nabla \theta)+G_C\nabla \theta, \quad f_{C}=\lambda^{1/2}1_{(0,\infty)}(\Phi_C),  \\
&\Phi_{C}(z,r)=
\begin{cases}
& \displaystyle\frac{3}{2}Wr^{2}\frac{c_{3/2}^{1/2} J_{3/2}(\lambda^{1/2} \rho)}{J_{5/2}(c_{3/2}) (\lambda^{1/2}\rho)^{3/2}},\qquad \rho=\sqrt{z^{2}+r^{2}}<R,  \\
& \displaystyle-\frac{1}{2}Wr^{2}\left(1-\frac{R^{3}}{\rho^{3}}\right),\hspace{62pt} \rho=\sqrt{z^{2}+r^{2}}\geq R, 
\end{cases}\\
&G_C(z,r)=\lambda^{1/2}\Phi_{C,+}, \qquad R=c_{3/2}\lambda^{-1/2}.
\end{aligned}
\end{equation}\\
The indicator function on $(0,\infty)$ is $1_{(0,\infty)}(s)$ and $s_{+}=s1_{(0,\infty)}(s)$. The strength of the current field $\nabla \times U_C\in L^{\infty}(\mathbb{R}^{3})$ supported in a ball with radius $R$ is denoted by the parameter $\lambda>0$. The parameter $W>0$ denotes the uniform field at infinity, i.e., $U\to -We_z$ as $|x|\to\infty$ for $e_z={}^{t}(0,0,1)$. The factor $f_C$ is a discontinuous function and the level set $f^{-1}_{C}(\lambda^{1/2})$ is a \textit{ball}, cf. \cite{EP16}.\\

\begin{figure}[h]
\begin{minipage}[b]{0.45\linewidth}
\hspace{35pt}
\includegraphics[scale=0.12]{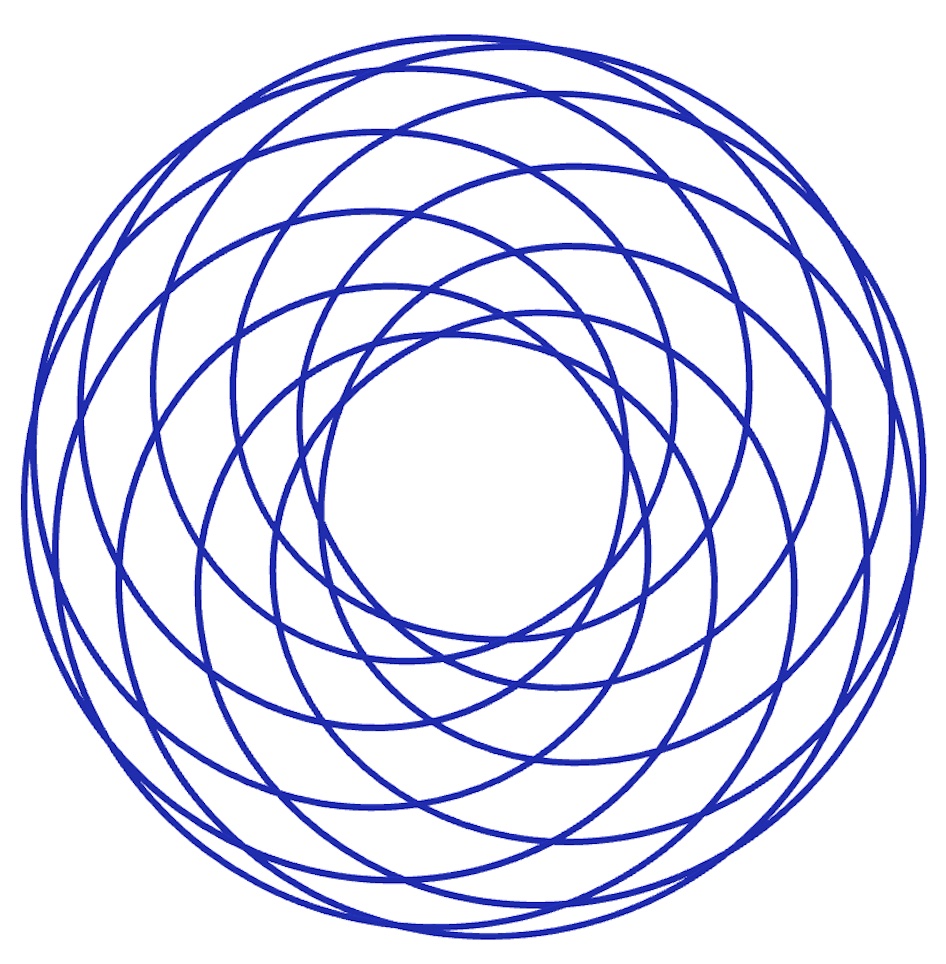}
  \end{minipage}
  \begin{minipage}[b]{0.45\linewidth}
\hspace{20pt}
\includegraphics[scale=0.13]{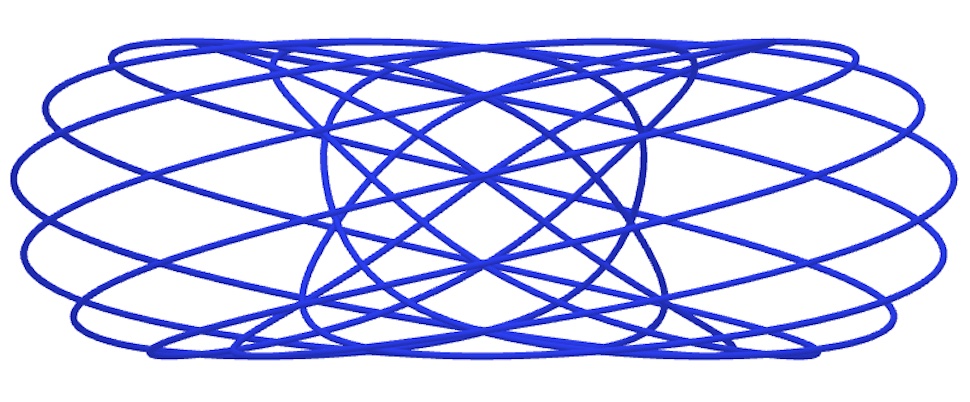}
  \vspace{30pt}
  \end{minipage}\\
  \caption{The (8,11)-torus knots: the field lines have 8 circuits in the toroidal direction and 11 circuits in the poloidal direction}
\end{figure}

Chandrasekhar's nonlinear force-free field inside a ball agrees with the  Taylor state symmetric about the $x_3$-axis in Remark 1.3. Its magnetic field lines are integrable in the sense that the ball $\{x\in \mathbb{R}^{3}\ |\ |x|<R\}$ is fibered into the nested tori $\{x\in \mathbb{R}^{3}\ |\ \Phi_C(z,r)=k\}$ for $k>0$ with field lines making torus knots and links. It is known \cite{Moffatt19} that this torus knot is close to the (8,11)-torus knot; see Figure 1.

The explicit solution (1.6) has two distinguishing features: the compactly supported current field $\nabla \times U_C$ and the uniform magnetic field at infinity $U_{C}\to -We_z$ as $|x|\to\infty$. The first property appears also in the computer simulation of turbulent relaxation \cite{PWHG}, \cite{YRH}, and \cite{PCRH}. We remark that the current fields of linear force-free fields cannot be compactly supported by the unique continuation principle for solutions to the Helmholtz equation $-\Delta U=f^{2}U$, e.g., \cite{Wolff93}. (This fact in $\Omega=\mathbb{R}^{3}$ also can be observed from the Liouville theorems \cite{Na14}, \cite{CC15}.) 

The second property is relevant to Alfv\'en waves \cite{Alf}. Namely, the MHD equations (1.1) with the uniform magnetic field at infinity have an aspect of the one-dimensional wave equations. Indeed, small disturbances to the uniform magnetic field propagate along the field lines and are swept away to infinity in the evolution of ideal MHD (and also viscous and resistive MHD)\cite{BSS}, \cite{He18}, and \cite{Lei18}. The explicit solution (1.6) can form a traveling wave solution to ideal MHD describing a large, robust, coherent structure of magnetic field lines in MHD turbulence in the presence of the Alfv\'en wave. We elaborate on this property in more detail below in subsection 1.3.

This paper investigates the orbital stability of the nonlinear force-free field (1.6) in the ideal MHD equations (1.1) subject to a uniform field condition at infinity. Under axisymmetry, the system (1.1) for $\mu=0$ has an additional Casimir invariant called \textit{generalized magnetic helicity} \cite{Moffatt97}, cf. \cite{KPY}. Based on ideas inspired by Taylor's relaxation theory \cite{W58}, \cite{Laurence91}, \cite{FL20}, \cite{FL22} and Energy-Casimir methods developed for the Euler orbital stability of vortex pairs/axisymmetric vortex rings without swirls \cite{BNL13}, \cite{B21}, \cite{AC22}, and  \cite{Choi22}, we consider a new approach to the ideal MHD orbital stability for axisymmetric nonlinear force-free fields with swirls including the explicit solution (1.6).

\subsection{The main result}

We consider the system (1.1) subject to the uniform field condition 

\begin{align}
(u,B)\to (u_{\infty}, B_{\infty})\quad \textrm{as}\ |x|\to\infty,   
\end{align}\\
for given constants $B_{\infty}=-We_z$, $W>0$ and $u_\infty \in \mathbb{R}^{3}$ parallel to $e_z$ (The system (1.1) for $u_{\infty}\neq 0$ can be reduced to the case $u_{\infty}=0$ by the Galilean transformation $u(x,t)\longmapsto \tilde{u}(x,t)=u(x-u_{\infty}t,t )+u_{\infty}$ and $B(x,t)\longmapsto \tilde{B}(x,t)=B(x-u_{\infty}t ,t)$.) In terms of the Els\"asser fields $Z^{\pm}=u\pm b$ for $B=b+B_{\infty}$, the system (1.1) (with $u_{\infty}=0$) can be expressed as 

\begin{equation*}
\begin{aligned}
Z^{+}_t-B_{\infty}\cdot \nabla  Z^{+}+Z^{-}\cdot \nabla  Z^{+}+\nabla p&=\frac{1}{2}(\nu+\mu) \Delta Z^{+}+\frac{1}{2}(\nu-\mu) \Delta Z^{-},\\
Z^{-}_t+B_{\infty}\cdot \nabla  Z^{-}+Z^{+}\cdot \nabla  Z^{-}+\nabla p&=\frac{1}{2}(\nu-\mu) \Delta Z^{+}+\frac{1}{2}(\nu+\mu) \Delta Z^{-},\\
\nabla \cdot Z^{\pm}&=0.
\end{aligned} 
\end{equation*}\\
The second terms in the equations transport the Els\"asser fields in the opposite directions $\pm B_{\infty}$ and have an aspect of a system of one-dimensional wave equations. It is known \cite{BSS}, \cite{He18}, and \cite{Lei18} that the system (1.1) for $B_{\infty}\neq 0$ ($u_\infty\in \mathbb{R}^{3}$) and $\nu=\mu\geq 0$ admits a unique global-in-time solution $u=v+u_{\infty}$ and $B=b+B_{\infty}$ for small and smooth initial disturbance $(v_0,b_0)$ and the solution approaches a trivial state as $t\to \infty$. 

We study the long-time behavior of \textit{large} axisymmetric solutions to (1.1) for $B_{\infty}\neq 0$, $\nu\geq 0$, and $\mu= 0$. The explicit solution (1.6) offers a specific traveling wave solution to (1.1) for $\mu= 0$,  

\begin{align*}
u=u_{\infty},\quad B(x,t)=U(x-u_{\infty}t).
\end{align*}\\
The profile $U$ is a force-free field satisfying

\begin{equation}
\begin{aligned}
\nabla \times U=fU,\quad \nabla \cdot U&=0\qquad \textrm{in}\ \mathbb{R}^{3}, \\
U&\to B_{\infty}\quad \textrm{as}\ |x|\to\infty.
\end{aligned}
\end{equation}\\
In terms of an equivalent system for the new variables $(v,b)=(u-u_{\infty},B-B_{\infty})$, we state a stability result for the explicit solution (1.6) in the system (1.1) subject to the condition (1.7) :

\begin{equation}
\begin{aligned}
v_t+(v+u_{\infty})\cdot \nabla  v+\nabla p&=(b+B_{\infty})\cdot \nabla  b+\nu \Delta v,\\
b_t+(v+u_{\infty})\cdot \nabla  b&=(b+B_{\infty})\cdot \nabla  v+\mu \Delta b, \\
\nabla \cdot v=\nabla \cdot b&=0.
\end{aligned} 
\end{equation}\\
We apply Clebsch representation for axisymmetric solenoidal vector fields 

\begin{align*}
b=\nabla \times (\phi \nabla \theta)+G \nabla \theta,
\end{align*}\\
with unique Clebsch potentials $\phi(z,r)$ and $G(z,r)$; see Section 3. Potential of $B_{\infty}=-We_z$ is $-\phi_{\infty}$ for $\phi_{\infty}=Wr^{2}/2+\gamma$ and arbitrary constant $\gamma$. The total energy and magnetic helicity of the system (1.1) subject to the condition (1.7) are expressed as 

\begin{align*}
{\mathcal{E}}=\frac{1}{2}\int_{\mathbb{R}^{3}}\left(|v|^{2}+|b|^{2}\right) \dd x, \quad {\mathcal{H}}=
2\int_{\mathbb{R}^{3}}(\phi-\phi_{\infty})\frac{G}{r^{2}} \dd x. 
\end{align*}\\
Unlike simply connected bounded domains, magnetic helicity in $\Omega=\mathbb{R}^{3}$ is ill-defined \cite[Appendix A]{FLS}. Namely, there exist some $B\in L^{2}_{\sigma}(\mathbb{R}^{3})$ whose magnetic helicity diverge. We instead apply generalized magnetic helicity

\begin{align}
H=2\int_{\mathbb{R}^{3}}(\phi-\phi_{\infty})_{+}\frac{G}{r^{2}} \dd x. 
\end{align}\\
The generalized magnetic helicity agrees with the magnetic helicity for $G$ supported in $\{\phi>\phi_{\infty}\}$. Indeed, the generalized magnetic helicity of the explicit solution (1.6) is the magnetic helicity 

\begin{align}
h_C=2\lambda^{1/2}\int_{\mathbb{R}^{3}}\Phi_{C,+}^{2}\frac{1}{r^{2}} \dd x >0.  
\end{align}\\
The constant $h_C$ is quadratic for $W/\lambda$ \cite[p.128]{Moffatt}, as shown in Proposition 8.4. We demonstrate that the generalized magnetic helicity is well-defined for axisymmetric $b=B-B_{\infty}\in L^{2}_{\sigma}(\mathbb{R}^{3})$ and constants $W>0$ and $\gamma\geq 0$ (The generalized magnetic helicity (1.10) is affected by the gauge $\gamma\geq 0$.)

For small axisymmetric disturbances $(v_0,b_0+B_{\infty}-U_{C})$, we show that there exists a weak ideal limit $(v,b)$ such that $(u,B)=(v+u_{\infty},b+B_{\infty})$ is close to the traveling wave solution $(u_{\infty},U_C(x-u_{\infty}t ))$ up to translation in $z$ for all time. This demonstrates the stability of MHS equilibrium that differs from Taylor states.\\

\begin{thm}[Stability of Chandrasekhar's nonlinear force-free fields]
Let $\lambda, W>0$. Let $B_{\infty}=-We_z$ and $\phi_{\infty}=Wr^{2}/2$. Let $u_{\infty}$ be a constant parallel to $e_z$. 

The force-free field $U_C$ in (1.6) with the magnetic helicity $h_C$ is orbitally stable in weak ideal limits of axisymmetric Leray--Hopf solutions to (1.1) subject to the condition (1.7) in the sense that for arbitrary $\varepsilon >0$ there exists $\delta >0$ such that for axisymmetric $v_0, b_0\in L^{2}_{\sigma}(\mathbb{R}^{3})$ satisfying

\begin{align*}
||v_0||_{L^{2}(\mathbb{R}^{3})}
+\inf_{z\in \mathbb{R}}||b_0+B_{\infty}-U_{C}(\cdot +ze_z) ||_{L^{2}(\mathbb{R}^{3})}
+\left|2\int_{\mathbb{R}^{3}} \left(\phi_0-\phi_{\infty}   \right)_{+}\frac{G_0}{r^{2}} \dd x -h_C\right| \leq \delta, 
\end{align*}\\
for the Clebsch potentials $\phi_0$, $G_0$ of $b_0$, there exists a weak ideal limit $(v,b)$ of axisymmetric Leray--Hopf solutions to (1.9) for $(v_0, b_0)$ such that

\begin{align*}
||v||_{L^{2}(\mathbb{R}^{3})}
+\inf_{z\in \mathbb{R}}||b+B_{\infty}-U_{C}(\cdot +ze_{z}) ||_{L^{2}(\mathbb{R}^{3})} \leq \varepsilon, \quad \textrm{for a.e.}\ t\geq 0.
\end{align*}
\end{thm}

\vspace{5pt}

\begin{rem}
The same stability result as Theorem 1.4 also holds for weak non-resistive limits of axisymmetric Leray--Hopf solutions and all time, as well as for unique strong axisymmetric solutions to ideal MHD and non-resistive MHD up to maximal existence time. Theorem 1.4 is a particular case $\gamma=0$ of a general stability theorem for a specific class of nonlinear force-free fields with discontinuous factors $f\in L^{\infty}(\mathbb{R}^{3})$ (Theorem 8.2).
\end{rem}

\begin{rem}[Vortex rings with swirls]
For the 3D Euler equations, the orbital stability of axisymmetric vortex rings with swirls is unknown. Indeed, the force-free (Beltrami) field (1.6) also offers a traveling wave solution to the 3D Euler equations (vanishing at infinity),

\begin{align*}
u_E=U_C(x+B_{\infty}t)-B_{\infty},   
\end{align*}\\
which describes an axisymmetric vortex ring with swirls whose vortex is supported in a ball moving at the speed of $B_{\infty}$. (Vortex lines of this solution are identical to magnetic field lines of Chandrasekhar's solution; see Figure 1.) The solution $u_E$ belongs to the lower regularity space $W^{1,\infty}(\mathbb{R}^{3})$ than the space $C^{1,\alpha}(\mathbb{R}^{3})$, $\alpha>0$, for which the Euler equations are locally well-posed \cite{Lichtenstein}, \cite{Gunther}. 
\end{rem}

The novelty of Theorem 1.4 might be the stability of the magnetic field with swirls in the ideal MHD equations. In contrast, the stability/instability of vortex rings with swirls in the 3D Euler equations has been an open question. Research on vortex rings has attracted attention concerning the Cauchy problems of the Euler and the Navier--Stokes equations. The recent breakthrough of Elgindi \cite{Elgindi} shows the existence of backward self-similar blow-up solutions to the Euler equations for some axisymmetric data without swirls $u_0\in C^{1,\alpha}(\mathbb{R}^{3})$ and small $\alpha>0$. In the study of weak solutions, Albritton et al. \cite{ABC} demonstrates the existence of non-unique axisymmetric Leray--Hopf solutions without swirls for the forced Navier--Stokes equations based on the instability of unbounded vortex; see also Vishik \cite{Vishik18}, \cite{Vishik18b}. 

The energy-norm stability in Theorem 1.4 does not contradict the strong solution blow-up to the ideal MHD equations. On the other hand, two (stable and unstable) weak ideal limits of axisymmetric Leray--Hopf solutions to viscous and non-resistive MHD can exist for merely square integrable initial disturbance.

\subsection{Stability: Euler and ideal MHD}

The Euler equations and the ideal MHD equations have different stability principles. We review stability results for the two equations relevant to Theorems 1.1 and 1.4.

\subsubsection{Taylor state stability}

Theorem 1.1 is based on Taylor's relaxation theory. We briefly review stability results until the early 1990s, when Woltjer's principle was established. \\

The concept of force-free field stability appeared in the 1950s in the work of Lundquist \cite{Lundquist}. Lundquist was motivated by Alfv\'en's observation \cite{Alfven50} on the instability of magnetic fields twisted by liquid motion and investigated the stability of linear force-free fields in a cylinder as a particular case of MHS equilibria; see also \cite{Trehan}. A little later, L\"{u}st and Schl\"{u}ter \cite{LS54} and Chandrasekhar \cite{Chandra} recognized the force-free field importance as a natural model of cosmic magnetic fields that explains a large current flowing in the space surrounding stars. Chandrasekhar and Woltjer \cite{CW58} and Woltjer \cite{Woltjera}, \cite{Woltjerb}, \cite{Woltjerc}, \cite{Woltjerb}, \cite{W58} investigated the linear force-free field stability (Chandrasekhar's explicit solution (1.6) was discovered in \cite{Chandra} before Woltjer's minimum energy principle work \cite{W58}.) Variational principles to other MHS equilibria are investigated in \cite{Woltjerb}, \cite{Woltjerc}, \cite{Woltjerd}, and \cite{Chandrasekhar1958}. After Lundquist's work and Chandrasekhar's series of works on the stability of MHS equilibria, Bernstein et al. \cite{Bernstein1958} established a nonlinear stability theorem for MHS equilibria with a sign condition for second derivatives of total energy; see also \cite{VC} and \cite[IX.]{Chandra61}.

In 1965, Arnold \cite{Arnold1965} established a well-known nonlinear stability theorem for steady states in the Euler equations with a sign condition for second derivatives of kinetic energy, i.e., local maxima or minima in equivortical velocity fields, cf. \cite{Bernstein1958}. Arnold's stability theorem is useful for 2D steady flows \cite{MP94}, \cite{AK98}, cf. \cite{WP85}, \cite{SV09}. By contrast, he stated \cite[p.1007]{Arnold1965}, \cite[p.349]{Arnold66} that he could not find any examples of 3D steady flows that met his criterion. Rouchon \cite[Theorem 1.1]{Rouchon} demonstrated that all 3D steady flows are kinetic energy saddle points and concluded that Arnold's criterion is never satisfied; see also \cite[9.3.3]{Moffatt21}. From the 1960s to the 1970s, the helicity concept was developed by Moreau \cite{Moreau} and Moffatt \cite{Moffatt} and incorporated into Taylor's relaxation theory \cite{T74}, \cite{T86}.

Later in the 1980s, Moffatt \cite[p.374]{Moffatt85}, \cite[p.369]{Moffatt85b}  observed that the analogy of the existence of steady flows between Euler and ideal MHD does not extend to their stability; see \cite[9.3]{Moffatt21}. He computed the sign of second derivatives for magnetic energy near the shear flow $G(r)\nabla \theta$, a steady axisymmetric magnetic/velocity field with a swirl. The sign condition implies that $G$ is stable in the ideal MHD equations for decreasing $|G/r^{2}|$ in $r>0$, e.g., $G=r^{-\alpha}$, $\alpha>1$. By contrast, such $G$ is unstable in the Euler equations by Rayleigh's criterion \cite{Rayleigh1917}, \cite[Chapter 3]{DR82}.

In the early 1990s, Laurence and Avellaneda \cite{Laurence91} revisited Woltjer's principle. Woltjer \cite{W58} assumed the restrictive condition to the magnetic field $\textrm{curl}^{-1}\ B=0$ on $\partial\Omega$ and the present form (1.5) is due to \cite{Laurence91}; see also \cite[p.15]{Yeates}. Around the same time, the work \cite{YG90} characterized the spectrum of the rotation operator. 

Taylor state stability (Theorem 1.1) is based on Woltjer's principle and Taylor's conjecture. The present form of Woltjer's principle (1.5) appeared in \cite{Laurence91}, and proof of Taylor's conjecture is due to the recent work of Faraco and Lindberg \cite{FL20}. Theorem 1.1 is the first stability theorem in ideal MHD using helicity, though the proof follows Taylor's conjecture. Three Taylor states in a ball (Remark 1.3) seem first noticed in this study. It questions whether one Taylor state is stable without restricting the solution class.

\subsubsection{Stability of Chandrasekhar's nonlinear force-free fields}

Theorem 1.4 is relevant to the stability of vortex rings in the Euler equations. \\

Research on the stability of vortex rings can be traced back to the works of Kelvin \cite{kelvin1880} and Benjamin \cite{Ben76}. Benjamin \cite[p.20]{Ben76} considered the stability principle for axisymmetric vortex rings without swirls by maximizing kinetic energy by rearranging a vortex with an impulse constraint. Wan \cite{Wan86} investigated the stability of Hill's spherical vortex rings.  The Euler and the ideal MHD equations are noncanonical Hamiltonian PDEs, and little is known about the orbital stability of traveling wave solutions; see \cite{GSS} and \cite{LZ22} for the Grillakis--Shatah--Strauss Hamiltonian PDEs stability theory.  

The Energy-Casimir method is an available stability principle for 2D and axisymmetric steady states and is different from the stability principles for general non-symmetric solutions \cite{Bernstein1958}, \cite{Arnold1965}; see \cite{Holm85}, \cite{BPM19}, \cite[6.10]{BV22}, and \cite{Rein23} for reviews. The Energy-Casimir method is based on additional conservation (Casimir) and developed in the study of other plasma/fluid systems such as the Vlasov--Poisson equations \cite{Guo99}, \cite{GR01} or the Euler--Poisson equations \cite{Rein03}, \cite{LS09}. The instability results in the Euler--Poisson equations can be found in \cite{jang08}, \cite{jang14}.

The MHD stability criteria via the Energy-Casimir method have been developed in the works of Holm et al. \cite{Holm85}, Moffatt \cite{Moffatt85}, \cite{Moffatt85b}, Friedlander and Vishik \cite{FV90}, Vladimirov and Moffatt \cite{Moffatt95}, and Vladimirov et al. \cite{Moffatt96}, \cite{Moffatt97}, \cite{Moffatt99}. See also \cite{FV95}, \cite{VF98}. The work \cite[Criterion 5.2]{Moffatt97} investigates the stability of axisymmetric MHS equilibria.

Axisymmetric vortex rings with swirls have been considered unstable in contrast to the case without swirls. Instability research can be found in the late 1980s in the works of Szeri and Holmes \cite[Theorem 4.2]{SH88} and Lifschitz and Hameiri \cite{LH93}. Around the same time, research on the spectrum of the linearized Euler operator began in the works of Friedlander and Vishik \cite{FV91}, \cite{FV92}, Lifschitz and Hameiri \cite{LH91}, \cite{LH93}, Vishik \cite{V96}; see also Shvydkoy and Friedlander \cite{FS05}. The nonlinear instability results can be found in \cite{Lin04}, \cite{BGS02}, \cite{VF03}, and the existence of unstable manifolds can be found in \cite{LZ13}, \cite{LZ14}, and \cite{LZ22}.

The orbital stability of vortex pairs and axisymmetric vortex rings without swirls has been investigated in recent works. Burton et al. \cite{BNL13} demonstrated the orbital stability of vortex pairs in the 2D Euler equations based on Benjamin's maximization of kinetic energy using a rearrangement with unknown vorticity functions; see also \cite{B21}. The author and Choi \cite{AC22} investigated the orbital stability of relatively restricted class vortex pairs using the minimization of (penalized) enstrophy with prescribed vorticity functions. This stability method also applies to the stability of axisymmetric vortex rings without swirls. It should be noted that stability results \cite{BNL13}, \cite{B21}, and \cite{AC22} are demonstrated for weak solutions without assuming boundedness of vorticity. A particular type of vortex ring without swirls is Hill's spherical vortex rings, which are patch-type solutions; see \cite{Choi22} for the stability result in the Yudovich class \cite{UI}. 

We remark that asymptotic stability, which is convergence to equilibrium as $t\to\infty$, is a more powerful stability concept than orbital (Lyapunov) stability. Bedrossian and Masmoudi \cite{BM15} established the nonlinear asymptotic stability of the Euler equations for 2D Couette flows and disturbances in Gevrey classes. This stability mechanism, inviscid damping, is related to Landau damping in plasma physics \cite{MV11}. See also \cite{IJ20}, \cite{IJ22}, \cite{MZ20}, and \cite{IJ22b} for nonlinear asymptotic stability results. We also discuss recent works \cite{RZ17}, \cite{ZZZ22}, and \cite{LMZZ22} on linear asymptotic stability for sheared velocity and magnetic fields in ideal MHD.

The swirl stability has been revisited in recent works. Gallay and Smets \cite{GalleySmets18}, \cite{GalleySmets19} investigated the linearized Euler equations in the vicinity of the columnar vortex $G(r)\nabla \theta$; see \cite{Gallay20} for a review. Albritton and O\.{z}a\'{n}ski \cite{AO} investigated the instability of the columnar vortices with nonzero axial flows $G(r)\nabla \theta+W(r)\nabla z$. The recent breakthrough of Guo et al. \cite{GHPW}, \cite{GPW} showed that the columnar vortex $r^{2}\nabla \theta=(-x_2,x_1,0)$ is asymptotically stable for slight and smooth axisymmetric disturbance via the Euler--Coriolis equations.

In summary, the stability or instability of axisymmetric vortex rings with swirls in the Euler equations has been an open question, though some orbital stability results are available for vortex pairs/axisymmetric vortex rings without swirls \cite{BNL13}, \cite{B21}, \cite{AC22}, and \cite{Choi22}. Chandrasekhar's nonlinear force-free field offers axisymmetric vortex rings with swirls for the Euler equations and axisymmetric magnetic fields with swirls for the ideal MHD equations. This work demonstrates the stability of axisymmetric magnetic fields with swirls, including Chandrasekhar's explicit solution in the ideal MHD equations. 

The most relevant work to Theorem 1.4 may be the stability criteria for axisymmetric MHS equilibria using the Energy-Casimir method \cite{Moffatt97}. Taylor's relaxation theory inspires the stability principle of Theorem 1.4 and is different from the stability criteria to general MHS equilibria \cite{Moffatt97}. We discuss the ideas of the proof in subsection 1.6.

\subsubsection{Vortex and magnetic fields}

We finally discuss the relationship between vortex and magnetic field evolutions. According to Taylor's relaxation theory, MHD turbulence does not conserve sub-helicities with small resistivity $0<\mu<<1$ due to the reconnection of magnetic field lines. The equations

\begin{align*}
\partial_t B+\nabla\times (B\times u)=\mu \Delta B,
\end{align*}\\
describe the evolution of magnetic field lines with the velocity $u$ obeying the Navier--Stokes equations with Lorenz force. This equation is the frozen-field equation at zero resistivity $\mu=0$. Namely, magnetic field lines are frozen in fluid, i.e., the fluid elements lying on a magnetic field line continue to lie on the same magnetic field line. Moreover, magnetic flux on a closed surface moving with fluid is also conserved. These properties, Alfv\'en's theorem \cite[4.3.1]{Davidson}, are analogs to Helmholtz's first law and Kelvin's theorem for vorticity fields. The reconnection event of magnetic field lines is relevant to the blow-up of strong solutions and the non-uniqueness of weak solutions.

The vortex evolution of the Navier--Stokes equations is described by 

\begin{align*}
\partial_t \omega+\nabla\times (\omega\times u)=\nu \Delta \omega.
\end{align*}\\
It is demonstrated in \cite{CSTY1}, \cite{CSTY2}, and \cite{KNSS} that axisymmetric solutions of the Navier--Stokes equations do not exhibit self-similar blow-ups (For smooth axisymmetric data without swirls, unique global-in-time solutions exist \cite{La59}, \cite{UI}, \cite{LMNP}.) The works \cite{FengSverak}, \cite{GallaySverak}, and  \cite{GallaySverak2} showed the existence and uniqueness of large viscous axisymmetric vortex rings without swirls for vortex filament initial data. See also \cite{BGH} and \cite{BG} for non-symmetric vortex filament solutions and \cite{Seis17} for the relationship between binormal curvature flow. The recent work of Gallay and {\v{S}}ver\'{a}k \cite{GS23} describes the detailed behavior of viscous axisymmetric vortex rings without swirls in the small viscosity regime $0<\nu <<1$ with the Kelvin--Saffman formula for the speed of viscous vortex rings by extending Arnold's stability approach to viscous flows; see also \cite{GS21}.

The non-uniqueness of finite energy weak solutions to the Navier--Stokes equations is demonstrated in the groundbreaking work of Buckmaster and Vicol \cite{BV19} by the intermitted convex integration. See also Buckmaster et al. \cite{BCV22}, Luo \cite{Luo19}, and Cheskidov and Luo \cite{CL22}. Jia, {\v{S}}ver\'{a}k, and Guillod \cite{JS15}, \cite{JS}, \cite{GuS} investigated the non-uniqueness of Leray--Hopf solutions to the Navier--Stokes equations based on the spectral stability of self-similar solutions. The non-uniqueness of forced axisymmetric Leray--Hopf solutions without swirls is demonstrated in \cite{ABC} based on the instability of self-similar solutions, cf. \cite{ABCDGK}. \\

\subsection{Remarks on magnetic relaxation}

The total energy of solutions to non-resistive MHD, i.e., (1.1)--(1.2) for $\mu=0$, $\nu>0$, decreases with the velocity field $\nabla u\in L^{2}(0,\infty; L^{2})$. Arnold's inequality, conversely, limits magnetic energy from below for initial data with nonzero magnetic helicity; see (2.4). By letting $t\to \infty$, one may construct steady Euler flows for a given $B_0$. This means of constructing steady Euler flows is called \textit{magnetic relaxation} posed by Arnold \cite{Arnold74}, \cite[ChapterIII]{AK98} and Moffatt \cite{Moffatt85}, \cite[Section 8]{Moffatt21}. The equations of the velocity field do not have to be the Navier--Stokes equations, and other models are considered in \cite{Moffatt85}, \cite{Moffatt90}, \cite{Vallis}, \cite{Ni02}, \cite{Nunez}, \cite{Brenier14}, \cite{Pasqualotto}, and \cite{BFV21}.

In Remark 1.2, weak non-resistive limits are near Taylor states for a.e. $t\geq 0$. It does not, however, imply convergence to a Taylor state as $t\to\infty$. Magnetic relaxation can produce a broader range of steady flows than that of force-free fields \cite{Kom21}. Remember that some $B_0$ can form nontrivial links with magnetic helicity $h=0$ \cite[p.367]{Moffatt85}, \cite[8.1.1]{Moffatt21}, \cite[p.2]{Kom21}. Komendarczyk \cite{Kom21} investigated magnetic energy minimization for given $B_0\in L^{2}_{\sigma}(\Omega)$ among weak $L^{2}$-closure of pushforward $B=\varphi_*B_0$ with non-increasing energy for volume-preserving diffeomorphism $\varphi$ of $\Omega$ whose restriction on $\partial\Omega$ is identity, i.e., 

\begin{align*}
\tilde{{\mathcal{I}}}(B_0)=\inf\left\{\frac{1}{2}\int_{\Omega}|B|^{2}\dd x\ \middle|\ B\in \overline{M}^{w}(\Omega, B_0) \right\},
\end{align*}\\
for $M(\Omega, B_0)=\{B\in L^{2}_{\sigma}(\Omega)|\ B=\varphi_*B_0,\  \varphi\in \textrm{Diff} (\Omega, \dd x),\ ||B||_{L^{2}}\leq ||B_0||_{L^{2}}  \}$. The work \cite{Kom21} demonstrated that for some $B_0\in L^{2}_{\sigma}(\Omega)$ with helicity $h=0$, a set of minimizers $\tilde{{\mathcal{S}}}(B_0)$ of $\tilde{{\mathcal{I}}}(B_0)$ is not empty, i.e., $\tilde{{\mathcal{S}}}(B_0)\neq \emptyset$. This means that for such $B_0$, magnetic relaxation minimum energy states are nontrivial, whereas Taylor states are trivial, i.e., ${\mathcal{S}}_0=\emptyset$. 

Beekie et al. \cite{BFV21} obtained detailed information on the large-time behavior for the velocity field satisfying 

\begin{align*}
\nabla p=B\cdot \nabla B-\nu(-\Delta)^{\kappa} u. 
\end{align*}\\
The frozen-field equation $(1.1)_2$ with the above velocity field (the MRE equations) is shown to be globally well-posed for $B_0\in H^{s}(\mathbb{T}^{n})$, $s,\kappa>n/2+1$, satisfying $\nabla \cdot B_0=0$, and the strong convergence limits $\lim_{t\to\infty}||\nabla u||_{L^{\infty}}=0$ holds \cite[Theorems 3.1, 4.1]{BFV21}. In the 2D setting, asymptotic stability of the steady state $B=e_1$ and $u=0$ is obtained, as are examples of 3D exact solutions exhibiting an exponential growth of the current field \cite[Theorems 5.1, 6.4]{BFV21}. See also \cite{Elgindi17} and \cite{CCL19} for asymptotic stability of the IPM equations and \cite{Elgindi20} for the exponential growth of the vorticity field of the 3D Euler equations. 

Recently, Constantin and Pasqualotto \cite{CP23} constructed steady Euler flow both in $\mathbb{T}^{3}$ and in a bounded domain by magnetic relaxation via the Voigt-MHD system. The Voigt-MHD system is a system that modifies the terms $u_t$ and $B_t$ in the non-resistive MHD (1.1) for $\nu>0$ and $\mu=0$ by $\mathbb{L}^{\alpha}u_t$ and $\mathbb{L}^{\alpha}B_t$ for the Stokes operator $\mathbb{L}$ and $\alpha>0$ (with Dirichlet boundary conditions for bounded domains.) It is shown in \cite{CP23} that both in $\mathbb{T}^{3}$ and in a bounded domain, for arbitrary initial data $(u_0, B_0)\in D(\mathbb{L}^{\alpha/2})$ and $\alpha\geq 1$ satisfying compatibility conditions, there exists a unique global-in-time solution to the Voigt-MHD system $(u,B)$, a steady state $B_{\infty}$, and a sequence $\{t_n\}$ satisfying $t_n\to\infty$ such that $B(\cdot, t_n)$ converges to $B_{\infty}$ weakly in $D(\mathbb{L}^{\alpha/2})$ and strongly in $D(\mathbb{L}^{\gamma/2})$ for $\gamma<\alpha$. The Voigt-MHD system conserves the modified helicity, and constructed steady states are non-trivial for initial data with non-zero modified helicity. The steady states are not force-free for bounded domains due to the Dirichlet boundary condition to the magnetic field.

Enciso and Peralta-Salas \cite{EPS23} investigate the relationship between initial conditions and end states in magnetic relaxation for the axisymmetric toroidal domain $\Omega$ in a class of braided fields,

\begin{align*}
{\mathcal{B}}(\overline{\Omega})=\{B_0\in C^{\infty}(\overline{\Omega})\cap L^{2}_{\sigma}(\Omega)\ |\ B_0\cdot \nabla \theta>0\ \}.
\end{align*}\\
It is shown in \cite[Theorem 1.1]{EPS23} that for an open subset of ${\mathcal{B}}(\overline{\Omega})$, there exists a concretely characterized dense set such that initial data in this set cannot be deformed into smooth MHS equilibrium by a volume-preserving diffeomorphism. The work \cite[Theorem 2.7]{EPS23} also provides a sufficient initial condition for relaxation toward linear force-free fields.

We close this remark by discussing the relaxation in the 2D Euler equations. The long-time behavior of solutions to the 2D Euler equations is relevant to 2D turbulence, and the question of how solutions are relaxed is a subject of active research; see Khesin et al. \cite{KMK23} for a review. The compactness of vorticity is relevant to the end state of the turbulence \cite{GSV15}, and it is conjectured in \cite{Sh13}, \cite{SverakLec} that vorticity of the 2D Euler equations is generically not compact as $t\to\infty$. The numerical works \cite{MV20}, \cite{MV22} indicate merging the identical sign vortices and turbulence relaxation toward a point vortex motion. Mathematically rigorous results for such relaxation are unknown \cite[Problem 27]{KMK23}. 

A similarity between the long-time behavior of magnetic relaxation and the 2D Euler equations is that both magnetic and vorticity fields can lose continuity at the end states. For magnetic relaxation, the current sheet formation is observed in, e.g., \cite[8.2.2]{Moffatt21}, and exemplified in \cite[Remark 6.3]{BFV21}. For the 2D Euler equations, the relaxation toward singular vortex is investigated in \cite{Elgindi2022} for forward self-similar solutions.

\subsection{Ideas and main ingredients}

The following are the main ideas for proving Theorem 1.4:\\ 

\noindent
(a) Total energy minimization under the constraint of conserved generalized magnetic helicity, \\
\noindent
(b) Existence of weak ideal limits of axisymmetric Leray--Hopf solutions with non-increasing total energy and conserved generalized magnetic helicity.\\

For the stability of traveling wave solutions in the 2D Euler equations, we can consider (a) enstrophy minimization under the kinetic energy constraint \cite{AC22}, cf. \cite{BNL13}, \cite{B21}, and (b) the existence of global weak solutions with conserved enstrophy and kinetic energy. Enstrophy conservation is due to the renormalized property of vorticity equations \cite{LNM06}, cf. \cite{CS15}. See \cite{CLNS} for more information on energy conservation. Due to additional Casimir invariants, axisymmetric vortex rings with no swirl have a similar minimization \cite{FT81} and a renormalized property \cite{NS22}. We take advantage of the analogy between 2D hydrodynamic and 3D MHD turbulence \cite{Hasegawa85}, \cite[7.3]{Biskamp93}.

\subsubsection{Taylor state stability: the toy model}

The simpler case is Taylor state stability (Theorem 1.1). Part (a) corresponds to Woltjer's principle \cite{W58}, \cite{Laurence91}, and part (b) to Taylor's conjecture \cite{FL20}, \cite{FL22}. We exploit from the part (a), (i) \textit{compactness of minimizing sequences}. Namely, we use the fact that any sequences $\{B_n\}\subset L^{2}_{\sigma}(\Omega)$ satisfying $\frac{1}{2}\int_{\Omega}|B_n|^{2}\dd x\to {\mathcal{I}}_h$, $\int_{\Omega}\textrm{curl}^{-1}B_n\cdot B_n\dd x\to h$ are relatively compact in $L^{2}(\Omega)$. The minimum $\mathcal{I}_{h}$ has the explicit form as shown in (2.6), i.e., 

\begin{align*}
\mathcal{I}_{h}
=
\begin{cases}
\ \mathcal{I}_{1}h & h\geq 0,\\
\ -\mathcal{I}_{-1}h & h< 0.
\end{cases}
\end{align*}\\
When $\Omega$ is a ball, the minimum $\mathcal{I}_{h}$ is an even function. We emphasize that the minimum $\mathcal{I}_{h}$ is additive, i.e., 

\begin{align*}
\mathcal{I}_{h_1+h_2}=\mathcal{I}_{h_1}+\mathcal{I}_{h_2},\quad  h_1,h_2>0.
\end{align*}\\
A key point of (b) is (ii) \textit{compactness of Leray--Hopf solutions}. For Leray--Hopf solutions $(u_j,B_j)$ to (1.1)--(1.2) satisfying the equality 

\begin{align}
\int_{\Omega}\textrm{curl}^{-1}B_j\cdot B_j\dd x+2\mu_j\int_{0}^{t}\int_{\Omega}\nabla \times B_j\cdot B_j\dd x\dd s=\int_{\Omega}\textrm{curl}^{-1}B_{0}\cdot B_0\dd x,
\end{align}\\
the vector potentials $\{\textrm{curl}^{-1}B_j\}$ are relatively compact in $L^{2}_{\textrm{loc}}(0,T; L^{2}(\Omega))$ by the Aubin--Lions lemma as $(\nu_j,\mu_j)\to (0,0)$. By letting $j\to\infty$, the magnetic helicity of weak ideal limits is conserved. (The energy inequality eliminates the second term.)

In weak ideal limits, the two points (i) and (ii) imply the stability of a set of minimizers (Taylor states) ${\mathcal{S}}_h$ to $\mathcal{I}_h$. Theorem 1.1 is derived by characterizing Taylor states as a finite number of eigenfunctions of the rotation operator associated with the principal eigenvalues using Rayleigh's formulas (Lemma 2.3). 

\subsubsection{The new ingredients}
To prove the stability of the nonlinear force-free field (1.6) (Theorem 1.4), in part (a), we consider a new minimization principle: 

\begin{equation}
\begin{aligned}
&I_{h,W,\gamma}=\\
&\inf\left\{\frac{1}{2}\int_{\mathbb{R}^{3}}|b|^{2}\dd x\ \middle|\  b=\nabla \times (\phi\nabla \theta)+G\nabla \theta\in L^{2}_{\sigma,\textrm{axi}}(\mathbb{R}^{3}),
 2\int_{\mathbb{R}^{3}}(\phi-\phi_{\infty})_{+}\frac{G}{r^{2}} \dd x=h\ \right\},   
\end{aligned}
\end{equation}\\
for given constants $h\in \mathbb{R}$, $W>0$, and $\gamma\geq 0$, where $L^{2}_{\sigma,\textrm{axi}}(\mathbb{R}^{3})$ denotes the space of all axisymmetric solenoidal vector fields in $L^{2}_{\sigma}(\mathbb{R}^{3})$ and $\phi_{\infty}=Wr^{2}/2+\gamma$. The minimization problem (1.13) is well-defined by the Arnold-type inequality to the generalized magnetic helicity $|H[b]|\lesssim ||b||_{L^{2}}^{8/3}$  as shown in Lemma 3.10, cf. (2.4). Minimizers of (1.13) produce axisymmetric nonlinear force-free fields (1.8) for $B_{\infty}=-We_z$ with discontinuous factors $f\in L^{\infty}(\mathbb{R}^{3})$ (Lemma 4.3). This appears to be the first variational principle that provides nonlinear force-free fields using the conservation of ideal MHD. Previous constructions use un-conserved quantities to ideal MHD \cite{Tu89} or a minimax method for flux functions \cite{A8}. 

In the part (b), a key quantity is \textit{generalized magnetic mean-square potential},

\begin{align}
\int_{\mathbb{R}^{3} }(\phi-\phi_{\infty} )_{+}^{2}\dd x.
\end{align}\\
Faraco and Lindberg \cite[Theorem 5.4]{FL20} demonstrated the conservation of magnetic mean-square potential at weak ideal limits of Leray--Hopf solutions for 2D bounded and multiply connected domains. We show the conservation of generalized magnetic mean-square potential for weak ideal limits of axisymmetric Leray--Hopf solutions for fixed initial data.

\subsubsection{The difficulties of the proof}

Theorem 1.4 is the first stability result for traveling wave solutions to ideal MHD in $\mathbb{R}^{3}$. The significant differences from the Euler orbital stability of traveling wave solutions (vortex pairs/axisymmetric vortex rings without swirls) \cite{BNL13}, \cite{B21}, \cite{AC22}, and \cite{Choi22} are the following: \\

\begin{itemize}
\item Minimization of magnetic (vector) fields in $\mathbb{R}^{3}$
\item Flux transport for merely square integrable drifts\\
\end{itemize}

The \textit{vector} minimization problem for axisymmetric magnetic fields with swirls (1.13) is first studied in this work. The orbital stability of vortex pairs/axisymmetric vortex rings without swirls \cite{BNL13}, \cite{B21}, \cite{AC22}, and \cite{Choi22} are based on \textit{scalar} minimization problems. We obtain the first strict subadditivity result for the minimum $I_h$ to the vector minimization problem (1.13).

The second significant difference appears in the transport structures of vorticity/magnetic fields. The vorticity of the 2D Euler equations, as well as (the normalized) vorticity of the 3D axisymmetric Euler equations without swirls, is transported by Sobolev regular drift, yielding the renormalized property for the Casimir invariant conservation. By contrast, the magnetic flux of axisymmetric ideal MHD with swirls is transported by merely square integrable drifts with no Sobolev regularity. We show, nevertheless, that the magnetic flux is a renormalized solution and conserves Casimir invariants.

In the subsequent sub-subsections, we discuss the two main ideas for the difficulties above.

\subsubsection{The strict subadditivity: the first main idea}

The approach to proving (i) compactness of minimizing sequences for $I_{h}=I_{h, W,\gamma}$ in (1.13) is the strict subadditivity of the minimum 

\begin{align}
I_{h_1+h_2}<I_{h_1}+I_{h_2},\quad h_1,h_2>0.
\end{align}\\
This strict subadditivity for the vector field minimization problem is first obtained in this study. For scalar function minimization problems, Lions \cite{Lions84a}, \cite{Lions84b} discovered the strict subadditivity importance for the compactness of a minimizing sequence in an unbounded domain. More specifically, the strict subadditivity of the minimum is a crucial property to exclude the possibility of a scalar-minimizing sequence dichotomy in the application of his concentration-compactness lemma and obtain the compactness of the minimizing sequence; see, e.g., \cite[Part 4]{Pava}. This concentration-compactness principle to scalar functions is widely used in plasma stability research \cite{Rein23}.

The vector field minimization problem (1.13) may be included in an extensive application of Lions's theory. However, the compactness of vector-minimizing sequences, especially the dichotomy absence, is a non-trivial question. The main issues are the validity of the strict subadditivity (1.15) and, secondly, the dichotomy exclusion for a vector-minimizing sequence. The first issue arises because a scalar function argument does not apply to (1.13) due to the sign-changing integrand of the generalized magnetic helicity. The second issue has been investigated for a (non-negative) scalar function minimization problem to vortex pairs/axisymmetric vortex rings without swirls \cite{BNL13}, \cite{B21}, \cite{AC22}, and \cite{Choi22} in which the dichotomy of a minimizing sequence in $\mathbb{R}^{2}_{+}$ is excluded. However, the strict subadditivity of the minimum for those problems has yet to be discovered.

The strict subadditivity of the minimum (1.15) for the whole space problem (1.13) is the opposite nature to the additivity of the minimum $\mathcal{I}_h$ for the bounded domain problem (1.5). The desired property of $I_h$ to exclude the dichotomy possibility of a vector-minimizing sequence (which can occur only in $\mathbb{R}^{3}$) is \textit{strict} subadditivity (The additivity does not exclude the dichotomy possibility.) We emphasize that the toy model (1.5) is available only for (simply connected) bounded domains and is not extendable for the whole space in which magnetic helicity can diverge \cite{FLS} and no force-free fields on $L^{2}_{\sigma}(\mathbb{R}^{3})$ exist \cite{Na14}, \cite{CC15}.

The clue to the dichotomy exclusion for vector-minimizing sequences to the problem (1.13) is the compactness of a minimizing sequence in $\mathbb{R}^{2}_{+}$ in the vortex pairs problem. The work \cite{AC22} used the following minimization principle yielding translating vortex pairs (traveling waves):\\

\begin{quote}
Minimize enstrophy (Casimir) minus kinetic energy subject to non-negative vortices in $\mathbb{R}^{2}_{+}$ under the constraint on impulse $= h$ and mass $\leq 1$. \\
\end{quote}
The minimum of this principle is monotone for $h>0$, but the strict subadditivity is unknown due to the multiple constraints \cite[Remarks 2.4]{AC22}. The work \cite{AC22} directly excluded the dichotomy of a non-symmetric minimizing sequence by using the existence of horizontally symmetric minimizers with compactly supported vortex. The dichotomy absence from minimizing sequences may be consistent with numerically observed 2D turbulence relaxation toward merging vortices \cite{MV20}, \cite{MV22}, and \cite{KMK23}. We emphasize that the compactness of a \textit{non-symmetric} minimizing sequence is essential for the orbital stability of traveling waves since the horizontal symmetry is not preserved for the stability solution class.

We show that, in contrast to vortex problems, the strict subadditivity (1.15) does hold for the magnetic field vector minimization problem (1.13) in $\mathbb{R}^{3}$, and that by using it, the dichotomy of a (non-symmetric) vector-minimizing sequence does not occur. A heuristic idea is to look at the Euler--Lagrange equation to the problem (1.13): $G=\kappa(\phi -\phi_{\infty})_{+}$ for $\phi_{\infty}=Wr^{2}/2+\gamma$, $W>0$, $\gamma\geq 0$ with a Lagrange multiplier $\kappa\in \mathbb{R}$ and the Grad--Shafranov equation

\begin{align}
-L\phi=\kappa^{2}(\phi-\phi_{\infty} )_{+}\quad \textrm{in}\ \mathbb{R}^{2}_{+},\quad \phi=0\quad \textrm{on}\ \partial\mathbb{R}^{2}_{+},
\end{align}\\
where $\mathbb{R}^{2}_{+}=\{{}^{t}(z,r)\ |\ z\in \mathbb{R},\ r>0\}$ and $L=\partial_z^{2}+\partial_r^{2}-r^{-1}\partial_r$. The minimizer provides the  axisymmetric nonlinear force-free field $U=b+B_{\infty}$ and $f=\kappa 1_{(0,\infty)}(\phi-\phi_{\infty})$ whose helicity is expressed as 

\begin{align*}
h=2\kappa \int_{\mathbb{R}^{3}}(\phi-\phi_{\infty})_{+}^{2}\frac{1}{r^{2}}\dd x. 
\end{align*}\\
This integrand is positive, and the strict subadditivity (1.15) follows the existence of minimizers to (1.13), as shown in Lemma 4.10.  

The technical idea for demonstrating the existence of minimizers to the problem (1.13) is to minimize the magnetic energy in a smaller class than that in (1.13). Namely, we consider a smaller set $T_{h}$ for $h\neq 0$ consisting of all axisymmetric solenoidal vector fields $b=\nabla \times (\phi\nabla \theta)+G\nabla \theta$ such that 

\begin{align*}
\phi\geq 0,\quad \int_{\mathbb{R}^{3}}(\phi-\phi_{\infty})_{+}^{2}\frac{1}{r^{2}}\dd x>0,\quad G=\kappa (\phi-\phi_{\infty})_{+},\quad \kappa=\frac{h}{2\int_{\mathbb{R}^{3}}(\phi-\phi_{\infty})_{+}^{2}\frac{1}{r^{2}}\dd x}.
\end{align*}\\
The set $T_{h}$ is identified with a set of non-negative scalar functions in $\mathbb{R}^{2}_{+}$ and it turns out that by investigating the properties of minimizers to the problem (1.13), the set $T_h$ is larger than the set of minimizers and hence the minimum $I_h$ is invariant by the restriction to $T_h$. Moreover, the constant $I_h$ is also invariant by the restriction to the set $\tilde{T}_{h}$ consisting of symmetric and non-increasing functions in $z$, i.e., $\phi(z,r)=\phi(-z,r)$, by Steiner symmetrization, i.e., $2 I_{h}=\inf_{b\in \tilde{T}_h}||b||_{L^{2}}^{2}$. We show the existence of symmetric minimizers by estimating symmetric functions for large $|x|$ (Theorem 4.9) and deduce that the minimum $I_h$ is a continuous, monotone, and strictly subadditive even function for $h>0$ and lower semicontinuous at $h=0$.

By using these properties of $I_h$ and the concentration--compactness principle, we show that any non-symmetric vector-minimizing sequences $\{b_n\}\subset L^{2}_{\sigma,\textrm{axi}}(\mathbb{R}^{3})$ with potentials $\phi_n$ and $G_n$ satisfying

\begin{align*}
\frac{1}{2}\int_{\mathbb{R}^{3}}|b_n|^{2}\dd x\to I_h,\quad 2\int_{\mathbb{R}^{3}}(\phi_n-\phi_{\infty})_{+}\frac{G_n}{r^{2}}\dd x\to h,
\end{align*}\\
are relatively compact in $L^{2}(\mathbb{R}^{3})$ up to translation in $z$. We use the Lipschitz continuity of generalized magnetic helicity $H[\cdot]: L^{2}_{\sigma,\textrm{axi}}(\mathbb{R}^{3})\to \mathbb{R}$ (Proposition 5.4).

\subsubsection{The global flux convergence: the second main idea}

To explain the idea in (ii) the compactness of Leray--Hopf solutions for generalized magnetic helicity conservation, we denote the axisymmetric Leray-Hopf solution to (1.9) with viscosity $\nu_j>0$ and resistivity $\mu_j>0$ by $(v_j,b_j)$ and its weak ideal limit by $(v,b)$. We denote Clebsch potentials of $b_j$ and $b$ by $(\phi_j, G_j)$ and $(\phi, G)$, respectively. The axisymmetric Leray--Hopf solution satisfies the generalized magnetic helicity equality (Lemma 6.7)

\begin{align}
\int_{\mathbb{R}^{3}}\Phi_{j,+}\frac{G_j}{r^{2}}\dd x
+\mu_j\int_{0}^{t}\int_{\mathbb{R}^{3}}\nabla \times B_j\cdot B_j 1_{(0,\infty)}(\Phi_{j} )\dd x\dd s=\int_{\mathbb{R}^{3}}\Phi_{0,+}\frac{G_0}{r^{2}}\dd x, \end{align}\\
for $\Phi_j=\phi_j-\phi_{\infty}$ and $B_j=b_j+B_{\infty}$. By the definition of the weak ideal limit and Aubin--Lions lemma, 
$r^{-1}G_j \overset{\ast}{\rightharpoonup} r^{-1}G$ in $L^{\infty}(0,T; L^{2}(\mathbb{R}^{3}) )$ and 

\begin{align*}
\Phi_{j}\to  \Phi \quad \textrm{in}\ L^{2}(0,T; L^{2}_{\textrm{loc}}(\mathbb{R}^{3})).
\end{align*}\\
The second term in the equality (1.17) is eliminated at the weak ideal limit by the energy inequality, as in the case of bounded domains. However, the local flux convergence is insufficient for the helicity conservation at the weak ideal limit because of the unboundedness of the domain.

A similar problem arises in the case of the vorticity equations to the 2D Navier--Stokes equations in $\mathbb{R}^{2}$

\begin{align*}
\partial_t\omega_j+u_j\cdot \nabla \omega_j=\nu_{j}\Delta \omega_j.
\end{align*}\\
The velocity $u_j$ is determined by the vorticity $\omega_j$ via the Biot--Savart law. The vorticity $\omega_j$ is continuous on $L^{2}(\mathbb{R}^{2})$ at time zero and satisfies the enstrophy equality

\begin{align*}
\int_{\mathbb{R}^{2} }|\omega_j|^{2}\dd x+2\nu_j \int_{0}^{t}\int_{\mathbb{R}^{2}}|\nabla \omega_j|^{2}\dd x\dd s=\int_{\mathbb{R}^{2} }|\omega_0|^{2}\dd x.
\end{align*}\\
At the vanishing viscosity limit, $\omega_j$ converges to $\omega$ weakly-star in $L^{\infty}(0,T; L^{2}(\mathbb{R}^{2}))$ and the limit satisfies the transport equation 

\begin{align*}
\partial_t \omega+u \cdot \nabla \omega=0.
\end{align*}\\
The renormalized property of the limit $\omega$ thanks to Sobolev regularity of $u\in L^{\infty}(0,T; \dot{H}^{1}(\mathbb{R}^{2}) )$ \cite[p.360]{LNM06} is a crucial fact strengthening the weak convergence of the vorticity $\omega_j$. The renormalized property implies the enstrophy conservation at the vanishing viscosity limit  

\begin{align*}
\int_{\mathbb{R}^{2} }|\omega|^{2}\dd x=\int_{\mathbb{R}^{2} }|\omega_{0}|^{2}\dd x,
\end{align*}\\
and therefore a weak convergence of $\omega_j$ is in fact a stronger convergence 

\begin{align*}
\omega_j\to \omega\quad  \textrm{in}\ L^{2}(0,T; L^{2}(\mathbb{R}^{2}) ).
\end{align*}\\
Sobolev regularity of the drift is critical for the renormalized property of transport equation solutions \cite[Corollary II.2]{DL89}, \cite[p.1213]{AC14} and is used in \cite{LNM06} for enstrophy conservation at vanishing viscosity limits, cf. \cite{CS15}.

The limit velocity $u=v+u_{\infty}$ of the (viscous and resistive) axisymmetric MHD is, by contrast, merely in $L^{\infty}(0,T; L^{2}_{\textrm{loc}}(\mathbb{R}^{3}))$ and lacks Sobolev regularity. The flux $\Phi_j$ of the axisymmetric MHD equations satisfies the drift-diffusion equation 

\begin{align*}
\partial_t \Phi_{j}+u_j\cdot \nabla \Phi_{j}=\mu_j\left(\Delta-\frac{2}{r}\partial_r \right)\Phi_j, 
\end{align*}\\
for the Navier--Stokes velocity field $u_j=v_j+u_{\infty}$ with Lorenz force and the generalized mean-square potential equality (Lemma 6.8)

\begin{align}
\int_{\mathbb{R}^{3} }|\Phi_{j,+}|^{2}\dd x+2\mu_j \int_{0}^{t}\int_{\mathbb{R}^{3}}|\nabla \Phi_{j,+}|^{2}\dd x\dd s=\int_{\mathbb{R}^{3} }|\Phi_{0,+}|^{2}\dd x.
\end{align}\\
At the weak ideal limit, $\Phi_{j}$ converges to $\Phi$ in $L^{2}(0,T; L^{2}_{\textrm{loc}}(\mathbb{R}^{3}))$ and $u_j$ converges to $u$ weakly-star in $L^{\infty}(0,T; L^{2}_{\textrm{loc}}(\mathbb{R}^{3}))$. 

The main observation is that, despite merely square integrable drifts in weak ideal limits, the limit flux possesses Sobolev regularity $r^{-1}\nabla \Phi\in L^{\infty}(0,T; L^{2}_{\textrm{loc}}(\mathbb{R}^{3}))$ thanks to the total energy bound. This Sobolev regularity implies that the nonlinear term $u\cdot \nabla \Phi$ is integrable and that the flux satisfies the transport equation 

\begin{align*}
\Phi_{t}+u\cdot \nabla \Phi=0.
\end{align*}\\
For the 2D MHD system in a bounded domain, the flux of a weak ideal limit is a weak solution to the transport equation \cite[Lemma 5.6]{FL20} in contrast to the 3D case. More specifically, the work \cite{FL20} showed that any solutions to the transport equation in $L^{\infty}_t\dot{H}^{1}_{x}$ conserve magnetic mean-square potential with divergence-free drifts in $L^{\infty}_{t}L^{2}_x$. We show that the flux of weak ideal limits to axisymmetric Leray--Hopf solutions in $\mathbb{R}^{3}$ is a weak solution to the transport equation and conserves the generalized mean-square potential 

\begin{align}
\int_{\mathbb{R}^{3} }|\Phi_{+}|^{2}\dd x=\int_{\mathbb{R}^{3} }|\Phi_{0,+}|^{2}\dd x.
\end{align}\\
Heuristically, Sobolev regularity of $\Phi$ (in $x$ and $t$) implies the renormalized property 

\begin{align*}
g(\Phi)_{t}+u\cdot \nabla g(\Phi)=0,
\end{align*}\\
for regular $g(s)$ by the chain rule in Sobolev space. The generalized mean-square potential conservation (1.19) strengthens the local flux convergence to the \textit{global} convergence 

\begin{align}
\Phi_{j,+}\to  \Phi_{+}\quad \textrm{in}\ L^{2}(0,T; L^{2}(\mathbb{R}^{3})).
\end{align}\\
This more robust convergence implies the generalized magnetic helicity conservation at the weak ideal limit

\begin{align}
\int_{\mathbb{R}^{3}}\Phi_{+}\frac{G}{r^{2}}\dd x
=\int_{\mathbb{R}^{3}}\Phi_{0,+}\frac{G_0}{r^{2}}\dd x.
\end{align}

\vspace{5pt}

\subsubsection{The gauge dependence of generalized magnetic helicity}

Theorem 1.4 is based on the usage of the gauge-dependent generalized magnetic helicity (1.10). For simply connected bounded domains, magnetic helicity is unaffected by the vector potential change $A\longmapsto A+\nabla\varphi$ (gauge-invariant). For multiply connected bounded domains, MacTaggart and Valli \cite{MV19} discovered the gauge-invariant definition of magnetic helicity; see also \cite[p.5, (6)]{FL22}. 

In Section 3, we show that axisymmetric solenoidal vector fields $b\in L^{2}_{\sigma,\textrm{axi}}(\mathbb{R}^{3})$ are expressed as $b=\nabla \times (\phi\nabla \theta)+G\nabla \theta$ with unique potentials $\phi$ and $G$ and admit the unique vector potential
 
\begin{align*}
a=\nabla \times (\eta\nabla \theta)+\phi\nabla \theta, 
\end{align*}\\
where $\eta$ is a solution to the Dirichlet problem $-L\eta=G$ in $\mathbb{R}^{2}_{+}$ and $\eta=0$ on $\partial \mathbb{R}^{2}_{+}$. The constant $B_{\infty}=-We_z$ can be expressed as $B_{\infty}=-\nabla \times (\phi_{\infty}\nabla \theta)$ with the axisymmetric potential $\phi_{\infty}=Wr^{2}/2+\gamma$ which is unique up to constant and we have Clebsch representation to 

\begin{align*}
B=b+B_{\infty}=\nabla \times (\Phi\nabla \theta)+G\nabla \theta,\quad \Phi=\phi-\phi_{\infty}.
\end{align*}\\
The vector field

\begin{align*}
A=\nabla \times (\eta\nabla \theta)+\Phi\nabla \theta
\end{align*}\\
is (one of) the vector potentials of $B$ and the pseudo-scalar $A\cdot B$ with this vector potential is expressed as 

\begin{align*}
A\cdot B=2\Phi \frac{G}{r^{2}}+\nabla \cdot (\Phi\nabla \theta \times A). 
\end{align*}\\
The magnetic helicity of the axisymmetric solenoidal vector field $B$ can be computed formally by integrating this pseudo-scalar in $\mathbb{R}^{3}$. The constant shift of the potential $\phi_{\infty}$ induces the change $A\longmapsto A+C\nabla \theta$ and affects magnetic helicity.

The generalized magnetic helicity is based on symmetry conservation (Casimir invariant). We show the identity

\begin{align*}
A_t\cdot Bg(\Phi)=\partial_t \left(g(\Phi)\frac{G}{r^{2}} \right)+\nabla \cdot (g(\Phi)\nabla \theta\times A ),
\end{align*}\\
for axisymmetric solenoidal vector field $B$ and any function $g(s)$ and deduce the generalized magnetic helicity conservation equality (1.17) from the vector potential equation  

\begin{align*}
A_{t}+B\times u+\nabla Q=-\mu \nabla \times B,
\end{align*}\\
with the choice $g(s)=2s_{+}$ (The choice $g(s)=2s$ corresponds to the magnetic helicity equality.)

\subsubsection{Orbital stability}

The generalized magnetic helicity is affected by the gauge $\gamma\geq 0$, and so are the minimum $I_{h,W,\gamma}$ and the set of minimizers $S_{h,W,\gamma}$ in the minimization (1.13). The current field $\nabla \times U$ of the minimizer $U=\nabla \times ((\phi-\phi_{\infty})\nabla \theta)+\kappa (\phi-\phi_{\infty})_{+}\nabla \theta$ supported in $\{ \phi>\phi_{\infty} \}$ does not intersect the $z$-axis for the positive gauge $\gamma>0$. For the particular gauge $\gamma=0$, the minimizer is a translation of the explicit solution (1.6) for the strength $\kappa^{2}$ by the uniqueness of the Grad--Shafranov equation (1.16) (Theorem 8.3) via the moving plane method \cite{Fra92}, \cite{Fra00}. For $h=h_{C}(W,\lambda)$ in (1.11), $\kappa^{2}=\lambda$ by the explicit form of the constant $h_C$ (Proposition 8.4) and the set $S_{h_C,W,0}$ agrees with translations of the explicit solution $U_C$ with the strength $\lambda$, i.e., 

\begin{align*}
S_{h_C,W,0}=\{U_C(\cdot +ze_z)-B_{\infty} |\ z\in \mathbb{R}\ \}.
\end{align*}\\
In weak ideal limits of axisymmetric Leray--Hopf solutions, the two points (i) and (ii) imply the stability of a set of minimizers $S_{h,W,\gamma}$ to $I_{h,W,\gamma}$ in (1.13) for $h\in \mathbb{R}$, $W>0$, and $\gamma\geq 0$ (Theorem 8.2). Theorem 1.4 is derived from the case $h=h_C$ and $\gamma=0$. 

The scaling property of the minimum (1.13) is $I_{h,W,\gamma}=(W/2)^{-2}I_{\tilde{h},2,\tilde{\gamma}}$ for $\tilde{h}=(W/2)^{-2}h$ and $\tilde{\gamma}=(W/2)^{-1}\gamma$. We will fix $W=2$ from Section 3 to Section 7 and deduce results for $W>0$ from those for $W=2$ in Section 8.

\subsection{Organization of the study} 

The following is how this paper is structured. Section 2 demonstrates the Taylor state stability (Theorem 1.1). In Section 3, we prove Clebsch representation and well-definedness of generalized magnetic helicity (1.10) and generalized magnetic mean-square potential (1.14). Section 4 develops the variational problem (1.13) (for $W=2$) and establishes the strict subadditivity (1.15). Section 5 demonstrates the compactness of minimizing sequences of (1.13) by the concentration--compactness principle. Section 6 shows the equalities (1.17) and (1.18) for axisymmetric Leray--Hopf solutions. Section 7 proves the conservation of generalized magnetic helicity at weak ideal limits. Section 8 demonstrates the orbital stability of a set of minimizers to (1.13) and deduces Theorem 1.4 from a uniqueness theorem. Appendix A is devoted to the existence of axisymmetric Leray--Hopf solutions to (1.9). Appendix B investigates Taylor states in a ball (Remark 1.3.)

\subsection{Acknowledgements} 
The author thanks Professors Daniel Faraco and Sauli Lindberg for sending him the review \cite{FLS22}. He is also grateful to Professors Alexey Cheskidov, Kyudong Choi, Kenji Nakanishi, Masahito Ohta, and Zensho Yoshida for their valuable comments. This work was made possible in part by the JSPS through the Grant-in-aid for Young Scientist 20K14347 and by  MEXT Promotion of Distinctive Joint Research Center Program JPMXP0619217849.\\


\section{Linear force-free fields}
We demonstrate Taylor state stability (Theorem 1.1). We show that minimizers of Woltjer's principle (1.5) (Taylor states) are eigenfunctions of the rotation operator associated with the principal eigenvalues by Rayleigh's formulas. We then prove Taylor state stability by a contradiction argument.

\subsection{Magnetic helicity}
Let $\Omega\subset \mathbb{R}^{3}$ be a bounded and simply connected domain with a $C^{1,1}$-boundary $\partial\Omega$ consisting of connected components $\Gamma_{1},\dots, \Gamma_{I}$. We consider the direct sum decomposition

\begin{equation}
\begin{aligned}
L^{p}(\Omega)&=L^{p}_{\sigma}(\Omega)\oplus G^{p} (\Omega),\quad 1<p<\infty,\\ 
L^{p}_{\sigma}(\Omega)&=\{u\in L^{p}(\Omega)\ |\ \nabla \cdot u=0\ \textrm{in}\ \Omega,\ u\cdot n=0\ \textrm{on}\ \partial\Omega\},\\
G^{p}(\Omega)&=\{\nabla q\in L^{p}(\Omega)\ |\ q\in L^{1}_{\textrm{loc}}(\Omega) \},
\end{aligned}
\end{equation}\\
and the associated projection operator $\mathbb{P}:  L^{p}(\Omega)\to L^{p}_{\sigma}(\Omega)$. We denote by $C_{c,\sigma}^{\infty}(\Omega)$ the space of all smooth solenoidal vector fields with compact support in $\Omega$. The $L^{p}$-closure of $C_{c,\sigma}^{\infty}(\Omega)$ is $L^{p}_{\sigma}(\Omega)$ and the $W^{1,p}$-closure of $C_{c,\sigma}^{\infty}(\Omega)$ is $L^{p}_{\sigma}\cap W^{1,p}_{0}(\Omega)$. The same decomposition and density properties hold also for $\Omega=\mathbb{R}^{3}$ \cite[Theorem III.1.2, 2.3, 5.1]{Gal}. 

According to \cite{ABDG} and \cite{FL20}, we define spaces of vector potentials. Let $H(\D,\Omega)=\{u\in L^{2}(\Omega)\ |\ \nabla \cdot u\in L^{2}(\Omega) \}$ and $H(\textrm{curl},\Omega)=\{u\in L^{2}(\Omega)\ |\ \nabla\times u\in L^{2}(\Omega) \}$. By the trace operator $T: H^{1}(\Omega)\ni \varphi\longmapsto T\varphi\in H^{1/2}(\partial\Omega)$, the  tangential trace $u\times n$ for $u\in H(\textrm{curl},\Omega)$ and the normal trace $u\cdot n$ for $u\in H(\textrm{div},\Omega)$ are defined as elements in $H^{-1/2}(\partial\Omega)=H^{1/2}(\partial\Omega)^{*}$ in the sense that  

\begin{equation}
\begin{aligned}
<u\times n, \xi>_{\partial\Omega}&=\int_{\Omega}\nabla \times u\cdot  \xi\dd x-\int_{\Omega}u\cdot \nabla \times \xi\dd x,\\
<u\cdot n, \varphi>_{\partial\Omega}&=\int_{\Omega}u\cdot \nabla \varphi\dd x+\int_{\Omega}\nabla \cdot  u \varphi\dd x,
\end{aligned}
\end{equation}\\
for $\xi, \varphi\in H^{1}(\Omega)$. Here, $<\cdot,\cdot>_{\partial\Omega}$ denotes the duality pairing between $H^{-1/2}(\partial\Omega)$ and $H^{1/2}(\partial\Omega)$. The traces can be defined also in the $L^{p}$ setting \cite{FL20}. Let $X(\Omega)=H(\D,\Omega)\cap H(\textrm{curl},\Omega)$. Let $X_{N}(\Omega)=\{u\in X(\Omega)\ |\ u\times n=0\ \textrm{on}\ \partial\Omega \}$ and $X_{T}(\Omega)=\{u\in X(\Omega)\ |\ u\cdot n=0\ \textrm{on}\ \partial\Omega \}$. When $\Omega$ is $C^{1,1}$ or convex, $X_N(\Omega)$ and $X_T(\Omega)$ continuously embed into $H^{1}(\Omega)$ \cite[Theorems 2.9, 2.12, 2.17]{ABDG}, i.e. 

\begin{equation}
X_N(\Omega), X_T(\Omega)\subset H^{1}(\Omega).
\end{equation}\\
When $\Omega$ is Lipschitz, $X_N(\Omega)$ and $X_T(\Omega)$ continuously embed into $H^{s}(\Omega)$ for some $s>1/2$ and compactly embed into $L^{2}(\Omega)$ \cite[Proposition 3.7, Theorem 2.8]{ABDG}. We set $K_{N}(\Omega)=\{u\in X_N(\Omega)\ |\ \nabla \times  u=0,\ \nabla \cdot u=0\ \textrm{in}\ \Omega \}$ and $K_{T}(\Omega)=\{u\in X_T(\Omega)\ |\ \nabla \times u=0,\ \nabla \cdot u=0\ \textrm{in}\ \Omega \}$. For a simply connected domain $\Omega$ with connected $\Gamma_{1},\dots, \Gamma_{I}$ of $\partial\Omega$, $\textrm{dim}\ K_{T}(\Omega)=0$ and $\textrm{dim}\ K_{N}(\Omega)=I$ \cite[Propositions 3.14, 3.18]{ABDG}.

For $B\in L^{2}_{\sigma}(\Omega)$, there exists a unique $A\in X_N(\Omega)$ such that $\nabla \cdot A=0$ in $\Omega$ and $\int_{\Gamma_i}A\cdot n\dd H=0$ for $1\leq i\leq I$ \cite[Theorem 3.17]{ABDG}. We denote such $A$ by $A=\textrm{curl}^{-1}B$ (Coulomb gauge). Magnetic helicity in a simply connected domain is gauge invariant in the sense that 

\begin{align*}
\int_{\Omega}\tilde{A}\cdot B\dd x=\int_{\Omega}\textrm{curl}^{-1}B\cdot B\dd x,
\end{align*}\\
for any vector potential $\tilde{A}$ such that $\nabla \times \tilde{A}=B$. We use the integrand $\textrm{curl}^{-1}B \cdot B$ for the magnetic helicity.

\subsection{Rayleigh's formulas}

We recall the fact that the rotation operator $S=\nabla \times$ is a self-adjoint operator on $L^{2}_{\sigma}(\Omega)$ \cite[Theorem 1.1]{YG90}. Let $D(S)=\{u\in L^{2}_{\sigma}(\Omega)\ |\ S u\in L^{2}_{\sigma}(\Omega) \}$ and $Su=\nabla \times u$ for $u\in D(S)$. By the continuous embedding (2.3), $u\times n\in H^{1/2}(\partial\Omega)$ for $u\in D(S)\subset H^{1}(\Omega)$. The vector field $u$ is irrotational on $\partial\Omega$ for $u\in D(S)$, i.e., $d\alpha=0$ for the dual 1-form $\alpha$ of $u$ on $\partial\Omega$ and the exterior derivative $d$. Thus there exists a 0-form $\phi\in H^{3/2}(\partial\Omega)$ such that $\alpha=d \phi$.

Let $D(S^{*})$ be a set of all $v\in L^{2}_{\sigma}(\Omega)$ such that there exists some $w\in L^{2}_{\sigma}(\Omega)$ such that $(Su,v)=(u,w)$ for all $u\in L^{2}_{\sigma}(\Omega)$, where $(\cdot,\cdot)$ is the inner product on $L^{2}_{\sigma}(\Omega)$. We set the adjoint operator by $S^{*}v=w$ for $v\in D(S^{*})$. For $u,v\in D(S)$, we denote by $\alpha$ and $\beta$, dual 1-forms of $u$ and $v$ on $\partial\Omega$. Since $\alpha=d\phi$ and $d\beta=0$, 

\begin{align*}
d(\phi\wedge \beta)=d\phi\wedge \beta-\phi\wedge d\beta=d\phi\wedge \beta=\alpha\wedge \beta.
\end{align*}\\
By Stokes's theorem,

\begin{align*}
<n\times u,v>_{\partial\Omega}=\int_{\partial\Omega}u\times v\cdot  n\dd H
=\int_{\partial\Omega}\alpha\wedge \beta=\int_{\partial\Omega}d(\phi \wedge \beta)=0.
\end{align*}\\
This implies that $(Su,v)=(u,Sv)$, $u,v\in D(S)$ and $D(S)\subset D(S^{*})$ by $(2.2)_1$. Thus, $S$ is a symmetric operator. For $\tilde{u}\in C^{\infty}_{c}(\Omega)$, $u=\mathbb{P}\tilde{u}\in D(S)$ since $\nabla \times u=\nabla \times \tilde{u}$. By $(\nabla \times \tilde{u},v)=( \tilde{u},S^{*}v)$, $\nabla \times v=S^{*}v\in L^{2}_{\sigma}(\Omega)$ and $D(S^{*})\subset D(S)$. Thus, $S$ is a self-adjoint operator. The paper \cite{YG90} assumes that $\partial\Omega$ is smooth. The same result holds for domains with  $C^{1,1}$-boundaries by the continuous embedding (2.3).

The operator $S: L^{2}_{\sigma}(\Omega)\supset D(S)\longrightarrow L^{2}_{\sigma}(\Omega)$ is injective since $K_T(\Omega)=\emptyset$. The operator $S$ is surjective since for $B\in L^{2}_{\sigma}(\Omega)$, there exists a unique $\hat{A}\in X_{T}(\Omega)$ such that $\nabla \times \hat{A}=B$, $\nabla \cdot \hat{A}=0$ in $\Omega$ \cite[Theorem 3.12]{ABDG}, i.e.,  $\hat{A}\in D(S)$. The inverse operator $S^{-1}: L^{2}_{\sigma}(\Omega)\ni B\longmapsto \hat{A}\in  D(S)$ is compact by (2.3). Thus, eigenvalues are countable real numbers. Each eigenfunction space is finite-dimensional by the Fredholm alternative, and there exist eigenfunctions $\{\bold{B}_j\}$ making a complete orthonormal basis on $L^{2}_{\sigma}(\Omega)$. They provide eigenvalues $\{\lambda_j\}$ and eigenfunctions $\{\bold{B}_j\}$ of the operator $S$.

We label $\{\bold{B}_j\}$ according to multiplicity of $\{\lambda_j\}$ so that each eigenvalue $\lambda_j$ corresponds to one eigenfunction $\bold{B}_j$. We also denote $\{\lambda_j\}$ by $\cdots<f_{2}^{-}\leq f_{1}^{-}<0<f_{1}^{+}\leq f_{2}^{+}<\cdots$. The eigenfunction $\bold{B}_j$ and $\bold{A}_j=\textrm{curl}^{-1}B_j$ satisfy

\begin{align*}
\int_{\Omega}\bold{A}_j\cdot \bold{B}_k\dd x=\frac{1}{\lambda_k}\int_{\Omega}\bold{A}_j\cdot \nabla \times \bold{B}_k\dd x=\frac{1}{\lambda_k}\int_{\Omega}\bold{B}_j\cdot  \bold{B}_k\dd x=\frac{1}{\lambda_k}\delta_{j,k}.
\end{align*}\\
The expansion of $B\in L^{2}_{\sigma}(\Omega)$ and $A=\textrm{curl}^{-1}B$ are 

\begin{align*}
B=\sum_{j}c_j\bold{B}_j,\quad c_j=\int_{\Omega}B\cdot \bold{B}_{j}\dd x,\quad A=\sum_{j}c_j\bold{A}_{j}.
\end{align*}\\
Magnetic helicity of $B$ is expressed as 

\begin{align*}
\int_{\Omega}A\cdot B\dd x=\sum_{j}\frac{1}{\lambda_j}c_j^{2}=\sum_{j}\frac{1}{f_j^{+}}c_j^{2}+\sum_{j}\frac{1}{f_j^{-}}c_j^{2}.
\end{align*}\\
This helicity expansion yields the best constants for Arnold's inequalities \cite[p.122]{AK98}
 
\begin{equation} 
\begin{aligned}
\int_{\Omega}A\cdot B\dd x&\leq \frac{1}{f_1^{+}}\int_{\Omega}|B|^{2}\dd x,\\
-\int_{\Omega}A\cdot B\dd x &\leq \frac{1}{-f_1^{-}}\int_{\Omega}|B|^{2}\dd x,
\end{aligned}
\end{equation}
for $B\in L^{2}_{\sigma}(\Omega)$ and $A=\textrm{curl}^{-1}B$. The equalities hold for eigenfunctions associated with $f_1^{+}$ and $f_1^{-}$. Thus, normalization implies Rayleigh's formulas

\begin{equation} 
\begin{aligned}
f_{1}^{+}&=\inf \left\{\int_{\Omega}|B|^{2}\dd x\  \middle|\ B\in L^{2}_{\sigma}(\Omega),\  \int_{\Omega}\textrm{curl}^{-1}B\cdot B\dd x=1 \right\}, \\
-f_{1}^{-}&=\inf \left\{\int_{\Omega}|B|^{2}\dd x\  \middle|\ B\in L^{2}_{\sigma}(\Omega),\  \int_{\Omega}\textrm{curl}^{-1}B\cdot B\dd x=-1 \right\}.
\end{aligned}
\end{equation}\\

\subsection{Woltjer's principle}

According to \cite[p.1246]{Laurence91}, we consider Woltjer's principle with the constant $h\in \mathbb{R}$,  

\begin{align*}
{\mathcal{I}}_h=\inf\left\{\frac{1}{2}\int_{\Omega}|B|^{2}\dd x\  \middle|\ B\in L^{2}_{\sigma}(\Omega),\  \int_{\Omega}\textrm{curl}^{-1}B\cdot B\dd x=h \right\}.
\end{align*}\\
We denote by ${\mathcal{S}}_h$ the set of minimizers to $\mathcal{I}_h$. By the scaling $B=|h|^{1/2}\tilde{B}$, 

\begin{equation}
\mathcal{I}_{h}=
\begin{cases}
\ h \mathcal{I}_{1}   &h>0, \\
\ -h \mathcal{I}_{-1} &h<0,\\
\ 0  &h=0,
\end{cases}
\qquad 
S_{h}=
\begin{cases}
\ h^{1/2} S_{1}   &h>0, \\
\ (-h)^{1/2} S_{-1} &h<0,\\
\ \emptyset  &h=0.
\end{cases}
\end{equation}\\
The minimum $\mathcal{I}_{h}$ is Lipschitz continuous and additive for $h\in \mathbb{R}$. Rayleigh's formulas (2.5) are expressed in terms of the minimum as 

\begin{align}
f_{1}^{+}=2{\mathcal{I}}_1,\quad -f_{-}^{+}=2{\mathcal{I}}_{-1}.  
\end{align}\\
We will identify the set of minimizers ${\mathcal{S}}_h$ with the set of a finite number of eigenfunctions associated with $f_1^{+}$ for $h>0$ (resp. $f_1^{-}$ for $h<0$) by using (2.7) and a Lagrange multiplier \cite[8.4.1]{E}, cf. \cite[p.1247]{Laurence91}.

\begin{prop}
Suppose that $U\in L^{2}_{\sigma}(\Omega)$ is an eigenfunction of the rotation operator associated with the eigenvalue $f_1^{+}$ (resp. $f^{-}_{1}$) with helicity $1$ (resp. $-1$). Then, $\frac{1}{2}\int_{\Omega}|U|^{2}\dd x=\mathcal{I}_1$ (resp. $\frac{1}{2}\int_{\Omega}|U|^{2}\dd x=\mathcal{I}_{-1}$). 
\end{prop}

\begin{proof}
By the formula (2.7),

\begin{align*}
\frac{1}{2}\int_{\Omega}|U|^{2}\dd x=\frac{1}{2}\int_{\Omega}U\cdot \nabla \times( \textrm{curl}^{-1}U)\dd x
&=\frac{1}{2}\int_{\Omega}\nabla \times U\cdot \textrm{curl}^{-1}U\dd x \\
&=\frac{f_{1}^{+}}{2}\int_{\Omega} U\cdot  \textrm{curl}^{-1}U\dd x
={\mathcal{I}}_1.
\end{align*}
\end{proof}

\begin{prop}
Suppose that $U\in L^{2}_{\sigma}(\Omega)$ satisfies $\frac{1}{2}\int_{\Omega}|U|^{2}\dd x={\mathcal{I}}_1$with helicity 1 (resp. $\frac{1}{2}\int_{\Omega}|U|^{2}\dd x={\mathcal{I}}_{-1}$ with helicity $-1$). Then, $U$ is an eigenfunction of the rotation operator for the eigenvalue $f_{1}^{+}$ (resp. $f_1^{-}$). 
\end{prop}

\begin{proof}
By $\int_{\Omega}\textrm{curl}^{-1}U\cdot U\dd x=1$, there exists $U_0\in L^{2}_{\sigma}(\Omega)$ such that $\int_{\Omega}\textrm{curl}^{-1}U \cdot U_0\dd x\neq 0$. We take an arbitrary $\tilde{U}\in L^{2}_{\sigma}(\Omega)$ and set 

\begin{align*}
j(\tau,s)=\int_{\Omega}\textrm{curl}^{-1}(U+\tau \tilde{U}+s U_0)\cdot (U+\tau \tilde{U}+s U_0)\dd x. 
\end{align*}\\
The function $j(\tau,s)$ satisfies $j(0,0)=1$ and 

\begin{align*}
\frac{\partial j}{\partial s}(0,0)=2\int_{\Omega}\textrm{curl}^{-1}U\cdot U_0\dd x\neq 0.
\end{align*}\\
By the implicit function theorem, there exists a $C^{1}$-function $s=s(\tau)$ such that $j(\tau,s(\tau))=1$ for sufficiently small $|\tau|$. By differentiating this by $\tau$, $\dot{s}(0)=-\partial_{\tau} j(0,0)/\partial_{s} j(0,0)$. Since $U$ is a minimizer, 

\begin{align*}
0=\frac{\dd}{\dd \tau}\int_{\Omega}|U+\tau \tilde{U}+s(\tau)U_0|^{2}\dd x\Bigg|_{\tau=0}
=2\int_{\Omega}U\cdot( \tilde{U}+\dot{s}(0) U_0) \dd x.
\end{align*}\\
By substituting $\dot{s}(0)$ into the above, 

\begin{align*}
\int_{\Omega}(U-f \textrm{curl}^{-1}U)\cdot \tilde{U}\dd x=0,\quad f=\frac{\int_{\Omega}U_0\cdot U\dd x}{\int_{\Omega}\textrm{curl}^{-1}U_0\cdot U\dd x}.
\end{align*}\\
Since $\tilde{U}\in L^{2}_{\sigma}(\Omega)$ is arbitrary, by (2.1) there exists some $q$ such that $U- f \textrm{curl}^{-1}U=\nabla q$. Thus $\nabla \times U=f U$. By Rayleigh's formulas (2.7), 

\begin{align*}
f_{1}^{+}=2\mathcal{I}_1=\int_{\Omega}|U|^{2}\dd x=\int_{\Omega}U\cdot \nabla \times (\textrm{curl}^{-1}U)\dd x=f\int_{\Omega}U\cdot  \textrm{curl}^{-1}U\dd x= f.
\end{align*}\\
Hence $U$ is an eigenfunction for the eigenvalue $f_1^{+}$. 
\end{proof}

\begin{lem}
Let $\{U_{j}^{+}\}_{j=1}^{N^{+}}$ (resp. $\{U_{j}^{-}\}_{j=1}^{N^{-}}$) be eigenfunctions of the rotation operator associated with the eigenvalues $f_1^{+}$ with helicity $h>0$ (resp. $f_1^{-}$ with helicity $h<0$). Then,

\begin{align}
{\mathcal{S}}_h=
\begin{cases}
\ \{U_j^{+}\}_{j=1}^{N^{+}} & h>0,\\
\ \{U_j^{-}\}_{j=1}^{N^{-}} &h<0,\\
\ \emptyset &h=0.
\end{cases}
\end{align}
\end{lem}

\begin{proof}
By Propositions 2.1 and 2.2, 

\begin{align*}
{\mathcal{S}}_1= \left\{U\in L^{2}_{\sigma}(\Omega)\ \middle|\ \nabla \times U=f_1^{+}U,\ \ \int_{\Omega}\textrm{curl}^{-1}U\cdot U\dd x=1 \right\}. 
\end{align*}\\
The multiplicity $N^{+}$ of $f_{1}^{+}$ is finite, and the right-hand side consists of finite elements. By scaling (2.6), $(2.8)_1$ follows. The same argument applies to the case $h<0$.
\end{proof}

\subsection{Minimizing sequences}

To prove the stability of the set ${\mathcal{S}}_h$, we use the compactness of minimizing sequences.

\begin{lem}
Let $h\in \mathbb{R}$. For a sequence $\{B_n\}\subset L^{2}_{\sigma}(\Omega)$ satisfying $\frac{1}{2}\int_{\Omega}|B_n|^{2}\dd x\to {\mathcal{I}}_h$ and $\int_{\Omega}\textrm{curl}^{-1}B_n\cdot  B_n\dd x\to h$, there exist $\{n_k\}$ and $B\in {\mathcal{S}}_{h}$ such that $B_{n_k}\to B$ in $L^{2}(\Omega)$. 
\end{lem}

\begin{proof}
By (2.3) and Rellich--Kondrakov theorem, there exists a subsequence (still denoted by $\{B_n\}$) and $B\in L^{2}_{\sigma}(\Omega)$ such that 

\begin{align*}
B_n&\rightharpoonup B\quad \textrm{in}\ L^{2}(\Omega),\\
A_n=\textrm{curl}^{-1}B_n&\to A=\textrm{curl}^{-1}B \quad \textrm{in}\ L^{2}(\Omega).
\end{align*}\\
They imply that

\begin{align*}
\left|\int_{\Omega} (A_n\cdot B_n-A\cdot B)\dd x\right|
\leq ||A_n-A||_{L^{2}} \sup_n||B_n||_{L^{2}}+\left|\int_{\Omega} (A\cdot B_n-A\cdot B)\dd x\right|\to 0.
\end{align*}\\
By $\int_{\Omega}A\cdot B\dd x=h$ and 

\begin{align*}
{\mathcal{I}}_h\leq \int_{\Omega}|B|^{2}\dd x\leq \liminf_{n\to\infty}\int_{\Omega}|B_n|^{2}\dd x={\mathcal{I}}_h,
\end{align*}\\
we conclude that $B\in {\mathcal{S}}_{h}$ and $B_n\to B$ in $L^{2}(\Omega)$. 
\end{proof}

\vspace{15pt}

For application to stability, we state Lemma 2.4 in terms of total energy. 

\vspace{15pt}

\begin{thm}
Let $h\in \mathbb{R}$. For a sequence $\{(u_n,B_n)\}\subset L^{2}_{\sigma}(\Omega)$ satisfying $\frac{1}{2}\int_{\Omega}(|u_n|^{2}+ |B_n|^{2})\dd x\to {\mathcal{I}}_h$ and $\int_{\Omega}\textrm{curl}^{-1}B_n\cdot  B_n\dd x\to h$. There exist $\{n_k\}\subset \mathbb{N}$ and $B \in {\mathcal{S}}_h$ such that $(u_{n_k}, B_{n_k})\to (0,B)$ in $L^{2}(\Omega)$.
\end{thm}

\begin{proof}
For $h_n=\int_{\Omega}\textrm{curl}^{-1}B_n\cdot  B_n\dd x$, 

\begin{align*}
{\mathcal{I}}_{h_n} \leq \frac{1}{2}\int_{\Omega}|B_n|^{2}\dd x\leq\frac{1}{2} \int_{\Omega}\left(|u_n|^{2}+ |B_n|^{2}\right)\dd x. 
\end{align*}\\
Since ${\mathcal{I}}_{h}$ is continuous for $h\in \mathbb{R}$, letting $n\to\infty$ implies that  

\begin{align*}
\int_{\Omega}|B_n|^{2}\dd x\to I_h,\quad \int_{\Omega}|u_n|^{2}\dd x\to 0.
\end{align*}\\
We apply Lemma 2.4 and conclude.
\end{proof}

\begin{rem}
The assertion of Lemma 2.4 holds even for domains with Lipschitz boundaries since the embedding from $X_{N}(\Omega)$ into $L^{2}(\Omega)$ is compact. (The $C^{1}$-regularity is assumed in \cite[p.1247]{Laurence91}). The paper \cite{Laurence91} also studies minimization with an inhomogeneous boundary condition by using relative helicity \cite{BF84}, \cite{JC84}, and \cite{FA85}.     
\end{rem}

\subsection{Leray--Hopf solutions}

To prove the stability of force-free fields in weak ideal limits (Theorem 1.1), we recall definitions of Leray--Hopf solutions to (1.1)--(1.2) \cite[p.716]{FL20}, \cite[p.60]{GLL} and of their weak ideal limits \cite[p.709]{FL20}. The existence of Leray--Hopf solutions has been demonstrated for simply connected domains in \cite{DL72}, \cite[p.647]{ST83}, \cite[p.60]{GLL}, and \cite[Theorem 2.1]{FL20} and for multiply connected domains in \cite[Appendix A]{FL20}.\\

\begin{defn}[Leray--Hopf solutions]
Let $u_0,B_0\in L^{2}_{\sigma}(\Omega)$. Let 

\begin{align*}
u\in C_{w}([0,T]; L^{2}_{\sigma}(\Omega))\cap L^{2}(0,T; H^{1}_{0}(\Omega) ),\quad B\in C_{w}([0,T]; L^{2}_{\sigma}(\Omega))\cap L^{2}(0,T; H^{1}(\Omega) ).
\end{align*}\\
Suppose that $u_t\in L^{1}(0,T; (L^{2}_{\sigma}\cap H^{1}_{0})(\Omega)^{*})$, $B_t\in L^{1}(0,T; (L^{2}_{\sigma}\cap H^{1})(\Omega)^{*})$ and 

\begin{align*}
&<u_t,\xi>+\int_{\Omega}(u\cdot \nabla u-B\cdot \nabla B)\cdot \xi\dd x+\nu \int_{\Omega}\nabla u:\nabla \xi\dd x=0,\\
&<B_t,\zeta>+\int_{\Omega}(B\times u)\cdot \nabla \times \zeta\dd x +\mu \int_{\Omega}\nabla \times B\cdot \nabla \times \zeta \dd x=0,
\end{align*}\\
for a.e. $t\in [0,T]$ and every $\xi\in L^{2}_{\sigma}\cap H^{1}_{0}(\Omega)$, $\zeta\in L^{2}_{\sigma}\cap H^{1}(\Omega)$. Suppose furthermore that $(u(\cdot ,0),B(\cdot ,0))=(u_0, B_0)$ and 

\begin{align*}
\frac{1}{2}\int_{\Omega}\left(|u|^{2}+|B|^{2}\right)\dd x
+\int_0^{t}\int_{\Omega}\left(\nu |\nabla u|^{2}+\mu |\nabla B|^{2}\right)\dd x\dd s
\leq \frac{1}{2}\int_{\Omega}\left(|u_0|^{2}+|B_0|^{2}\right)\dd x,
\end{align*}\\
for all $t\in [0,T]$. Then, we call $(u,B)$ Leray--Hopf solutions to (1.1)--(1.2).\\
\end{defn}

\begin{thm}
There exists a Leray--Hopf solution to (1.1)--(1.2). \\
\end{thm}

\begin{defn}[Weak ideal limits]
Let $(u_j,B_j)$ be a Leray--Hopf solution to (1.1)--(1.2) for $\nu_j,\mu_j>0$ and $(u_{0,j}, B_{0,j})$ such that $(u_{0,j},B_{0,j}  )\rightharpoonup (u_0, B_0)$ in $L^{2}(\Omega)$ as $(\nu_j,\mu_j)\to (0,0)$. Assume that 

\begin{align*}
(u_j, B_j)\overset{\ast}{\rightharpoonup} (u,B)\quad \textrm{in}\ L^{\infty}(0,T; L^{2}(\Omega) ).
\end{align*}\\
Then, we call $(u,B)$ a weak ideal limit of $(u_j,B_j)$. If instead $\nu_j=\nu>0$ for every $j$ and $\mu_j\to 0$, we call $(u,B)$ a weak non-resistive limit of $(u_j,B_j)$.\\
\end{defn}

By the equation $(1.1)_2$, the vector potential $A=\textrm{curl}^{-1}B$ and  some potential $Q$ satisfy

\begin{align}
A_t+B\times u+\nabla Q=-\mu\nabla \times B.
\end{align}\\
For smooth solutions, the equality of magnetic helicity (1.12) follows by multiplying $B$ by (2.9) and $A$ by $(1.1)_2$, respectively, and integration by parts. The equality (1.12) for Leray--Hopf solutions for all $t\in [0,T]$ is proved in \cite[Lemma 4.5]{FL20} by an approximation argument for the time variable.

\subsection{Weak ideal limits}

We state magnetic helicity conservation at weak ideal limits \cite[Theorem 1.2]{FL20} in terms of the vector potential $\textrm{curl}^{-1}B$. 

\begin{thm}
Suppose that $(u,B)$ is a weak ideal limit of Leray--Hopf solutions to (1.1)--(1.2). Then,  

\begin{align}
\int_{\Omega}\textrm{curl}^{-1}B\cdot B\dd x=\int_{\Omega}\textrm{curl}^{-1}B_0\cdot B_0\dd x,
\end{align}
for a.e. $t\in [0,T]$.
\end{thm}

\begin{thm}
There exists a weak ideal limit $(u,B)$ of Leray--Hopf solutions to (1.1)--(1.2) for $(u_0,B_0)\in L^{2}_{\sigma}(\Omega)$ satisfying (2.10) and 

\begin{align*}
\int_{\Omega}\left(|u|^{2}+|B|^{2}\right) \dd x
\leq \int_{\Omega}\left(|u_0|^{2}+|B_0|^{2}\right)\dd x,    
\end{align*}\\ 
for a.e. $t\in [0,T]$.
\end{thm}

\begin{proof}
The assertion follows Theorems 2.8, 2.10, and lower semi-continuity of the norm for the weak-star convergence in $L^{\infty}(0,T; L^{2}(\Omega))$. \\
\end{proof}

\begin{rem}
When $\Omega$ is multiply connected, helicity conservation at weak ideal limits is demonstrated in \cite{FL22} by using a gauge-invariant definition of magnetic helicity \cite{MV19}, cf. \cite{FL20}. (The boundary regularity is reduced from $C^{1,1}$ to  Lipschitz in \cite{MV19}.)
\end{rem}

The proof of Theorem 2.10 is based on the following approximation lemma and the Aubin--Lions lemma \cite[Lemmas 2.10, 2.12]{FL20}, \cite[Proposition 1.2.32]{TNVW}. We state these lemmas for the later usage in Sections 6 and 7.

Let $X$ be a Banach space. Let $0<\delta <T/2$. Let $g\in C^{\infty}_{c}(\mathbb{R})$ satisfy $\textrm{spt}\ g\in (-\delta,\delta)$. For $f\in L^{1}(0,T; X)$, we denote the convolution in $(0,T)$ by $f*g=\int_{0}^{T}f(s)g(t-s)\dd s\in C^{\infty}((\delta,T-\delta); X )$. For $\chi\in C^{\infty}_{c}(\mathbb{R})$ satisfying $\textrm{spt}\ \chi\in (-1,1)$ and $\int_{\mathbb{R}}\chi\dd t=1$, we set $\chi^{\varepsilon}(t)=\varepsilon^{-1}\chi(\varepsilon^{-1}t)$ for $\varepsilon>0$ and $f^{\varepsilon}=f*\chi^{\varepsilon}$. 

\begin{lem}
Let $0<\delta<T/2$ and $1\leq p<\infty$. For $f\in L^{p}(0,T; X)$, $f^{\varepsilon}\to f$ in $L^{p}(\delta,T-\delta;X)$ as $\varepsilon\to 0$.
\end{lem}

\begin{lem}[Aubin--Lions lemma]
Let $X$, $Y$, and $Z$ be reflexive Banach spaces such that $X$ embeds compactly into $Y$ and $Y$ embeds into $Z$. Let $1<p<\infty$ and $1\leq q\leq \infty$. The space $\{A\in L^{p}(0,T; X)\ |\ A_t\in L^{q}(0,T; Z)\  \}$ embeds compactly into $L^{p}(0,T; Y)$.
\end{lem}

A vital part of the proof of Theorem 2.10 for a simply connected domain is the strong convergence of the vector potential 

\begin{align*}
A_j=\textrm{curl}^{-1}B_j\to A=\textrm{curl}^{-1}B\quad \textrm{in}\ L^{2}_{\textrm{loc}}(0,T; L^{2}(\Omega)), 
\end{align*}\\
for Leray--Hopf solutions $(u_j,B_j)$ and the weak ideal limit $(u,B)$ by application of the Aubin--Lions lemma. This convergence implies (2.10) by letting $j\to\infty$ to the equality (1.12). To apply the Aubin--Lions lemma, the paper \cite{FL20} showed a uniform bound 

\begin{align*}
\sup_{j} ||\partial_tA_{j}^{\varepsilon_j} ||_{L^{2}(\delta,T-\delta; (L^{2}_{\sigma}\cap W^{1,4}_{0})(\Omega)^{*} ) }<\infty.
\end{align*}\\
This uniform bound is demonstrated by the equation (2.9) and the Sobolev embedding $W^{1,4}_0(\Omega)\subset L^{\infty}(\Omega)$. The vector potential $A_j$ is approximated by $A^{\varepsilon_j}_{j}$ with $\varepsilon_j>0$ to deduce (2.9) from the definition of Leray--Hopf solutions (Definition 2.7). 

\begin{rem}[2D case]
For 2D bounded and multiply connected domains, magnetic mean-square potential conservation at weak ideal limits of Leray--Hopf solutions is demonstrated in \cite[Theorem5.4]{FL20}. For a simply connected domain, there exists a flux function $\phi$ and a stream function $\psi$ vanishing on $\partial\Omega$ such that $B=\nabla^{\perp}\phi $ and $u=\nabla^{\perp}\psi $ for $\nabla^{\perp}={}^{t}(\partial_{x_2},-\partial_{x_1})$. An equivalent equation to $(1.1)_2$ in this setting is 

\begin{align*}
\phi_t+u\cdot \nabla \phi=\mu\Delta \phi.
\end{align*}\\
Leray--Hopf solutions satisfy this equation on $L^{2}(\Omega)$ for a.e. $t\in [0,T]$, and by integration by parts, the equality

\begin{align*}
\int_{\Omega}|\phi|^{2}\dd x+2\mu\int_{0}^{t}\int_{\Omega}|\nabla \phi|^{2}\dd x\dd s=\int_{\Omega}|\phi_0|^{2}\dd x,
\end{align*}\\
holds for $t\in [0,T]$. The paper \cite{FL20} showed a uniform bound 

\begin{align*}
\sup_{j} ||\partial_t\phi_j||_{L^{1}(0,T; H^{1}_{0}(\Omega)^{*} ) }<\infty,
\end{align*}\\
by applying the Hardy space theory of compensated compactness of Coifmann et al. \cite{CLMS} and Fefferman's ${\mathcal{H}}^{1}-\textrm{BMO}$ duality \cite{FS72} for the zero extension of $(u,B)$ to $\mathbb{R}^{2}$. The uniform bound implies the strong convergence 

\begin{align*}
\phi_j\to \phi\quad \textrm{in}\ L^{2}(0,T; L^{2}(\Omega) ), 
\end{align*}\\
and that the limit $\phi$ is a distributional solution to the transport equation 

\begin{align*}
\phi_t+u\cdot \nabla \phi=0.
\end{align*}\\
It is remarked that the convergence of $\phi_j$ implies $\phi_j^{2}\to \phi^{2}$ in $L^{2}(0,T; L^{2}(\Omega) )$ by the uniform bound of $\nabla \phi_j\in L^{\infty}(0,T; L^{2}(\Omega))$, and the limit $\phi^{2}$ satisfies the equation $\partial_t\phi^{2}+u\cdot \nabla \phi^{2}=0$ in the distributional sense. Magnetic mean-square potential conservation at weak ideal limits follows by integrating this equation in time. The paper \cite[Lemma 5.7]{FL20} showed magnetic mean-square potential conservation more generally for any transport equation solutions under the condition $u\in L^{\infty}(0,T; L^{2}(\Omega) )$ and $\nabla \phi\in C_w([0,T]; L^{2}(\Omega) )$ by an approximation argument for spatial and time variables, cf. \cite{DL89}, \cite{AC14}.
\end{rem}

\subsection{Application to stability}

Theorems 2.5 and 2.11 imply the stability of the set of minimizers $\mathcal{S}_{h}$. 

\begin{prop}
Let $h\in \mathbb{R}$. Let ${\mathcal{S}}_h$ be as in (2.6). For $\varepsilon>0$, there exists $\delta>0$ such that for $u_0,B_0\in L^{2}_{\sigma}(\Omega)$ satisfying 

\begin{align*}
||u_0||_{L^{2}}+\inf_{U\in S_h}||B_0-U||_{L^{2}}+\left|\int_{\Omega}\textrm{curl}^{-1}B_0\cdot B_0\dd x-h\right|\leq \delta,
\end{align*}\\
there exists a weak ideal limit $(u,B)$ of Leray--Hopf solutions to (1.1)--(1.2) for $(u_0,B_0)$ such that 

\begin{align*}
\textrm{ess sup}_{t>0} \left\{||u||_{L^{2}}+\inf_{U\in {\mathcal{S}}_h}||B-U||_{L^{2}}   \right\}\leq \varepsilon.
\end{align*}
\end{prop}

\vspace{5pt}

\begin{proof}
We argue by contradiction. Suppose that the assertion was false. Then there exists some $\varepsilon_0>0$ such that for any $n\geq 1$, there exists $u_{0,n}$, $B_{0,n}\in L^{2}_{\sigma}(\Omega)$ such that

\begin{align*}
||u_{0,n}||_{L^{2}}+\inf_{U\in {\mathcal{S}}_h}||B_{0,n}-U||_{L^{2}}+\left|\int_{\Omega}\textrm{curl}^{-1}B_{0,n}\cdot B_{0,n}\dd x-h\right|\leq \frac{1}{n},
\end{align*}\\
and the weak ideal limit $(u_n,B_n)$ of Leray--Hopf solutions in Theorem 2.11 satisfying 

\begin{align*}
\textrm{ess sup}_{t>0} \left\{||u_n||_{L^{2}}+\inf_{U\in {\mathcal{S}}_h}||B_n-U||_{L^{2}}   \right\}\geq  \varepsilon_0>0.
\end{align*}\\
We denote by $F_n$ the set of all points $t\in (0,\infty)$ such that 

 \begin{align*}
&\int_{\Omega}(|u_n|^{2}+|B_n|^{2}) \dd x
\leq \int_{\Omega}(|u_{0,n}|^{2}+|B_{0,n}|^{2})\dd x,    \\
&\int_{\Omega}\textrm{curl}^{-1}B_{n}\cdot B_{n}\dd x
=\int_{\Omega}\textrm{curl}^{-1}B_{0,n}\cdot B_{0,n}\dd x.  
\end{align*}\\ 
The set $F_n^{c}=(0,\infty)\backslash F_n$ has measure zero. For $F=\cap_{n=1}^{\infty} F_n$, $F^{c}$ has measure zero and the above inequalities hold for all $t\in F$ and $n\geq 1$. We take a point $t_n\in F$ such that 

 \begin{align*}
||u_n||_{L^{2}}(t_n)+\inf_{U\in {\mathcal{S}}_h}||B_n-U||_{L^{2}}(t_n)  \geq  \frac{\varepsilon_0}{2}.
\end{align*}\\ 
We write $(u_n,B_n)= (u_n,B_n)(\cdot,t_n)$ by suppressing $t_n$. For $h_n=\int_{\Omega}\textrm{curl}^{-1}B_{0,n}\cdot B_{0,n}\dd x$, 

\begin{align*}
{\mathcal{I}}_{h_n}\leq\frac{1}{2}\int_{\Omega} |B_{0,n}|^{2}\dd x \leq \frac{1}{2} \left(\inf_{U\in {\mathcal{S}}_h}||B_{0,n}-U||_{L^{2}}+\sqrt{2}{\mathcal{I}}_h^{1/2}\right)^{2}.
\end{align*}\\
By continuity of ${\mathcal{I}}_h$ for $h\in \mathbb{R}$, letting $n\to\infty$ implies that 

\begin{align*}
h_n&\to h,\\
\frac{1}{2}\int_{\Omega}\left(|u_{0,n}|^{2}+|B_{0,n}|^{2}\right)\dd x &\to {\mathcal{I}}_h.
\end{align*}\\
By helicity conservation and non-increasing total energy, 

\begin{align*}
&\int_{\Omega}\textrm{curl}^{-1}B_{n}\cdot B_{n}\dd x=\int_{\Omega}\textrm{curl}^{-1}B_{0,n}\cdot B_{0,n}\dd x=h_n, \\
&\mathcal{I}_{h_n}\leq \frac{1}{2}\int_{\Omega}\left(|u_{n}|^{2}+|B_{n}|^{2}\right)\dd x\leq  \frac{1}{2}\int_{\Omega}\left(|u_{0,n}|^{2}+|B_{0,n}|^{2}\right)\dd x,
\end{align*}\\
letting $n\to\infty$ implies that

\begin{align*}
\frac{1}{2}\int_{\Omega}\left(|u_{n}|^{2}+|B_{n}|^{2}\right)\dd x \to {\mathcal{I}}_h.
\end{align*}\\
By Theorem 2.5, there exists a subsequence (still denoted by $\{(u_n,B_n)\}$) and some $
B \in {\mathcal{S}}_h$ such that $(u_n,B_n)\to (0,
B)$ in $L^{2}(\Omega)$. Thus 

\begin{align*}
0=\lim_{n\to\infty}\left\{ ||u_n||_{L^{2}}+||B_n-B||_{L^{2}} \right\} 
\geq\liminf_{n\to\infty}\left\{  ||u_n||_{L^{2}}+\inf_{U\in {\mathcal{S}}_h}||B_n-U||_{L^{2}}\right\}
\geq \frac{\varepsilon_0}{2}>0.
\end{align*}\\
We obtained a contradiction. The proof is complete.
\end{proof}

\begin{proof}[Proof of Theorem 1.1]
By characterizing ${\mathcal{S}}_h$ in Lemma 2.3, we deduce Theorem 1.1 from Proposition 2.16. 
\end{proof}

\begin{rem}
It is observed from the proof of Theorem 1.1 that the same stability result holds for weak non-resistive limits of Leray--Hopf solutions (decreasing total energy) by their magnetic helicity conservation \cite[Theorem 1.5]{FL20}. The stability also holds for unique strong solutions to ideal MHD and non-resistive MHD up to maximal existence time. 
\end{rem}

\section{Nonlinear force-free fields}
We apply Clebsch representation for axisymmetric solenoidal vector fields in $L^{2}(\mathbb{R}^{3})$ and their vector potential and define generalized magnetic helicity (1.10) and generalized magnetic mean-square potential (1.14) with the function $\phi_{\infty}=r^{2}+\gamma$ for $\gamma\geq 0$.

\subsection{Clebsch representation}

Clebsch representation is a generalization of the 2D stream (flux) function representation of solenoidal vector fields to 3D symmetric solenoidal vector fields. The simplest Clebsch representation is that for translationally symmetric solenoidal vector fields $b(x_1,x_2)={}^{t}(b_1(x_1,x_2),b_2(x_1,x_2),b_3(x_1,x_2))\in L^{2}_{\sigma}(\mathbb{R}^{2})$, 

\begin{align*}
b(x_1,x_2)=\nabla \times (\phi(x_1,x_2)\nabla z)+G(x_1,x_2)\nabla z.
\end{align*}\\
Here, $\nabla=\nabla_x$ is gradient for $x\in \mathbb{R}^{3}$ and $z=x_3$. We call $\phi$ and $G$ Clebsch potentials. The $L^{2}$-norm of $b$ is expressed as 

\begin{align*}
\int_{\mathbb{R}^{2}}|b|^{2}\dd x'=\int_{\mathbb{R}^{2}}\left(|\nabla \phi|^{2}+|G|^{2}\right)\dd x',
\end{align*}\\
and thus the 3-dimensional space $L^{2}_{\sigma}(\mathbb{R}^{2})$ is identified with the product space $ \dot{H}^{1}(\mathbb{R}^{2})\times L^{2}(\mathbb{R}^{2})$. Clebsch representation of the vector potential to $b$ is expressed as 

\begin{align*}
a(x_1,x_2)=\nabla \times (\eta(x_1,x_2)\nabla z)+\phi(x_1,x_2)\nabla z,
\end{align*}\\
where $\eta$ is a solution of the Poisson equation $-\Delta_{\mathbb{R}^{2}}\eta=G$ in $\mathbb{R}^{2}$. The vector field $a\in \textrm{BMO}(\mathbb{R}^{2})$ is a unique potential such that $\nabla \times a=b$ and $\nabla \cdot a=0$. 

We consider Clebsch representation for axisymmetric solenoidal vector fields $b\in L^{2}_{\sigma,\textrm{axi}}(\mathbb{R}^{3})=\{b\in L^{2}_{\sigma}(\mathbb{R}^{3})\ |\ b: \textrm{axisymmetric}\ \}$,  

\begin{align}
b(x)=\nabla \times (\phi(z,r)\nabla \theta)+G(z,r)\nabla \theta. 
\end{align}\\
Here, $(r,\theta,z)$ is the cylindrical coordinate. The $L^{2}$-norm of $b$ is expressed as 

\begin{align}
\int_{\mathbb{R}^{3}}|b|^{2}\dd x=2\pi\int_{\mathbb{R}^{2}_{+}}\left(|\nabla \phi|^{2}+|G|^{2}\right)\frac{1}{r}\dd z\dd r. 
\end{align}\\
Clebsch representation of the vector potential of $b$ is expressed as 

\begin{align}
a(x)=\nabla \times (\eta(z,r)\nabla \theta)+\phi(z,r)\nabla \theta,
\end{align}\\
where $\eta$ is a solution of the Dirichlet problem 

\begin{align}
-L\eta=G\quad \textrm{in}\ \mathbb{R}^{2}_{+},\quad \eta=0\quad \textrm{on}\ \partial\mathbb{R}^{2}_{+}, 
\end{align}\\
for the operator $L=\Delta_{z,r}-r^{-1}\partial_r$. The Green function of this problem \cite[p.4]{FT81}, \cite[p.472]{Friedman82}, \cite[19.1]{SverakLec} is of the form 

\begin{equation}
\begin{aligned}
&\eta(z,r)=\int_{\mathbb{R}^{2}_{+}}{\mathcal{G}}(z,r,z',r')\frac{G(z',r')}{r'}\dd z'\dd r',\\
&{\mathcal{G}}(z,r,z',r')= \frac{rr'}{2\pi}\int_{0}^{\pi}\frac{\cos\theta \dd \theta}{\sqrt{|z-z'|^{2}+r^{2}+r'^{2}-2rr'\cos\theta  }}.
\end{aligned}
\end{equation}\\
We denote $(3.5)_1$ by $\eta=(-L)^{-1}G$. The Green function can be estimated by asymptotic expansions of complete elliptic integrals of the first and the second kind, e.g., \cite{Friedman82}. Its higher-order estimates are derived in \cite[Corollary 2.9]{FengSverak} from the representation

\begin{equation}
\begin{aligned}
{\mathcal{G}}(z,r,z',r')&=\frac{\sqrt{rr'}}{2\pi}F(s),\quad s= \frac{|z-z'|^{2}+|r-r'|^{2}}{rr'},\\
F(s)&=\int_{0}^{\pi}\frac{\cos\theta}{\sqrt{2(1-\cos\theta)+s}}\dd \theta.
\end{aligned}
\end{equation}\\
The function $F(s)$ has the asymptotics $F(s)=-(1/2)\log s+\log{8}-2+O(s\log s)$ as $s\to 0$ and $O(s^{-3/2})$ as $s\to\infty$. The $k$-th derivative $F^{(k)}(s)$ is also estimated from the asymptotic expansion. They satisfy

\begin{equation}
\begin{aligned}
&F(s)\lesssim \frac{1}{s^{\tau}},\quad 0<\tau\leq \frac{3}{2}, \\
&F^{(k)}(s)\lesssim \frac{1}{s^{k+\tau}},\quad 0\leq \tau\leq \frac{3}{2},\ k\in \mathbb{N}.
\end{aligned}
\end{equation}\\

\subsection{The weighted Hilbert space}
Let $L^{2}(\mathbb{R}^{2}_{+};r^{-1})$ denote the weighted $L^{2}$ space on the cross section $\mathbb{R}^{2}_{+}=\{{}^{t}(z,r)\ |\ z\in \mathbb{R},\ r>0\}$ with the measure $r^{-1}\dd z\dd r$. Let $\dot{H}^{1}_{0}(\mathbb{R}^{2}_{+};r^{-1} )$ denote the homogeneous $L^{2}(\mathbb{R}^{2}_{+};r^{-1})$-Sobolev space of trace zero functions on $\partial\mathbb{R}^{2}_{+}$. We take the trace at $r=0$ in $H^{1/2}_{\textrm{loc}}(\mathbb{R})$ by $L^{2}(\mathbb{R}^{2}_{+};r^{-1})\subset L^{2}_{\textrm{loc}}(\overline{\mathbb{R}^{2}_{+}})$. By the weighted Sobolev inequality \cite[Lemma 3]{Van13},  

\begin{align}
\left(\int_{\mathbb{R}^{2}_{+}}|\phi|^{p}\frac{1}{r^{2+p/2}}\dd z\dd r\right)^{1/p}
\leq C \left(\int_{\mathbb{R}^{2}_{+}}|\nabla\phi|^{2}\frac{1}{r}\dd z\dd r\right)^{1/2},\quad 2\leq p<\infty, 
\end{align}\\
the space $\dot{H}^{1}_{0}(\mathbb{R}^{2}_{+};r^{-1} )$ continuously embeds into $L^{p}(\mathbb{R}^{2}_{+}; r^{-2-p/2})$. Moreover, Rellich--Kondrakov theorem holds in the weighted space \cite[Lemma 3.1]{A8}, i.e.,  

\begin{align}
\dot{H}^{1}_{0}(\mathbb{R}^{2}_{+};r^{-1})\subset \subset L^{p}_{\textrm{loc}}(\overline{\mathbb{R}^{2}_{+}};r^{-1}),\quad 1\leq p <\infty.    
\end{align}\\
The space $\dot{H}^{1}_{0}(\mathbb{R}^{2}_{+};r^{-1} )$ is a Hilbert space isometrically isomorphic to $\dot{H}^{1}_{\textrm{axi}}(\mathbb{R}^{5})= \{ \varphi\in \dot{H}^{1}(\mathbb{R}^{5})\ | \ \varphi: \textrm{axisymmetric } \}$ \cite[Lemma 2.2]{AF88} by the transform

\begin{align}
\phi(z,r)\longmapsto \varphi(y)=\frac{\phi(z,r)}{r^{2}},\quad y={}^{t}(y_1,y'),\ y_1=z,\ |y'|=r,     
\end{align}\\
 in the sense that 

\begin{equation}
\begin{aligned}
\int_{\mathbb{R}^{2}_{+}}\nabla \phi\cdot \nabla \tilde{\phi}\frac{2\pi^{2}}{r}\dd z\dd r
=\int_{\mathbb{R}^{5}}\nabla_y\varphi \cdot \nabla_y\tilde{\varphi} \dd y,\quad \varphi=\frac{\phi}{r^{2}}, \tilde{\varphi}=\frac{\tilde{\phi}}{r^{2}}.
\end{aligned}
\end{equation}\\
The inverse transform of (3.10) induces the isometries from $L^{1}(\mathbb{R}^{5})$ into $L^{1}(\mathbb{R}^{2})$ and from $L^{2}(\mathbb{R}^{5})$ into $L^{2}(\mathbb{R}^{2}_{+};r^{-1})$, respectively, in the sense that 

\begin{equation}
\begin{aligned}
||\phi||_{L^{1}(\mathbb{R}^{3})}&=\frac{1}{\pi} ||\varphi||_{L^{1}(\mathbb{R}^{5}) }, \\
||\phi||_{L^{2}(\mathbb{R}^{2}_{+};r^{-1} )}&=\frac{1}{\sqrt{2}\pi} ||\varphi||_{L^{2}(\mathbb{R}^{5}) }.
\end{aligned}
\end{equation}

\vspace{5pt}

We show the Poincar\'e inequality with the weighted measure for the later usage in Section 5.

\begin{prop}
There exists a constant $C$ such that 

\begin{align}
\int_{D(0,2R)\backslash D(0,R)}|\phi |^{2}\frac{1}{r}\dd z\dd r\leq CR^{2} \int_{D(0,2R)\backslash D(0,R)}|\nabla \phi |^{2}\frac{1}{r}\dd z\dd r, 
\end{align}\\
for $\phi\in \dot{H}^{1}_{0}(\mathbb{R}^{2}_{+};r^{-1})$, where $D(0,R)=\{{}^{t}(z,r)\in \mathbb{R}^{2}_{+} \ |\ |z|^{2}+r^{2}<R^{2} \}$ and $R>0$.  
\end{prop}

\vspace{5pt}

\begin{proof}
We reduce to the case $R=1$ by dilation. Suppose that (3.13) were false, then there exists a sequence $\{\phi_n\}\subset \dot{H}^{1}_{0}(\mathbb{R}^{2}_{+};r^{-1})$ such that 

\begin{align*}
\int_{D(0,2)\backslash D(0,1)}|\phi_n |^{2}\frac{1}{r}\dd z\dd r=1,\quad 
\int_{D(0,2)\backslash D(0,1)}|\nabla \phi_n |^{2}\frac{1}{r}\dd z\dd r\to 0.
\end{align*}\\
By the transform (3.10), $\varphi_n=\phi_n/r^{2}$ satisfies 

\begin{align*}
\int_{B(0,2)\backslash B(0,1)}|\nabla \varphi_n |^{2}\dd y=
\int_{D(0,2)\backslash D(0,1)}|\nabla \phi_n |^{2}\frac{2\pi^{2}}{r}\dd z\dd r,
\end{align*}\\
where $B(0,R_0)$ denotes the open ball in $\mathbb{R}^{5}$ centered at the origin with radius $R_0>0$. By Rellich--Kondrakov theorem, there exist $\{n_k\}$ and some axisymmetric $\varphi$ such that $\varphi_{n_{k}}\to \varphi$ in $L^{2}(B(0,2)\backslash B(0,1))$. Thus the function $\phi=r^{2}\varphi$ satisfies 

\begin{align*}
\int_{D(0,2)\backslash D(0,1)}| \phi_{n_{k}}-\phi |^{2}\frac{2\pi^{2}}{r}\dd z\dd r
=\int_{B(0,2)\backslash B(0,1)}| \varphi_{n_{k}}-\varphi |^{2}\dd y
\to 0.
\end{align*}\\
This implies that 

\begin{align*}
\int_{D(0,2)\backslash D(0,1)}|\phi |^{2}\frac{1}{r}\dd z\dd r=1.
\end{align*}\\
Since $\nabla \phi=0$ and $\phi(z,0)=0$, we have $\phi\equiv 0$. We obtained a contradiction. 
\end{proof}

\begin{rem}
The weighted Sobolev inequality (3.8) also holds in a disk $D(0,R)$. For example, 

\begin{align}
\left(\int_{D(0,R)}|\phi|^{10/3}\frac{1}{r^{11/3}}\dd z\dd r\right)^{3/10}
\leq C \left(\int_{D(0,R)}|\nabla\phi|^{2}\frac{1}{r}\dd z\dd r\right)^{1/2},
\end{align}\\
holds for $\phi\in \dot{H}^{1}_{0}(\mathbb{R}^{2}_{+};r^{-1})$. The Poincar\'e inequality $||\phi||_{L^{2}(D(0,1);r^{-1}) }\leq C ||\nabla \phi||_{L^{2}(D(0,1);r^{-1}) }$ can be demonstrated in a similar way as (3.13) and hence the homogeneous Sobolev inequality $||\varphi||_{L^{10/3}(B(0,1))}\leq C|| \nabla \varphi||_{L^{2}(B(0,1))}$ holds for $\varphi\in \dot{H}^{1}_{\textrm{axi}}(\mathbb{R}^{5})$ by the transform (3.10). The inverse transform of (3.10) and a dilation argument imply (3.14).
\end{rem}

\subsection{Anisotropic estimates}
We demonstrate the Clebsch representation (3.1) and (3.3).

\begin{lem}
For $b\in L^{2}_{\sigma,\textrm{axi}}(\mathbb{R}^{3})$, there exist unique $\phi\in \dot{H}^{1}_{0} (\mathbb{R}^{2}_{+};r^{-1})$ and $G\in L^{2}(\mathbb{R}^{2}_{+};r^{-1})$ such that (3.1) holds.
\end{lem}

\vspace{5pt}

\begin{proof}
For $b\in L^{2}_{\sigma,\textrm{axi}}(\mathbb{R}^{3})$, we set $G(z,r)=r^{2} b\cdot \nabla \theta$. Then $b-G\nabla \theta$ is axisymmetric without a swirl, and there exists an axisymmetric stream function $\phi(z,r)$ such that $b-G\nabla \theta=\nabla \times (\phi\nabla \theta)$. The functions $\nabla \phi$ and $G$ belong to $L^{2}(\mathbb{R}^{2}_{+};r^{-1})$ by (3.2). For smooth $b$, we may assume that $\phi(z,0)=0$ since $\nabla \phi$ vanishes on $\{r=0\}$ by

\begin{align*}
rb=\nabla \phi\times r\nabla \theta+G r\nabla \theta.
\end{align*}\\
For general $b\in L^{2}_{\sigma,\textrm{axi}}(\mathbb{R}^{3})$, we take a sequence $\{b_m\} \subset C^{\infty}\cap L^{2}_{\sigma,\textrm{axi}}(\mathbb{R}^{3})$ such that $b_m\to b$ in $L^{2}(\mathbb{R}^{3})$. The potentials $\phi_m\in \dot{H}^{1}_{0}(\mathbb{R}^{2}_{+};r^{-1}) $ and $G_m\in L^{2}(\mathbb{R}^{2}_{+};r^{-1})$ of $b_m$ satisfy

\begin{align*}
0=\lim_{m\to\infty}\int_{\mathbb{R}^{3}}|b-b_m|^{2}\dd x=\lim_{m\to\infty}\int_{\mathbb{R}^{2}_{+}}\left(|\nabla (\phi-\phi_m) |^{2}+|G-G_m|^{2}\right)\frac{2\pi}{r}\dd z\dd r.
\end{align*}\\
Thus $\phi(z,0)=0$. The potentials $\phi\in \dot{H}^{1}_{0}(\mathbb{R}^{2}_{+};r^{-1})$ and $G \in L^{2}(\mathbb{R}^{2}_{+};r^{-1})$ of (3.1) are unique by the trace condition at $r=0$. 
\end{proof}

\vspace{5pt}

\begin{prop}
For $G\in L^{2}(\mathbb{R}^{2}_{+};r^{-1})$, $\eta=(-L)^{-1}G$ satisfies 

\begin{align}
||\nabla_{z,r}\eta||_{L^{p}(\mathbb{R}) }(r)\lesssim r^{1/p+1/2}||G||_{L^{2}(\mathbb{R}^{2}_{+};r^{-1}) },\quad r>0,\ 2\leq p<\infty.
\end{align}
\end{prop}

\vspace{5pt}

\begin{proof}
By the 1-dimensional convolution $*_1$ for $z$-variable,

\begin{align*}
\nabla_{z,r}\eta(z,r)=\int_{0}^{\infty}\nabla_{z,r}{\mathcal{G}}*_1G(\cdot,r')\frac{1}{r'}\dd r'.
\end{align*}\\
By Young's inequality for $1/p=1/q-1/2$, $1\leq q<2$, 

\begin{align*}
||\nabla_{z,r}\eta||_{L^{p}(\mathbb{R}) }(r)
&\leq \int_{0}^{\infty}||\nabla {\mathcal{G}}||_{L^{q}(\mathbb{R}) }(r,r') ||G ||_{L^{2}(\mathbb{R}) }\frac{1}{r'}\dd r' \\
&\leq \left(\int_{0}^{\infty}||\nabla {\mathcal{G}}||_{L^{q}(\mathbb{R}) }^{2}(r,r') \frac{1}{r'}\dd r'\right)^{1/2}||G||_{L^{2}(\mathbb{R}^{2}_{+};r^{-1} ) }.
\end{align*}\\
By the pointwise estimates (3.7), 

\begin{align*}
\left|\nabla_{z,r} {\mathcal{G}}(z,r,z',r')\right|\lesssim \frac{1}{s^{1/2+\tau}}+\sqrt{\frac{r'}{r}}\frac{1}{s^{\tilde{\tau}}},\quad 0<\tau\leq \frac{3}{2},\ 0\leq \tilde{\tau}\leq \frac{3}{2}.
\end{align*}\\
By using  

\begin{align*}
\left\|\frac{1}{s^{\alpha}}\right\|_{L^{q}(\mathbb{R}) }=C\frac{(rr')^{\alpha}}{|r-r' |^{2\alpha-1/q}},\quad \alpha>\frac{1}{2q},
\end{align*}\\
we estimate 

\begin{align*}
\left\|\nabla_{z,r} {\mathcal{G}}\right\|_{L^{q}(\mathbb{R})}^{2}(r,r')
\lesssim \left( \frac{(rr')^{2\tau+1}}{|r-r' |^{4\tau +2-2/q}}+\frac{r' (rr')^{2\tilde{\tau}}}{r |r-r' |^{4\tilde{\tau} -2/q}}  \right),\quad \frac{1}{2q}- \frac{1}{2}<\tau\leq \frac{3}{2},\  \frac{1}{2q}<\tilde{\tau}\leq \frac{3}{2}.
\end{align*}\\
For $1/(2q)-1/2<\tau\leq 3/2$, 

\begin{align*}
\int_{\{|r-r'|<r/2\}}\frac{(rr')^{2\tau+1}}{|r-r' |^{4\tau +2-2/q}r'}\dd r'=Cr^{2/q}.
\end{align*}\\
For $1/q-1/2<\tau\leq 3/2$, 

\begin{align*}
\int_{\{|r-r'|\geq r/2\}}\frac{(rr')^{2\tau+1}}{|r-r' |^{4\tau +2-2/q}r'}\dd r'=Cr^{2/q}.
\end{align*}\\
Thus 

\begin{align*}
\int_{\mathbb{R}}\frac{(rr')^{2\tau+1}}{|r-r' |^{4\tau +2-2/q}r'}\dd r'=Cr^{2/q}.
\end{align*}\\
By choosing different $\tilde{\tau}$ in $|r-r'|<r/2$ and in $|r-r'|\geq r/2$ in a similar way,
 
\begin{align*}
\int_{\mathbb{R}}\frac{r'(rr')^{2\tilde{\tau}+1}}{|r-r' |^{4\tilde{\tau} -2/q}r'}\dd r'=Cr^{2/q}.
\end{align*}\\
Thus (3.15) holds.
\end{proof}

\vspace{5pt}

\begin{lem}
The vector potential $a\in L^{6}(\mathbb{R}^{3})$ in (3.3) is a unique potential of $b\in L^{2}_{\sigma,\textrm{axi}}(\mathbb{R}^{3})$ such that $\nabla \times a=b$ and $\nabla \cdot a=0$.
\end{lem}

\vspace{5pt}

\begin{proof}
By (3.8) and (3.15), $a$ is locally integrable in $\mathbb{R}^{3}$ and $\nabla \times a=b$ and $\nabla \cdot a=0$. By $||\nabla a||_{L^{2}}=||\nabla \times a||_{L^{2}}=||b||_{L^{2}}$ and Sobolev embedding, $a$ belongs to $L^{6}(\mathbb{R}^{3})$. The uniqueness follows from the Liouville theorem.
\end{proof}

\begin{rem}
The condition $a\in L^{6}(\mathbb{R}^{3})$ implies that $\nabla_{z,r}\eta$ and $\phi$ belong to $L^{6}(\mathbb{R}^{2}_{+}; r^{-5})$. The weighted Sobolev inequality (3.8) implies more general properties $\phi\in L^{p}(\mathbb{R}^{2}_{+}; r^{-2-p/2})$ for $2\leq p<\infty$, whereas (3.15) yields the weaker integrability for $\nabla_{z,r}\eta$ at $r=0$ and $r=\infty$, i.e., 

\begin{align*}
||\nabla_{z,r} \eta||_{L^{p}(\mathbb{R})}^{p}\frac{1}{r^{2+p/2}}\leq \frac{C}{r}||G||_{L^{2}(\mathbb{R}^{2}_{+};r^{-1}) }^{p}\quad r>0.
\end{align*}\\
The trace of $\nabla_{z,r} \eta$ at $r=0$ vanishes on $L^{p}(\mathbb{R})$ by (3.15). 
\end{rem}

\begin{rem}
The constant $B_{\infty}=-2e_z$ can be expressed as $B_{\infty}=-\nabla \times (\phi_{\infty}\nabla \theta)$ by the function $\phi_{\infty}=r^{2}+\gamma$ for any $\gamma\geq 0$. By the Clebsch representation (3.1) and (3.3), $B=b+B_{\infty}$ and one of its vector potentials $A$ are expressed as   

\begin{equation}
\begin{aligned}
B&=\nabla \times( \Phi\nabla \theta)+G\nabla \theta,\\
A&=\nabla \times (\eta\nabla \theta)+\Phi\nabla \theta, \\
\Phi&=\phi-\phi_{\infty}.  
\end{aligned}
\end{equation}\\
The pseudo-scalar $A\cdot B$ is expressed as 

\begin{align*}
A\cdot B=2\Phi \frac{G}{r^{2}}+\nabla \cdot ( \Phi\nabla \theta\times A ).
\end{align*}
\end{rem}

\begin{prop}
Let $1<q<3$ and $1/p=1/q-1/3$. For $b\in L^{q}_{\sigma}(\mathbb{R}^{3})$, there exists a unique $a\in L^{p}_{\sigma}(\mathbb{R}^{3})$ such that $\nabla a\in L^{q}(\mathbb{R}^{3})$ and $\nabla \times a=b$.
\end{prop}

\begin{proof}
For $b\in C^{\infty}_{c,\sigma}(\mathbb{R}^{3})$, we set $a=E*(\nabla \times b)$ by the fundamental solution of the Laplace equation $E(x)=(4\pi)^{-1}|x|^{-1}$ and the convolution $*$ in $\mathbb{R}^{3}$. Then, $\nabla \times a=b$ and $\nabla \cdot a=0$. By the Calderon-Zygmund inequality for the Newton potential \cite[Theorem 9.9]{GT} and the Sobolev inequality, 

\begin{align*}
||\nabla a||_{L^{q}(\mathbb{R}^{3}) }+||a||_{L^{p}(\mathbb{R}^{3}) }\leq C||b||_{L^{q}(\mathbb{R}^{3}) }.
\end{align*}\\
For $b\in L^{q}_{\sigma}(\mathbb{R}^{3})$, we take a sequence $\{b_n\}\subset  C^{\infty}_{c,\sigma}(\mathbb{R}^{3})$ such that $b_n\to b$ in $L^{q}(\mathbb{R}^{3})$. Then, $a_n=E*(\nabla \times b_n)\to a$  in $L^{p}(\mathbb{R}^{3})$ and $\nabla a_n\to \nabla a$ in $L^{q}(\mathbb{R}^{3})$ for some $a$ by the above inequality. The limit $a$ satisfies the desired properties. The uniqueness follows from the Liouville theorem.  
\end{proof}

\subsection{Generalized magnetic helicity}

We define generalized magnetic helicity (1.10) and generalized mean-square potential (1.14) for any axisymmetric solenoidal vector fields in $L^{2}_{\sigma,\textrm{axi}}(\mathbb{R}^{3})$.

\begin{prop}
Let $1\leq q<\infty$. Let $\phi_{\infty}=r^{2}+\gamma$ and $\gamma\geq 0$. For $\phi\in \dot{H}^{1}_{0}(\mathbb{R}^{2}_{+}; r^{-1})$, 

\begin{align}
|\{x\in \mathbb{R}^{3}\ |\ \phi(x)>\phi_{\infty} \}|
&\lesssim  ||\nabla \phi||_{L^{2}(\mathbb{R}^{2}_{+}; r^{-1}) }^{2},  \\
||(\phi-\phi_{\infty})_{+}||_{L^{q} (\mathbb{R}^{2}_{+};r^{-1}) }
&\lesssim  ||\nabla \phi||^{4/3+2/(3q) }_{L^{2}(\mathbb{R}^{2}_{+}; r^{-1}) },    \\
||(\phi-\phi_{\infty})_{+}||_{L^{q}(\mathbb{R}^{3}) }
&\lesssim  ||\nabla \phi||_{L^{2}(\mathbb{R}^{2}_{+}; r^{-1}) }^{4/3+2/q},   
\end{align}\\
where $|\cdot |$ denotes the Lebesgue measure in $\mathbb{R}^{3}$. 
\end{prop}

\begin{proof}
By the weighted Sobolev inequality (3.8), 

\begin{align*}
|\{x\in \mathbb{R}^{3}\ |\ \phi(x)>\phi_{\infty} \}|
=\int_{\{\phi>\phi_{\infty}\}}\dd x
=2\pi \int_{\{\phi>\phi_{\infty}\}}\frac{1}{r}\dd z\dd r
\leq 2\pi \int_{\{\phi>r^{2}\}}\phi^{2}\frac{1}{r^{3}}\dd z\dd r
\lesssim ||\nabla \phi||_{L^{2}(\mathbb{R}^{2}_{+};r^{-1}) }^{2}.
\end{align*}\\
Thus (3.17) holds. For $1\leq q<\infty$,  we set $p=l+q$ and $l=(q+2)/3$ to estimate 

\begin{align*}
\int_{\mathbb{R}^{2}_{+}}(\phi-\phi_{\infty})_{+}^{q}\frac{1}{r}\dd z\dd r
\lesssim \int_{\mathbb{R}^{2}_{+}}\phi^{l+q}\frac{1}{r^{2l+1}}\dd z\dd r= \int_{\mathbb{R}^{2}_{+}}\phi^{p}\frac{1}{r^{p/2+2}}\dd z\dd r\lesssim  ||\nabla \phi||_{L^{2}(\mathbb{R}^{2}_{+};r^{-1}) }^{p}.
\end{align*}\\
Thus (3.18) holds. Similarly, we set $p=l+q$ and $l=q/3+2$ to estimate 

\begin{align*}
\int_{\mathbb{R}^{3}}(\phi-\phi_{\infty})_{+}^{q}\dd x\lesssim \int_{\mathbb{R}^{2}_{+}}\phi^{q+l}\frac{1}{r^{2l-1}}\dd z\dd r
= \int_{\mathbb{R}^{2}_{+}}\phi^{p}\frac{1}{r^{p/2+2}}\dd z\dd r
\lesssim ||\nabla \phi||_{L^{2}(\mathbb{R}^{2}_{+};r^{-1}) }^{p}.
\end{align*}\\
Thus (3.19) holds.
\end{proof}

\begin{lem}[Arnold-type inequality]
\begin{align}
\left|\int_{\mathbb{R}^{3}}(\phi-\phi_{\infty})_{+}G\frac{1}{r^{2}} \dd x\right|
&\lesssim \left(\int_{\mathbb{R}^{3}}|b|^{2}\dd x\right)^{4/3},  \\   
\left|\int_{\mathbb{R}^{3}}(\phi-\phi_{\infty})_{+}^{2} \dd x\right|  
&\lesssim  \left(\int_{\mathbb{R}^{3}}|b|^{2}\dd x\right)^{7/3},   
\end{align}\\
for $b=\nabla \times (\phi\nabla \theta)+G\nabla \theta\in L^{2}_{\sigma,\textrm{axi}}(\mathbb{R}^{3})$.  
\end{lem}

\begin{proof}
By H\"older's inequality, (3.18) and (3.2), 

\begin{align*}
\left|\int_{\mathbb{R}^{3}}(\phi-\phi)_{+}G\frac{1}{r^{2}}\dd x\right|=2\pi \left|\int_{\mathbb{R}^{2}_{+}}(\phi-\phi)_{+}G\frac{1}{r}\dd z\dd r\right|
&\leq 2\pi ||(\phi-\phi)_{+}||_{L^{2}(\mathbb{R}^{2}_{+};r^{-1}) }||G||_{L^{2}(\mathbb{R}^{2}_{+};r^{-1}) } \\
&\lesssim ||b||_{L^{2}(\mathbb{R}^{3}) }^{8/3}.
\end{align*}\\
Thus (3.20) holds. The inequality (3.21) follows from (3.19) and (3.2).  
\end{proof}

\begin{defn}
Let $\phi_{\infty}=r^{2}+\gamma$ and $\gamma\geq 0$. For $v, b=\nabla \times (\phi\nabla \theta)+G\nabla \theta\in L^{2}_{\sigma,\textrm{axi}}(\mathbb{R}^{3})$, we set

\begin{align}
E[b]&=\frac{1}{2}\int_{\mathbb{R}^{3}}|b|^{2} \dd x \hspace{50pt} \textrm{(magnetic energy),} \\
H_{\gamma}[b]&=2\int_{\mathbb{R}^{3}}(\phi-\phi_{\infty})_{+}\frac{G}{r^{2}} \dd x\hspace{13pt} \textrm{(generalized magnetic helicity),} \\
M_{\gamma}[b]&=\int_{\mathbb{R}^{3}}(\phi-\phi_{\infty})_{+}^{2} \dd x \hspace{32pt} \textrm{(generalized magnetic mean-square potential),}\\
{\mathcal{E}}[v,b]&=E[v]+E[b] \hspace{51pt} \textrm{(total energy)}. 
\end{align}\\
We denote by $H[\cdot]=H_\gamma[\cdot]$ and $M[\cdot]=M_\gamma[\cdot]$ by suppressing $\gamma$. 
\end{defn}

Magnetic helicity reflects direction of vector fields since $\textrm{curl}^{-1}B\cdot B$ is a pseudo-scalar, i.e., $\textrm{curl}^{-1}\tilde{B}\cdot \tilde{B}(x)=-(\textrm{curl}^{-1}B\cdot B)(-x)$ for $\tilde{B}(x)=B(-x)$. A special property of generalized magnetic helicity is symmetry concerning the $\theta$-component of axisymmetric solenoidal vector fields.\\ 

\begin{prop}[Symmetry of generalized magnetic helicity]
\begin{align}
H[\nabla \times (\phi\nabla \theta)+G\nabla \theta ]=-H[\nabla \times (\phi\nabla \theta)-G\nabla \theta ]
\end{align}\\
for $b=\nabla \times (\phi\nabla \theta)+G\nabla \theta\in L^{2}_{\sigma,\textrm{axi}}(\mathbb{R}^{3})$.
\end{prop}

\begin{rem}[Gauge dependence]
We will see in Section 6 that (3.23) and (3.24) are conserved for smooth solutions of ideal MHD (1.1) under the condition (1.7). The quantities (3.23) and (3.24) depend on the gauge $\gamma\geq 0$ since the potential $\phi_\infty=r^{2}+\gamma$ of $B_{\infty}=-e_z$ has freedom on the choice of $\gamma$. 
\end{rem}

\begin{rem}
Generalized magnetic helicity (3.23) and generalized magnetic mean-square potential (3.24) are also well-defined for $\phi_{\infty}=Wr^{2}/2+\gamma$, $W>0$, $\gamma\geq 0$, and $ b\in L^{2}_{\sigma,\textrm{axi}}(\mathbb{R}^{3})$.
\end{rem}

\section{The variational principle}

We develop the variational problem (1.13). We aim to demonstrate that the minimum (1.13) is symmetric and lower semi-continuous for $h\in \mathbb{R}$ and increasing and strictly subadditive for $|h|$ as stated in Lemma 4.10. Using Steiner symmetrization, we derive these properties from the existence of symmetric minimizers for $z$-variable. In the next section, we apply these minimum properties to demonstrate the compactness of (non-symmetric) minimizing sequences.

\subsection{Axisymmetric nonlinear force-free fields}

We set a variational problem and demonstrate that minimizers provide axisymmetric nonlinear force-free fields with discontinuous factors in $\mathbb{R}^{3}$. We also show that flux functions of minimizers are non-negative solutions to the Grad--Shafranov equation.

\begin{defn}
Let $h\in \mathbb{R}$. Let $\phi_{\infty}=r^{2}+\gamma$ and $\gamma\geq 0$. Let $E[\cdot]$ and $H[\cdot]$ be as in (3.22) and (3.23). We set 

\begin{align}
I_{h,\gamma}&=\inf \left\{E[b] \ \Bigg|\ b \in L^{2}_{\sigma,\textrm{axi}}(\mathbb{R}^{3}),\  H[b]=h   \right\},\\
S_{h,\gamma}&=\left\{b\in L^{2}_{\sigma,\textrm{axi}}(\mathbb{R}^{3})\ \middle|\ E[b]=I_{h,\gamma},\ H[b]=h \right\}.
\end{align}\\
We denote by $I_{h}=I_{h,\gamma}$ and $S_{h}=S_{h,\gamma}$ by suppressing $\gamma$. 
\end{defn}

\begin{prop}[Lagrange multiplier]
Let $h\in \mathbb{R}$ and $\gamma\geq 0$. Let $\phi_{\infty}=r^{2}+\gamma$. Let $b=\nabla\times( \phi \nabla \theta) + G\nabla \theta\in L^{2}_{\sigma,\textrm{axi}}(\mathbb{R}^{3})$ satisfy $E[b]=I_{h}$ and $H[b]=h$. There exists a constant $\kappa\in \mathbb{R}$ such that $G=\kappa (\phi-\phi_\infty)_{+}$ and $\phi\in \dot{H}^{1}_{0}(\mathbb{R}^{2}_{+};r^{-1})$ is a weak solution of  

\begin{align}
-L\phi =\kappa^{2} (\phi -\phi_\infty)_{+}\quad \textrm{in}\ \mathbb{R}^{2}_{+},\quad
\phi =0\quad \textrm{on}\ \partial \mathbb{R}^{2}_{+}, 
\end{align}\\
in the sense that 

\begin{align}
\int_{\mathbb{R}^{2}_{+}}\nabla \phi \cdot \nabla \tilde{\phi} \frac{1}{r}\dd z\dd r =\kappa^{2} \int_{\mathbb{R}^{2}_{+}} (\phi-\phi_{\infty})_{+}\tilde{\phi} \frac{1}{r}\dd z\dd r,
\end{align}\\
for all $\tilde{\phi}\in \dot{H}^{1}_{0}(\mathbb{R}^{2}_{+};r^{-1})$. The constant $\kappa$ satisfies  

\begin{align}
h=2\kappa \int_{\mathbb{R}^{3}}(\phi-\phi_\infty)_{+}^{2}\frac{1}{r^{2}}\dd x. 
\end{align}\\
For $h=0$, $\phi=0$ and $G=0$. For $h\neq 0$, $\int_{\mathbb{R}^{3}}(\phi-\phi_\infty)_{+}^{2}r^{-2}\dd x$ is positive.
\end{prop}

\begin{proof}
For $b_0, \tilde{b}\in L^{2}_{\sigma,\textrm{axi}}(\mathbb{R}^{3})$, we set 

\begin{align*}
j(\tau,s)=H[b+\tau \tilde{b}+sb_0],\quad \tau,s\in \mathbb{R}.
\end{align*}\\
By differentiating this by $s$, for $b_0=\nabla \times (\phi_0\nabla \theta)+G_0\nabla \theta$,  

\begin{align*}
\frac{\partial j}{\partial s}(0,0)=2\int_{\mathbb{R}^{3}}(\phi_0 1_{(0,\infty)}(\phi-\phi_{\infty})G+(\phi-\phi_{\infty})_{+}G_0   )\frac{1}{r^{2}}\dd x.
\end{align*}\\
Suppose that $\partial_s j(0,0)=0$ for all $b_0\in L^{2}_{\sigma,\textrm{axi}}(\mathbb{R}^{3})$. Then,  $(\phi-\phi_{\infty})_{+}=0$ and 

\begin{align*}
h=H[b]=2\int_{\mathbb{R}^{3}}(\phi-\phi_{\infty})_{+}\frac{G}{r^{2}}\dd x=0.
\end{align*}\\
By $0=I_0=E[b]$, we have $b= 0$. Thus, the assertion holds for $\kappa=0$.  

We may assume that $\partial_s j(0,0)\neq 0$ for some $b_0\in L^{2}_{\sigma,\textrm{axi}}(\mathbb{R}^{3})$. By the implicit function theorem, there exists a function $s=s(\tau)$ such that for small $|\tau|$, 

\begin{align*}
j(\tau,s(\tau))=h.
\end{align*}\\
By differentiating this by $\tau$, 

\begin{align*}
\dot{s}(0)=-\frac{\partial_{\tau}j(0,0)}{\partial_{s} j(0,0)}.
\end{align*}\\
We set 

\begin{align*}
\kappa=\frac{2}{\partial_{s} j(0,0)}\left(\int_{\mathbb{R}^{3}}b\cdot b_0\dd x\right).
\end{align*}\\
Since $b$ is a minimizer,

\begin{align*}
0=\frac{\dd}{\dd \tau} E[b+\tau \tilde{b}+s(\tau)b_0]\Bigg|_{\tau=0}
&=\int_{\mathbb{R}^{3}}b\cdot (\tilde{b}+\dot{s}(0) b_0 ) \dd x \\
&=\int_{\mathbb{R}^{3}}b\cdot \tilde{b}\dd x-\frac{1}{\partial_{s} j(0,0)}\left(\int_{\mathbb{R}^{3}}b\cdot b_0\dd x\right)\partial_{\tau}j(0,0), \\
&=\int_{\mathbb{R}^{3}}b\cdot \tilde{b}\dd x-\frac{\kappa}{2} \partial_{\tau}j(0,0).
\end{align*}\\
Thus   

\begin{align*}
\int_{\mathbb{R}^{3}}\left(\nabla \phi \cdot \nabla \tilde{\phi}+G\tilde{G}\right)\frac{1}{r^{2}}\dd x =\kappa \int_{\mathbb{R}^{3}}\left( \tilde{\phi} 1_{(0,\infty)}(\phi-\phi_{\infty})G+(\phi-\phi_{\infty})_{+}\tilde{G} \right)\frac{1}{r^{2}}\dd x.
\end{align*}\\
Since $\tilde{b}\in L^{2}_{\sigma,\textrm{axi}}(\mathbb{R}^{3})$ is arbitrary, $G=\kappa (\phi-\phi_{\infty})_{+}$ and (4.4) holds. The identity (4.5) follows from $G=\kappa (\phi-\phi_{\infty})_{+}$.

For $h=0$, $G=\kappa(\phi-\phi_{\infty})_{+}=0$  by (4.5). By (4.4) and (4.3$)_2$, $\phi=0$ follows. For $h\neq 0$, the identity (4.5) implies that $\int_{\mathbb{R}^{3}}(\phi-\phi_\infty)_{+}^{2}r^{-2}\dd x$ is positive.
\end{proof}

\begin{lem}
Let $B_{\infty}=-2e_z$. Let $b\in L^{2}_{\sigma,\textrm{axi}}(\mathbb{R}^{3})$ be a minimizer of (4.1). Let $\phi$ and $\kappa$ be as in Proposition 4.2. Then, $U=b+B_{\infty}$ is a nonliner force-free field (1.8) for $f=\kappa 1_{(0,\infty)}(\phi-\phi_\infty)$.
\end{lem}

\begin{proof}
By $-L\phi_{\infty}=0$ and $B_{\infty}=-2e_{z}=-\nabla \times (\phi_{\infty}\nabla \theta )$, 

\begin{align*}
U&=b+B_{\infty}
=\nabla \times ((\phi-\phi_{\infty})\nabla \theta) +G\nabla \theta,\\
\nabla \times U&=\nabla \times (G\nabla \theta)-L\phi \nabla \theta.
\end{align*}\\
By $G=\kappa(\phi-\phi_{\infty})_{+}$ and (4.3), 

\begin{align*}
\nabla \times U=\kappa 1_{(0,\infty)}(\phi-\phi_{\infty} ) U.
\end{align*}\\
Thus $U$ satisfies (1.8) for $f=\kappa1_{(0,\infty)}(\phi-\phi_{\infty})$.
\end{proof}

To investigate the properties of the minimum in (4.1) for $h\in \mathbb{R}$, we show that the minimum $I_{h}$ ($h\neq 0$) is invariant under the restriction of admissible functions $b=\nabla \times (\phi\nabla \theta)+G\nabla \theta\in L^{2}_{\sigma,\textrm{axi}}(\mathbb{R}^{3})$ to those with non-negative $\phi\geq 0$ and $G=\kappa(\phi-\phi_{\infty})_{+}$. The constant $\kappa$ is chosen so that $H[b]=h$, i.e., $h=2\kappa \int_{\mathbb{R}^{3}}(\phi-\phi_{\infty})_{+}^{2}r^{-2}\dd x$.

\begin{prop}
The function $\phi$ in Proposition 4.2 is non-negative.
\end{prop}

\begin{proof}
By (4.4) and the transform (3.10), $\varphi=\phi/r^{2}\in \dot{H}^{1}_{\textrm{axi}}(\mathbb{R}^{5})$ satisfies 

\begin{align*}
\int_{\mathbb{R}^{5}}\nabla_y \varphi \cdot \nabla_y \tilde{\varphi}\dd y=\kappa^{2}\int_{\mathbb{R}^{5}}\left(\varphi-1-\frac{\gamma}{r^{2}}\right)_{+}\tilde{\varphi}\dd y,  
\end{align*}\\
for $\tilde{\varphi}\in \dot{H}^{1}_{\textrm{axi}}(\mathbb{R}^{5})$. This equality is extendable for all $\tilde{\varphi}\in \dot{H}^{1}(\mathbb{R}^{5})$ since 

\begin{align*}
\tilde{\varphi}=\fint_{\mathbb{S}^{3}}\tilde{\varphi}\dd H+\left(\tilde{\varphi}-\fint_{\mathbb{S}^{3}}\tilde{\varphi}\dd H\right)=\tilde{\varphi}_1+\tilde{\varphi}_2\in \dot{H}^{1}_{\textrm{axi}}(\mathbb{R}^{5})\oplus \dot{H}^{1}_{\textrm{axi}}(\mathbb{R}^{5})^{\perp},
\end{align*}
and 

\begin{align*}
&\int_{\mathbb{R}^{5}}\nabla_y \varphi \cdot \nabla_y \tilde{\varphi}_2\dd y
=\int_{\mathbb{R}^{5}}\nabla_y \varphi \cdot \nabla_y \left(\fint_{\mathbb{S}^{3}} \tilde{\varphi}_2\dd H\right)\dd y=0,\\
&\int_{\mathbb{R}^{5}}\left(\varphi-1-\frac{\gamma}{r^{2}}\right) \tilde{\varphi}_2\dd y
=\int_{\mathbb{R}^{5}}\left(\varphi-1-\frac{\gamma}{r^{2}}\right) \left(\fint_{\mathbb{S}^{3}} \tilde{\varphi}_2\dd H\right)\dd y=0.
\end{align*}\\
Here, $\mathbb{S}^{3}$ is the unit sphere in $\mathbb{R}^{4}$. Thus, $\varphi$ is a weak solution of 

\begin{align*}
-\Delta_y \varphi=\kappa^{2}\left(\varphi-1-\frac{\gamma}{r^{2}}\right)_{+} \quad \textrm{in}\ \mathbb{R}^{5}.
\end{align*}\\
By Sobolev embedding, $\varphi\in L^{10/3} (\mathbb{R}^{5})$ and $(\varphi-1-\gamma/r^{2} )_{+}\in L^{10/3} (\mathbb{R}^{5})$. By differentiability of weak solutions to the Poisson equation \cite[Theorem 8.8]{GT}, $\varphi\in H^{2}_{\textrm{loc}}(\mathbb{R}^{5})$ and $-\Delta \varphi=\kappa^{2}(\varphi-1-\gamma/r^{2})_{+}\in L^{10/3}(\mathbb{R}^{5})$. By the Calderon--Zygmund inequality \cite[Corollary 9.10]{GT}, 

\begin{align*}
||\nabla^{2}\varphi ||_{L^{10/3}(\mathbb{R}^{5}) }\lesssim ||\Delta \varphi ||_{L^{10/3}(\mathbb{R}^{5}) }.
\end{align*}\\
Thus $\varphi$ is continuous and $\lim_{R\to \infty}\sup_{|x|\geq R}|\varphi(x)|=0$. Since $\varphi$ is super-harmonic, by the maximum principle \cite[9.4 THEOREM]{LL01},

\begin{align*}
\varphi(y)\geq \inf\left\{\varphi(y)\ | \ |y|=R\right\},\quad |y|\leq R. 
\end{align*}\\
By letting $R\to\infty$, $\varphi(y)=\phi(z,r)/r^{2}\geq 0$ follows.
\end{proof}

\begin{prop}
Let $h\in \mathbb{R}\backslash\{0\}$. Let $T_h$ be a set of all $b=\nabla \times (\phi\nabla \theta)+G\nabla \theta\in L^{2}_{\sigma,\textrm{axi}}(\mathbb{R}^{3})$ such that 

\begin{equation}
\phi\geq 0,\quad \int_{\mathbb{R}^{3}}(\phi-\phi_{\infty})_{+}^{2}\frac{1}{r^{2}}\dd x>0,\quad G=\kappa(\phi-\phi_{\infty})_{+},\quad \kappa=\frac{h}{2\int_{\mathbb{R}^{3}}(\phi-\phi_{\infty})_{+}^{2}\frac{1}{r^{2}}\dd x}.    
\end{equation}\\
Let $T_0=\{0\}$. Then, $H[b]=h$ for $b\in T_h$ and 

\begin{equation}
\begin{aligned}
&S_h\subset T_h\subset \left\{b\in L^{2}_{\sigma,\textrm{axi}}(\mathbb{R}^{3})\ \middle|\ H[b]=h \right\},\\
&I_h=\inf\left\{E[b]\ \middle|\ b\in T_h \right\}.
\end{aligned}
\end{equation}\\
\end{prop}

\begin{proof}
By Propositions 4.2 and 4.4, $(4.7)_1$ follows. The property $(4.7)_2$ follows from $(4.7)_1$. 
\end{proof}

\subsection{Symmetric minimizers}

To prove the existence of symmetric minimizers for $z$-variable, we recall Steiner symmetrization (symmetric decreasing rearrangement in the variable  $z$)\cite[Appendix I]{FB74}, \cite[p.293]{Friedman82}, \cite[Chapter 3]{LL01}. For a non-negative measurable function $\phi(z,r)$, the Steiner symmetrization $\phi^{*}$ satisfies 

\begin{equation}
\begin{aligned}
\phi^{*}(z,r)&\geq 0,\quad z\in \mathbb{R},\ r>0,\\
\phi^{*}(z,r)&=\phi^{*}(-z,r),\quad z\in \mathbb{R},\ r>0,\\
\phi^{*}(z_1,r)&\geq \phi^{*}(z_2,r),\quad 0\leq z_1\leq z_2,\ r>0, \\
\int_{\mathbb{R}^{2}_{+}}g(\phi(z,r),r)\dd z\dd r&=\int_{\mathbb{R}^{2}_{+}}g(\phi^{*}(z,r),r)\dd z\dd r,\\
\int_{\mathbb{R}^{2}_{+}}|\nabla \phi^{*}(z,r)|^{2}\frac{1}{r}\dd z\dd r
&\leq \int_{\mathbb{R}^{2}_{+}}|\nabla \phi(z,r)|^{2}\frac{1}{r}\dd z\dd r,
\end{aligned}
\end{equation}
for increasing functions $g(s,r)$ for $s\geq 0$ satisfying $g(0,r)=0$. The property $(4.8)_4$ follows from the layer cake representation of $g(\phi,r)$. The property $(4.8)_5$ follows from the isometry (3.11) and Riesz's rearrangement inequality for the heat kernel \cite[THEOREM 1.13, 3.7]{LL01}.

\begin{prop}
Let $\tilde{T}_h$ denote the space of all $b=\nabla \times (\phi\nabla \theta)+G\nabla \theta\in T_h$ such that $\phi(z,r)$ is non-negative, symmetric and non-increasing in the sense of $(4.8)_1-(4.8)_3$. Then, 

\begin{align}
I_h=\inf \left\{E[b]\ \middle|\ b\in \tilde{T}_h \right\}.  
\end{align}
\end{prop}

\vspace{5pt}

\begin{proof}
The left-hand side is smaller than the right-hand side by $\tilde{T}_h\subset T_h$ and $(4.7)_2$. For $b=\nabla \times (\phi\nabla \theta)+G\nabla \theta\in T_h$, we take the Steiner symmetrization $\phi^{*}$. By $(4.8)_4$, 

\begin{align*}
\int_{\mathbb{R}^{2}_{+}}(\phi-\phi_{\infty})_{+}^{2}\frac{1}{r}\dd z\dd r
=\int_{\mathbb{R}^{2}_{+}}(\phi^{*}-\phi_{\infty})_{+}^{2}\frac{1}{r}\dd z\dd r.
\end{align*}\\
With the constant $\kappa$ in (4.6), we set $\tilde{G}=\kappa(\phi^{*}-\phi_{\infty})_{+}$ and $\tilde{b}=\nabla \times (\phi^{*}\nabla \theta)+\tilde{G}\nabla \theta$. Then, $b\in \tilde{T}_{h}$ and by $(4.8)_5$, 

\begin{align*}
E[\tilde{b}]=\pi \int_{\mathbb{R}^{2}_{+}}\left(|\nabla \phi^{*}|^{2}+\kappa^{2}(\phi^{*}-\phi)_{+}^{2} \right)\frac{1}{r}\dd z\dd r
\leq \pi \int_{\mathbb{R}^{2}_{+}}\left(|\nabla \phi|^{2}+\kappa^{2}(\phi-\phi)_{+}^{2} \right)\frac{1}{r}\dd z\dd r=E[b].
\end{align*}\\
By $(4.7)_2$, 

\begin{align*}
I_h \leq \inf\left\{E[b]\ \middle|\ b\in \tilde{T}_h \right\}  
\leq \inf\left\{E[b]\ \middle|\ b\in T_h \right\}
=I_h,
\end{align*}\\
and (4.9) holds.
\end{proof}

We show the existence of symmetric minimizers to (4.9). The magnetic helicity of $b\in \tilde{T}_h$ concentrates near the origin $Q=\{{}^{t}(z,r)\in \mathbb{R}^{2}_{+}\ |\ |z|<Z,\ r<R \}$ because of the symmetry and nonincreasing properties for $z$-variable. We use the fact that the set $\{{}^{t}(z,r)\in \mathbb{R}^{2}_{+}\ |\  \phi(z,r)>\phi_{\infty}\}$ is below the graph $r=C|z|^{-1/2}$ \cite[Lemma 4.6]{FT81}, and estimate magnetic helicity in $\mathbb{R}^{2}_{+}\backslash Q$ to reduce the compactness of minimizing sequences in $\tilde{T}_h$ to that in the bounded domain $Q$.

\begin{prop}
For $\phi\in \dot{H}^{1}_{0}(\mathbb{R}^{2}_{+};r^{-1})$ satisfying $(4.8)_1-(4.8)_3$,  

\begin{align}
2r|z|^{1/2}1_{(0,\infty)}(\phi-\phi_{\infty}) \leq ||\nabla \phi||_{L^{2}(\mathbb{R}^{2}_{+};r^{-1}) }, \quad {}^{t}(z,r)\in \mathbb{R}^{2}_{+}.   
\end{align}
\end{prop}

\begin{proof}
We take a point ${}^{t}(z,r)\in \mathbb{R}^{2}_{+}$ such that $\phi(z,r)>\phi_{\infty}$. Since $\phi$ is non-increasing for $0\leq z'\leq z$ and vanishes on $\{r=0\}$,

\begin{align*}
r^{2}<\phi(z,r)\leq \phi(z',r)-\phi(z',0)=\int_{0}^{r}\partial_{r'}\phi(z',r')\dd r.
\end{align*}\\
By integrating both sides in $[0,z]$ and applying H\"older's inequality,

\begin{align*}
r^{2}z\leq \int_{0}^{z}\int_{0}^{r}\partial_{r'}\phi(z',r')\dd r'\dd z'
&\leq \left(\int_{0}^{z}\int_{0}^{r}r'\dd r'\dd z'\right)^{1/2}\left(\int_{0}^{z}\int_{0}^{r}|\partial_{r'}\phi(z',r')|^{2}\frac{1}{r'}\dd r'\dd z'\right)^{1/2} \\
&\leq \frac{r\sqrt{z}}{2}||\nabla \phi||_{L^{2}(\mathbb{R}^{2}_{+};r^{-1}) }.
\end{align*}\\
Thus (4.10) holds.
\end{proof}

\begin{prop}
Let $Z,R\geq 1$. Let $Q=\{{}^{t}(z,r)\in \mathbb{R}^{2}_{+}\ |\ |z|<Z,\ r<R \}$. There exists $C>0$ such that 

\begin{align}
\int_{\mathbb{R}^{2}_{+}\backslash Q}(\phi-\phi_{\infty})_{+}^{2}\frac{1}{r}\dd z\dd r
\leq \frac{C}{\min\{Z,A\}} || \nabla\phi||^{4}_{L^{2}(\mathbb{R}^{2}_{+};r^{-1}) } ,
\end{align}\\
for $\phi \in \dot{H}^{1}_{0}(\mathbb{R}^{2}_{+};r^{-1})$ satisfying $(4.8)_1-(4.8)_3$.
\end{prop}

\begin{proof}
We estimate 

\begin{align*}
\int_{\mathbb{R}^{2}_{+}\backslash Q}(\phi-\phi_{\infty})_{+}^{2}\frac{1}{r}\dd z\dd r
=\int_{\{r\geq R\}}(\phi-\phi_{\infty})_{+}^{2}\frac{1}{r}\dd z\dd r+2\int_{\{z\geq Z, r<R\}}(\phi-\phi_{\infty})_{+}^{2}\frac{1}{r}\dd z\dd r.
\end{align*}\\
By the weighted Sobolev inequality (3.8), 

\begin{align*}
\int_{\{r\geq R\}}(\phi-\phi_{\infty})_{+}^{2}\frac{1}{r}\dd z\dd r
\leq \int_{\{r\geq R\}}(\phi-\phi_{\infty})_{+}^{2}\frac{\phi^{2}}{r^{5}}\dd z\dd r
\leq \frac{1}{R}\int_{\mathbb{R}^{2}_{+}}\phi^{4}\frac{1}{r^{4}}\dd z\dd r 
\lesssim \frac{1}{R}||\nabla \phi||^{4}_{L^{2}(\mathbb{R}^{2}_{+};r^{-1} ) }.
\end{align*}\\
By the pointwise estimate (4.10) and (3.8),

\begin{align*}
\int_{\{z>Z, r<R\}}(\phi-\phi_{\infty})_{+}^{2}\frac{1}{r}\dd z\dd r
&\lesssim ||\nabla \phi||^{2}_{L^{2}(\mathbb{R}^{2}_{+};r^{-1} ) }\int_{\{z>Z, r<R\}}(\phi-\phi_{\infty})_{+}^{2}\frac{1}{zr^{3}}\dd r \dd z \\
&\leq \frac{1}{Z}||\nabla \phi||^{2}_{L^{2}(\mathbb{R}^{2}_{+};r^{-1} ) }\int_{\mathbb{R}^{2}_{+}}\phi^{2}\frac{1}{r^{3}}\dd z\dd r \\
& \lesssim \frac{1}{Z}||\nabla \phi||^{4}_{L^{2}(\mathbb{R}^{2}_{+};r^{-1} ) }.
\end{align*}\\
Thus (4.11) holds.
\end{proof}

\begin{thm}[Existence of symmetric minimizers]
Let $h\in \mathbb{R}$. Let $\{b_n\}\subset  \tilde{T}_h$ be a sequence such that $E[b_n]\to I_h$. Then, there exist $\{n_k\}$ and $b\in \tilde{T}_{h}\cap S_h$ such that $b_{n_{k}}\to b$ in $L^{2}(\mathbb{R}^{3})$.  
\end{thm}

\begin{proof}
By Rellich--Kondrakov theorem in the weighted space (3.9), there exists a subsequence of $b_n=\nabla \times (\phi_n\nabla \theta)+G_n\nabla \theta$ (still denoted by $\{b_n\}$) and $b=\nabla\times (\phi\nabla \theta)+G\nabla \theta$ such that 

\begin{align*}
b_n &\rightharpoonup  b \quad \textrm{in}\ L^{2}(\mathbb{R}^{3}),\\
\phi_n &\to  \phi \quad \textrm{in}\ L^{2}_{\textrm{loc}}(\overline{\mathbb{R}^{2}_{+}};r^{-1}).
\end{align*}\\
For $Z,R\geq 1$ and $Q=\{ |z|<Z, r<R \}$, 

\begin{align*}
H[b_n]
&=4\pi \int_{\mathbb{R}^{2}_{+}}(\phi_n-\phi_\infty)_{+}G_n\frac{1}{r}\dd z\dd r\\
&=4\pi \int_{Q}(\phi_n-\phi_\infty)_{+}G_n\frac{1}{r}\dd z\dd r
+4\pi \int_{\mathbb{R}^{2}_{+}\backslash Q}(\phi_n-\phi_\infty)_{+}G_n\frac{1}{r}\dd z\dd r.
\end{align*}\\
By H\"older's inequality, (4.11) and (3.2), 

\begin{align*}
\left|\int_{\mathbb{R}^{2}_{+}\backslash Q}(\phi_n-\phi_\infty)_{+}G_n\frac{1}{r}\dd z\dd r \right| 
&\leq \left|\int_{\mathbb{R}^{2}_{+}\backslash Q}(\phi_n-\phi_\infty)_{+}^{2}\frac{1}{r}\dd z\dd r \right|^{1/2}||G_n||_{L^{2}(\mathbb{R}\backslash Q;r^{-1} ) } \\
&\leq \frac{C}{\min\{Z,R\}^{1/2}}\sup_{n}||b_n||^{3}_{L^{2}(\mathbb{R}^{3}) }.  
\end{align*}\\
By $|\tau_+-s_{+}|\leq |\tau-s|$, H\"older's inequality and (3.2), 

\begin{align*}
\left|\int_{Q}(\phi_n-\phi_\infty)_{+}G_n\frac{1}{r}\dd z\dd r-\int_{Q}(\phi-\phi_\infty)_{+}G_n\frac{1}{r}\dd z\dd r \right|\leq ||\phi_n-\phi||_{L^{2}(Q;r^{-1})}\left(\sup_{n}||b_n||_{L^{2}(\mathbb{R}^{3}) }\right). 
\end{align*}\\
Thus 

\begin{align*}
|H[b_n]-H[b]|
&\leq 4\pi ||\phi_n-\phi||_{L^{2}(Q;r^{-1})}\left(\sup_{n}||b_n||_{L^{2}(\mathbb{R}^{3}) }\right) \\
&+4\pi \left|\int_{Q}(\phi-\phi_\infty)_{+}G_n\frac{1}{r}\dd z\dd r-\int_{Q}(\phi-\phi_\infty)_{+}G\frac{1}{r}\dd z\dd r \right| \\
&+ \frac{C}{\min\{Z,R\}^{1/2}}\left( \sup_{n}||b_n||^{3}_{L^{2}(\mathbb{R}^{3}) }+||b||^{3}_{L^{2}(\mathbb{R}^{3}) } \right).
\end{align*}\\
Since $G_n$ weakly converges in $L^{2}(Q;r^{-1})$ and $H[b_n]=h$, letting $n\to\infty$ and then $Z\to\infty$, $R\to\infty$ imply $h=H[b]$. By 

\begin{align*}
I_h\leq E[b]\leq \liminf_{n\to\infty}E[b_n]=I_h,
\end{align*}\\
the limit $b\in L^{2}_{\sigma,\textrm{axi}}(\mathbb{R}^{3})$ is a minimizer of $I_h$. The convergence $\lim_{n\to\infty}||b_n||_{L^{2}(\mathbb{R}^{3})}=||b||_{L^{2}(\mathbb{R}^{3})}$ implies that $b_n\to b$ in $L^{2}(\mathbb{R}^{3})$. The limit $b$ belongs to $\tilde{T}_{h}$ since $\phi$ satisfies $(4.8)_1$-- $(4.8)_3$ and $b\in S_h\subset T_h$. 
\end{proof}

\subsection{Properties of minimum}

We derive properties of the minimum $I_h$ for $h\in \mathbb{R}$ from (4.9) and the existence of symmetric minimizers $b\in \tilde{T}_{h}$ in Theorem 4.9.

\begin{lem}
\begin{align}
I_h&=I_{-h}\geq I_{0}=0,\quad h\in \mathbb{R},\qquad (\textrm{symmetry}), \\
0&<I_{h_1}<I_{h_2},\quad 0<h_1<h_2,   \hspace{12pt} (\textrm{monotonicity}),\\
I_{\theta h}&<\theta I_{h},\quad \theta>1, h>0,  \\
I_{h_1+h_2}&<I_{h_1}+I_{h_2},\quad h_1,h_2>0, \hspace{22pt} (\textrm{strict subadditivity}), \\
I_{h_1}&= \lim_{h\to h_1}I_h,\quad h_1>  0,\\
I_{h_1}&\leq \liminf_{h\to h_1}I_h,\quad h_1\in \mathbb{R} \hspace{37pt} (\textrm{lower semi-continuity}),
\end{align}
\end{lem}

\begin{proof}
The symmetry (4.12) follows from (3.26). By Theorem 4.9, for $h>0$, there exists $b=\nabla \times (\phi\nabla \theta)+G\nabla \theta\in S_h$ such that $I_h=E[b]$ and $h=H[b]$. If $I_h=0$, $E[b]=0$ and $h=H[b]=0$. Thus $I_h>0$. For $\tau>1$, 

\begin{align*}
\tilde{b}=\nabla \times (\phi \nabla \theta)+\frac{1}{\tau}G\nabla \theta, 
\end{align*}\\
satisfies $H[\tilde{b}]=H[b]/\tau=h/\tau$ and 

\begin{align*}
I_{h/\tau}\leq E[\tilde{b}]=\frac{1}{2}\int_{\mathbb{R}^{3}}|\nabla \phi|^{2}\frac{1}{r^{2}}\dd x+\frac{1}{2\tau^{2}}\int_{\mathbb{R}^{3}}|G|^{2}\frac{1}{r^{2}}\dd x
=E[b]-\frac{1}{2}\left(1-\frac{1}{\tau^{2}}\right)\int_{\mathbb{R}^{3}}|G|^{2}\frac{1}{r^{2}}\dd x<I_h.
\end{align*}\\
Thus, the monotonicity (4.13) holds. 

By Proposition 4.2, for some $\kappa\in \mathbb{R}$,  

\begin{align*}
H[b]=2\kappa\int_{\mathbb{R}^{3}}(\phi-\phi_{\infty})_{+}^{2}\frac{1}{r^{2}}\dd x.
\end{align*}\\
By $h>0$, $\kappa>0$ and $G=\kappa (\phi-\phi_{\infty})_{+}\geq 0$. For $t>1$, $H[tb]$ is continuous and 

\begin{align*}
H[tb]=2\int_{\mathbb{R}^{3}} (t\phi-\phi_{\infty})_{+}(tG)\frac{1}{r^{2}}\dd x
>2t^{2} \int_{\mathbb{R}^{3}}(\phi-\phi_{\infty})_{+}G\frac{1}{r^{2}}\dd x=t^{2}H[b]=t^{2}h.
\end{align*}\\
For $\theta>1$, we take $t_1>1$ such that $H[t_1b]>t_1^{2}h>\theta h$. By the intermediate value theorem, there exists $1<t=t(\theta)<t_1$ such that $\theta h=H[t(\theta)b]$ and $\theta>t(\theta)^{2}$. Thus 

\begin{align*}
I_{\theta h}\leq E[t(\theta)b]=t(\theta)^{2}E[b]=t(\theta)^{2}I_{h}<\theta I_h,
\end{align*}\\
and (4.14) holds. 

We take $h_1,h_2>0$. We may assume that $h_1<h_2$. By (4.14), $I_{h_2}<(h_2/h_1)I_{h_1}$. For $\vartheta=(h_1+h_2)/h_2$, 

\begin{align*}
I_{h_1+h_2}=I_{\vartheta h_2}<\vartheta I_{h_2}=\left(\frac{h_1}{h_2}+1  \right)I_{h_2}<I_{h_1}+I_{h_2}.
\end{align*}\\
Thus, the strict subadditivity (4.15) holds. 

By (4.13) and (4.14), for $\varepsilon>0$, 

\begin{align*}
&I_{h}<I_{h+\varepsilon},\quad
I_{h}<\frac{h}{h-\varepsilon}I_{h-\varepsilon},\\
&I_{h-\varepsilon}<I_{h},\quad
I_{h+\varepsilon}<\frac{h+\varepsilon}{h}I_{h}.
\end{align*}\\
Letting $\varepsilon\to 0$ implies the continuity (4.16) for $h_1>0$. By (4.12) and (4.13), $I_h$ is lower semi-continuous at $h_1=0$ and (4.17) holds. 
\end{proof}

\section{Minimizing sequences}

We demonstrate the compactness of (non-symmetric) minimizing sequences to the variational problem (4.1) in $L^{2}_{\sigma,\textrm{axi}}(\mathbb{R}^{3})$ up to translation in $z$. We apply the concentration--compactness principle in $\mathbb{R}^{2}_{+}$ \cite{Lions84a}, \cite{Lions84b} and exclude possibilities of dichotomy and vanishing of minimizing sequences to obtain the desired compactness. The crucial part of the proof is the exclusion of dichotomy by application of the strict subadditivity, monotonicity, and lower semi-continuity of the minimum shown in Lemma 4.10.

\subsection{Concentration--compactness lemma}

We derive a concentration--compactness lemma for sequences in $L^{2}_{\sigma,\textrm{axi}}(\mathbb{R}^{3})$ from the concentration--compactness lemma in $L^{1}(\mathbb{R}^{2}_{+})$ \cite{Lions84a}, \cite{Lions84b}. We then modify the derived lemma in the case of dichotomy so that two sequences have disjoint supports by a cut-off function argument.

\begin{prop}
Let $\{\rho_n\} \subset  L^{1}(\mathbb{R}^{2}_{+})$ be a sequence such that $\rho_n\geq 0$ and 

\begin{align*}
\int_{\mathbb{R}^{2}_{+}}\rho_n\dd z\dd r \to  l>0\quad \textrm{as}\ n\to\infty.     
\end{align*}\\
Then, there exists a subsequence $\{\rho_{n_k}\}$ such that one of the following holds:

\noindent
(i) Compactness: there exists $z_k \in \mathbb{R}$ such that for $\varepsilon>0$ there exists $R_{\varepsilon}>0$ such that for $D(z_k,R_\varepsilon)=\{{}^{t}(z,r)\in \mathbb{R}^{2}_{+}\ |\ |z-z_k|^{2}+r^{2}<R^{2}_\varepsilon \}$, 

\begin{align*}
\liminf_{k\to\infty}\int_{D(z_k,R_\varepsilon)} \rho_{n_{k}}\dd z\dd r \geq l-\varepsilon. 
\end{align*}\\
(ii) Vanishing: for each $R>0$,

\begin{align*}
\lim_{k\to\infty}\sup_{z_0\in \mathbb{R}}\int_{D(z_0,R) }\rho_{n_{k}}\dd z\dd r=0.  
\end{align*}\\
(iii) Dichotomy: there exists $\alpha \in (0,l)$ such that for $\varepsilon>0$ there exist $z_k \in \mathbb{R}$ and $ R_k\geq R_0$ such that $R_k\to\infty$ and, 

\begin{equation*}
\begin{aligned}
\limsup_{k\to\infty}\left\{ \left| \int_{D(z_k,R_0)}\rho_{n_k}\dd z \dd r -\alpha \right|
+\left|\int_{\mathbb{R}^{2}_{+}\backslash D(z_k,R_k)} \rho_{k}\dd z \dd r -(l-\alpha)\right|+\int_{D(z_k,R_k) \backslash D(z_k,R_0)} \rho_{k}\dd z \dd r\right\} 
\leq \varepsilon.  
\end{aligned}
\end{equation*}
\end{prop}

\begin{proof}
For the case of fixed mass, $l_n=l$ for $l_n=\int_{\mathbb{R}^{2}_{+}}\rho_n\dd z\dd r$, we apply a similar argument as $\mathbb{R}^{2}$ for the L\'evy's (partial) concentration function $Q_n(t)=\sup_{z\in \mathbb{R}}\int_{D(z,t)}\rho_n \dd z\dd r$ and conclude \cite[Lemma I.1]{Lions84a}, \cite[p.279]{Lions84b}. The case of varying mass $l_n\to l$ is reduced to the case of fixed mass by the normalization $\tilde{\rho}_n=\rho_n l/l_n$.
\end{proof}

\begin{lem}
Let $\{b_n\}\subset L^{2}_{\sigma,\textrm{axi}}(\mathbb{R}^{3})$ be a sequence such that 

\begin{align}
\int_{\mathbb{R}^{3}}|b_n|^{2}\dd x \to  l>0\quad \textrm{as}\ n\to\infty.     
\end{align}\\
There exists a subsequence $\{b_{n_k}\}$ such that one of the following holds: 

\noindent
(i) There exists $z_k \in \mathbb{R}$ such that for $\varepsilon>0$ there exists $R_{\varepsilon}>0$ such that for $B(z_ke_z,R_\varepsilon)=\{x\in \mathbb{R}^{3}\ |\ |x-z_ke_z|<R_\varepsilon \}$, 

\begin{align}
\liminf_{k\to\infty}\int_{B(z_ke_z,R_\varepsilon)} |b_{n_{k}  }|^{2}\dd x \geq l-\varepsilon. 
\end{align}\\
(ii) For each $R>0$,

\begin{align}
\lim_{k\to\infty}\sup_{z_0\in \mathbb{R}}\int_{B(z_0e_z,R) }|b_{n_{k}}|^{2} \dd x=0.  
\end{align}\\
(iii) There exists $\alpha \in (0,l)$ such that for $\varepsilon>0$, there exists $z_k\in \mathbb{R}$  and $R_k\geq R_0$ such that $R_k\to\infty$ and, 

\begin{equation}
\begin{aligned}
\limsup_{k\to\infty}\Bigg\{ &\left| ||b_{n_k}||_{L^{2}(B(z_ke_z,R_0)) }^{2}  -\alpha \right|
+\left| ||b_{n_k}||_{L^{2}(\mathbb{R}^{3}\backslash B(z_ke_z,R_k)) }^{2} -(l-\alpha)\right| \\
&+||b_{n_k}||_{L^{2}( B(z_ke_z,R_k)\backslash B(z_ke_z,R_0))}^{2} \Bigg\} 
\leq \varepsilon.  
\end{aligned}
\end{equation}
\end{lem}

\begin{proof}
We apply Proposition 5.1 for $b_n=\nabla \times (\phi_n\nabla \theta)+G_n\nabla \theta\in L^{2}_{\sigma,\textrm{axi}}(\mathbb{R}^{3})$ and 

\begin{align*}
\rho_n=\left(|\nabla \phi_n|^{2}+|G_n|^{2}\right)\frac{2\pi}{r}. 
\end{align*}
\end{proof}

\begin{prop}
In the case of the dichotomy (iii) in Lemma 5.2, there exists $b_{i,n_k}=\nabla \times (\phi_{i,n_k}\nabla \theta)+G_{i,n_k}\nabla \theta \in L^{2}_{\sigma,\textrm{axi}}(\mathbb{R}^{3})$, $i=1,2$, such that 

\begin{equation}
\begin{aligned}
&\textrm{spt}\ \phi_{1,n_k},\textrm{spt}\ G_{1,n_k}\subset D(z_k,2R_0),   \\
&\textrm{spt}\ \phi_{2,n_k},\textrm{spt}\ G_{2,n_k}\subset \mathbb{R}^{2}_{+}\backslash \overline{D(z_k,R_k/2)},   \\
&\limsup_{k\to\infty}\left\{ \left| \|b_{1,n_k}\|^{2}_{L^{2}(\mathbb{R}^{3}) } -\alpha \right|
+\left| \|b_{2,n_k}\|^{2}_{L^{2}(\mathbb{R}^{3}) } -(l-\alpha) \right|
+\|b_n-b_{1,n_k}-b_{2,n_k}\|^{2}_{L^{2}(\mathbb{R}^{3}) }   \right\} \\
&\leq C\varepsilon.  
\end{aligned}
\end{equation}
\end{prop}

\begin{proof}
We construct $b_{i,n_k}\in L^{2}_{\sigma,\textrm{axi}}(\mathbb{R}^{3}) $ satisfying (5.5) by a cut-off function argument. We may assume $z_k=0$ by translation. For simplicity of notation, we denote $b_{n_k}$ by $b_n$. We take a function $\chi\in C^{\infty}_{c}[0,\infty)$ such that $\chi= 1$ in $[0,1]$ and $\chi= 0$ in $[2,\infty)$ and set $\chi_{R_0}(z,r)=\chi(R_0^{-1}\sqrt{|z|^{2}+|r|^{2}})$ so that $\chi_{R_0}\in C^{\infty}_c(\mathbb{R}^{2})$ satisfies $\chi_{R_0}= 1$ in $D(0,R_0)$ and $\chi_{R_0}= 0$ in $\mathbb{R}^{2}_{+}\backslash \overline{D(0,2R_0)}$. We set $b_{1,n}=\nabla \times (\phi_{1,n}\nabla \theta)+G_{1,n}\nabla \theta$ by

\begin{align*}
\phi_{1,n}=\phi_{n}\chi_{R_0},\quad G_{1,n}=G_{n}\chi_{R_0},
\end{align*}\\
so that $(5.5)_1$ holds. By the Poincar\'e inequality (3.13), 

\begin{align}
\int_{D(0,2R_0)\backslash D(0,R_0)}|\phi_{n}|^{2}\frac{1}{r}\dd z\dd r\leq CR_0^{2} \int_{D(0,2R_0)\backslash D(0,R_0)}|\nabla \phi_n|^{2}\frac{1}{r}\dd z\dd r,
\end{align}\\
and we estimate 

\begin{align*}
\int_{\mathbb{R}^{2}_{+}}|\nabla \phi_{1,n}|^{2}\frac{1}{r}\dd z\dd r-\int_{D(0,R_0)}|\nabla \phi_{n}|^{2}\frac{1}{r}\dd z\dd r
&=\int_{D(0,2R_0)\backslash D(0,R_0)}|\nabla \phi_{n}\chi_{R_0}+\phi_{n}\nabla \chi_{R_0} |^{2}\frac{1}{r}\dd z\dd r \\
&\leq C\int_{D(0,2R_0)\backslash D(0,R_0)}|\nabla \phi_{n}|^{2}\frac{1}{r}\dd z\dd r.
\end{align*}\\
For $\rho_{1,n}=\left(|\nabla \phi_{1,n}|^{2}+|G_{1,n}|^{2}\right)2\pi r^{-1}$ and $\rho_{n}=\left(|\nabla \phi_{n}|^{2}+|G_{n}|^{2}\right)2\pi r^{-1}$,

\begin{align*}
\int_{\mathbb{R}^{2}_{+}}\rho_{1,n}\dd z\dd r
-\int_{D(0,R_0)}\rho_{n}\dd z\dd r
\leq C\int_{D(0,2R_0)\backslash D(0,R_0)}\rho_{n}\dd z\dd r.
\end{align*}\\
In terms of $b_{1,n}$ and $b_{n}$, 

\begin{align*}
||b_{1,n}||^{2}_{L^{2}( \mathbb{R}^{3}) }-||b_{n}||^{2}_{L^{2}( B(0,R_0)) }\leq C ||b_{n}||^{2}_{L^{2}( B(0,2R_0)\backslash B(0,R_0)) }.
\end{align*}\\
By applying (5.4) to 

\begin{align*}
\left| \|b_{1,n}\|^{2}_{L^{2}(\mathbb{R}^{3}) }   -\alpha\right|
&\leq  \left|\|b_{1,n}\|^{2}_{L^{2}(B(0,2R_0)) }
- ||b_{n}||^{2}_{L^{2}( B(0,R_0)) } \right|+\left| ||b_{n}||^{2}_{L^{2}( B(0,R_0)) }-\alpha\right|  \\
&\leq C ||b_{n}||^{2}_{L^{2}( B(0,2R_0)\backslash B(0,R_0)) }+\left| ||b_{n}||^{2}_{L^{2}( B(0,R_0)) }-\alpha\right|, 
\end{align*}\\
and we have 

\begin{align*}
\limsup_{n\to\infty}\left| \|b_{1,n}\|^{2}_{L^{2}(\mathbb{R}^{3}) }   -\alpha\right|\leq C\varepsilon.
\end{align*}\\
Similarly, we set $b_{2,n}=\nabla \times (\phi_{2,n}\nabla \theta)+G_{2,n}\nabla \theta$ by 

\begin{align*}
\phi_{2,n}=\phi_{n}(1-\chi_{R_n/2}),\quad G_{2,n}=G_{n}(1-\chi_{R_n/2}), 
\end{align*}\\
so that $(5.5)_2$ holds. By (3.13),

\begin{align}
\int_{D(0,R_n)\backslash D(0,R_n/2)}|\phi_n|^{2}\frac{1}{r}\dd z\dd r\leq CR_n^{2} \int_{D(0,R_n)\backslash D(0,R_n/2)}|\nabla \phi_n|^{2}\frac{1}{r}\dd z\dd r.
\end{align}\\
Similarly as we estimated $b_{1,n}$, by (5.4) and

\begin{align*}
\left| \|b_{2,n}\|^{2}_{L^{2}(\mathbb{R}^{3}) }   -(l-\alpha)\right|
&\leq  \left|\|b_{2,n}\|^{2}_{L^{2}(\mathbb{R}^{3}\backslash B(0,R_n/2)) }
- ||b_{n}||^{2}_{L^{2}(\mathbb{R}^{3}\backslash B(0,R_n)) } \right|+\left| ||b_{n}||^{2}_{L^{2}(\mathbb{R}^{3}\backslash B(0,R_n)) }-(l-\alpha)\right|  \\
&\leq C ||b_{n}||^{2}_{L^{2}( B(0,R_n)\backslash B(0,R_n/2)) }+\left| ||b_{n}||^{2}_{L^{2}(\mathbb{R}^{3}\backslash B(0,R_n)) }-(l-\alpha)\right|,
\end{align*}\\
we obtain 

\begin{align*}
\limsup_{n\to\infty}\left| \|b_{2,n}\|^{2}_{L^{2}(\mathbb{R}^{3}) }   -(l-\alpha)\right|\leq C\varepsilon.
\end{align*}\\
By using (5.6) and (5.7) for 

\begin{align*}
b_n-b_{1,n}-b_{2,n}=\nabla \times (\phi_n(\chi_{R_n/2}-\chi_{R_0})\nabla \theta )+G_n(\chi_{R_n/2}-\chi_{R_0})\nabla \theta, 
\end{align*}\\
we estimate 

\begin{align*}
||b_n-b_{1,n}-b_{2,n}||_{L^{2}(\mathbb{R}^{3}) }^{2}
\leq C\int_{D(0,R_n)\backslash D(0,R_0) }(|\nabla \phi_n |^{2}+|G_n|^{2})\frac{1}{r}\dd z\dd r=C||b_n||_{L^{2}(B(0,R_n)\backslash B(0,R_0)) }^{2}.
\end{align*}\\
By (5.4), 

\begin{align*}
\limsup_{n\to\infty} \left\{
\left| \|b_{1,n}\|^{2}_{L^{2}(\mathbb{R}^{3}) }   -\alpha\right|
+\left| \|b_{2,n}\|^{2}_{L^{2}(\mathbb{R}^{3}) }   -(l-\alpha)\right|
+\|b_n-b_{1,n}-b_{2,n}\|^{2}_{L^{2}(\mathbb{R}^{3}) }   \right\} \leq C\varepsilon.
\end{align*}\\
Thus $(5.5)_3$ holds. 
\end{proof}

\subsection{Lipschitz estimates}

Proposition 5.3 states that $||b_{n_k}||_{L^{2}(\mathbb{R}^{3})}^{2}\approx l$ can be divided into two parts 

\begin{align*}
||b_{1,n_k}||_{L^{2}(\mathbb{R}^{3})}^{2}\approx \alpha,\quad ||b_{2,n_k}||_{L^{2}(\mathbb{R}^{3})}^{2}\approx l-\alpha.
\end{align*}\\
We further show that $H[b_{n_k}]$ can be divided into two parts $H[b_{1,n_k}]$ and $H[b_{2,n_k}]$. Recall that $H[\cdot ]: L^{2}_{\sigma,\textrm{axi}}(\mathbb{R}^{3})\to \mathbb{R}$ is locally bounded by the Arnold-type inequality (3.20), 

\begin{align*}
|H[b]|\lesssim ||b||_{L^{2}(\mathbb{R}^{3})}^{8/3}.
\end{align*}\\
We extend this estimate to the following Lipschitz estimate and derive the desired decomposition property for magnetic helicity.

\begin{prop}
\begin{align}
\left|H[b_1]
-H[b_2]
\right|&\lesssim  \left(\max_{i=1,2} ||b_i||_{L^{2}(\mathbb{R}^{3}) } ^{5/3}\right) ||b_1-b_2||_{L^{2}(\mathbb{R}^{3}) },  \quad \textrm{for}\ b_1,b_2\in L^{2}_{\sigma,\textrm{axi}}(\mathbb{R}^{3}),
\end{align}

\begin{align}
\begin{aligned}
\left| H[b_0]-H[b_1]-H[b_2]\right| 
&\lesssim  \left(\max_{0\leq i\leq 2}||b_i||_{L^{2}(\mathbb{R}^{3})}^{5/3}\right) ||b_0-b_1-b_2||_{L^{2}(\mathbb{R}^{3}) } ,  
\end{aligned}
\end{align}\\
for $b_i=\nabla \times (\phi_i\nabla \theta)+G_i\nabla \theta\in L^{2}_{\sigma,\textrm{axi}}(\mathbb{R}^{3})$, $i=0, 1,2$, such that $\textrm{spt}\ \phi_i\cap \textrm{spt}\ G_j=\emptyset$, $i\neq j$, $i,j=1,2$. 
\end{prop}

\begin{proof}
By the monotonicity of the indicator function, 

\begin{equation}
\begin{aligned}
|(\phi_1-\phi_{\infty})_{+}-(\phi_2-\phi_{\infty})_{+}|
&=\left|\int_{0}^{1}\frac{\dd}{\dd \tau}(\tau \phi_1+(1-\tau)\phi_2-\phi_{\infty})_{+}\dd \tau\right|  \\
&=\left|\int_{0}^{1}(\phi_1-\phi_2)1_{(0,\infty)}(\tau \phi_1+(1-\tau)\phi_2-\phi_{\infty})\dd \tau\right| \\
&\leq |\phi_1-\phi_2| 1_{(0,\infty)}(|\phi_1|+|\phi_2|-\phi_{\infty}).
\end{aligned}
\end{equation}\\
By H\"older's inequality, (5.10), (3.2) and (3.18), 

\begin{align*}
|H[b_1]-H[b_2]|
&\leq 2\int_{\mathbb{R}^{3}} |((\phi_1-\phi_{\infty})_{+}-(\phi_2-\phi_{\infty})_{+})G_1|\frac{1}{r^{2}}\dd x+\int_{\mathbb{R}^{3}} (\phi_2-\phi_{\infty})_{+}|G_1-G_2|\frac{1}{r^{2}}\dd x \\
&\lesssim ||(\phi_1-\phi_{\infty})_{+}-(\phi_2-\phi_{\infty})_{+}||_{L^{2}(\mathbb{R}^{2}_{+};r^{-1}) }  ||G_1||_{L^{2}(\mathbb{R}^{2}_{+};r^{-1}) } \\
&+||(\phi_2-\phi_{\infty})_{+}||_{L^{2}(\mathbb{R}^{2}_{+};r^{-1}) }\  ||G_1-G_2||_{L^{2}(\mathbb{R}^{2}_{+};r^{-1}) } \\
&\lesssim ||(\phi_1-\phi_2) 1_{(0,\infty)}(|\phi_1|+|\phi_2|-\phi_{\infty})||_{L^{2}(\mathbb{R}^{2}_{+};r^{-1}) }\  ||b_1||_{L^{2}(\mathbb{R}^{3}) } \\
&+ ||b_2||^{5/3}_{L^{2}(\mathbb{R}^{3}) }\ ||b_1-b_2||_{L^{2}(\mathbb{R}^{3}) }.
\end{align*}\\
By H\"older's inequality, (3.8) and (3.2),

\begin{align*}
&||(\phi_1-\phi_2) 1_{(0,\infty)}(|\phi_1|+|\phi_2|-\phi_{\infty})||_{L^{2}(\mathbb{R}^{2}_{+};r^{-1}) }^{2} \\
&\leq \int_{\mathbb{R}^{2}_{+}} |\phi_1-\phi_2|^{2}\left(\frac{|\phi_1|+|\phi_2|}{\phi_{\infty}}\right)^{4/3}\frac{1}{r}\dd z\dd r \\
& \leq \int_{\mathbb{R}^{2}_{+}} |\phi_1-\phi_2|^{2}\frac{1}{r^{11/5}} (|\phi_1|+|\phi_2|)^{4/3}\frac{1}{r^{22/15}}\dd z\dd r \\
&\leq \left(\int_{\mathbb{R}^{2}_{+}} |\phi_1-\phi_2|^{10/3}\frac{1}{r^{2+5/3}}\dd z\dd r \right)^{3/5}
\left(\int_{\mathbb{R}^{2}_{+}}(|\phi_1|+|\phi_2|)^{10/3}\frac{1}{r^{2+5/3}}\dd z\dd r \right)^{2/5} \\
&\lesssim  ||b_1-b_2||^{2}_{L^{2}(\mathbb{R}^{3}) } \left(\max_{i=1,2} ||b_i||_{L^{2}(\mathbb{R}^{3})} \right)^{4/3}.
\end{align*}\\
Thus (5.8) holds.

In a similar way as (5.10),

\begin{align*}
|(\phi_1+\phi_2-\phi_{\infty})_{+}-(\phi_1-\phi_{\infty})_{+}|
\leq |\phi_2| 1_{(0,\infty)}(|\phi_1+\phi_2|+|\phi_1|-\phi_{\infty}).
\end{align*}\\
Since $\textrm{spt}\ \phi_2\cap \textrm{spt}\ G_1=\emptyset $, 

\begin{align*}
&\left|\int_{\mathbb{R}^{3}}( \phi_1+\phi_2-\phi_\infty)_{+}G_1\frac{1}{r^{2}}\dd x
-\int_{\mathbb{R}^{3}}( \phi_1-\phi_\infty)_{+}G_1\frac{1}{r^{2}}\dd x\right|\\
&\leq \int_{\mathbb{R}^{3}}|\phi_2| 1_{(0,\infty)}(|\phi_1+\phi_2|+|\phi_1|-\phi_{\infty})|G_1|\frac{1}{r^{2}}\dd x=0.
\end{align*}\\
Thus 

\begin{align*}
\int_{\mathbb{R}^{3}}( \phi_1+\phi_2-\phi_\infty)_{+}G_1\frac{1}{r^{2}}\dd x
=\int_{\mathbb{R}^{3}}( \phi_1-\phi_\infty)_{+}G_1\frac{1}{r^{2}}\dd x.
\end{align*}\\
In a similar way, by $\textrm{spt}\ \phi_1\cap \textrm{spt}\ G_2=\emptyset $, 

\begin{align*}
\int_{\mathbb{R}^{3}}( \phi_1+\phi_2-\phi_\infty)_{+}G_2\frac{1}{r^{2}}\dd x
=\int_{\mathbb{R}^{3}}( \phi_2-\phi_\infty)_{+}G_2\frac{1}{r^{2}}\dd x.
\end{align*}\\
Thus $H[b_1+b_2]=H[b_1]+H[b_2]$. By applying (5.8) to $H[b_0]-H[b_1]-H[b_2]=H[b_0]-H[b_1+b_2]$, (5.9) follows.
\end{proof}

\subsection{Compactness}

We now demonstrate the compactness of minimizing sequences for the variational problem (4.1) using Lemma 4.10.

\begin{thm}
Let $h\in \mathbb{R}$. Let $\{b_n\}\subset L^{2}_{\sigma,\textrm{axi}}(\mathbb{R}^{3})$ be a sequence such that $E[b_n]\to I_h$ and $H[b_n]\to h$. There exists $\{n_k\}\subset \mathbb{N}$, $\{z_k\}\subset \mathbb{R}$ and $b \in S_h$ such that $b_{n_k}(\cdot+z_{k}e_z)\to b$ in $L^{2}(\mathbb{R}^{3})$. 
\end{thm}

\begin{proof}
For $h=0$, the assertion holds for $b=0$ since $I_0=0$. We may assume that $h>0$ by the symmetry (4.12). For a sequence $\{b_n\}$ satisfying $E[b_n]\to I_h$ and $H[b_n]\to h$, we apply Lemma 5.2 for $l=2I_h$. Then, by choosing a subsequence (still denoted by $\{b_n\}$), one of the 3 cases should occur: compactness, vanishing, dichotomy.\\

\noindent
\textit{Case 1}: dichotomy. \\

By Proposition 5.3, there exist $b_{1,n}, b_{2,n}\in L^{2}_{\sigma,\textrm{axi}}(\mathbb{R}^{3})$ such that (5.5) holds. By $(5.5)_3$ and (3.20), $E[b_{i,n}]$ and $H[b_{i,n}]$ are uniformly bounded for $n$ and $\varepsilon$. By choosing a subsequence, we may assume that $h_{i,n}=H[b_{i,n}]\to \bar{h}_i$ for some $\bar{h}_i\in \mathbb{R}$ as $n\to\infty$ and $\varepsilon\to0$. By $(5.5)_1$ and $(5.5)_2$, we apply Proposition 5.4 to estimate

\begin{align*}
&|H[b_n]-H[b_{1,n}]-H[b_{2,n}]|\\
&\lesssim \max\left\{||b_{1,n}||_{L^{2}(\mathbb{R}^{3})},||b_{2,n}||_{L^{2}(\mathbb{R}^{3})},||b_{n}||_{L^{2}(\mathbb{R}^{3})}\right\}^{5/3}||b_{n}-b_{1,n}-b_{2,n}||_{L^{2}(\mathbb{R}^{3})}.
\end{align*}\\
By $(5.5)_3$, 

\begin{align*}
\limsup_{n\to \infty}|H[b_n]-H[b_{1,n}]-H[b_{2,n}]| \leq C\varepsilon^{1/2}. 
\end{align*}\\
Letting $\varepsilon\to 0$ implies that $h=\bar{h}_1+\bar{h}_2>0$. Since $\textrm{spt}\ b_{1,n}\cap \textrm{spt}\ b_{2,n}=\emptyset$,  

\begin{align*}
E[b_n]
&=E[b_{1,n}]+E[b_{2,n}]+E[b_{n}-b_{1,n}-b_{2,n}]+\int_{\mathbb{R}^{3}} (b_{1,n}+b_{2,n})\cdot (b_{n}-b_{1,n}-b_{2,n} )\dd x \\
&\geq E[b_{1,n}]+E[b_{2,n}]-\left(\sup_{i,n} ||b_{i,n}||_{L^{2}(\mathbb{R}^{3}) }\right)  ||b_{n}-b_{1,n}-b_{2,n} ||_{L^{2}(\mathbb{R}^{3}) }  \\
&\geq I_{h_{1,n}}+I_{h_{2,n}}-\left(\sup_{i,n} ||b_{i,n}||_{L^{2}(\mathbb{R}^{3}) } \right)||b_{n}-b_{1,n}-b_{2,n} ||_{L^{2}(\mathbb{R}^{3}) }. 
\end{align*}\\
By the lower semi-continuity of the minimum (4.17) and (5.5), letting $n\to\infty$ and $\varepsilon\to 0$ imply 

\begin{align*}
I_h\geq I_{\bar{h}_1}+I_{\bar{h}_2}.
\end{align*}\\
If $\bar{h}_1>0$ and $\bar{h}_2>0$, this contradicts the strict subadditivity (4.15). Thus $\bar{h}_i\leq 0$ for $i=1$ or $i=2$. We may assume $\bar{h}_1\leq 0$. Since 

\begin{align*}
E[b_n]
\geq E[b_{1,n}]+I_{h_{2,n}}-\left(\sup_{i,n}||b_{i,n}||_{L^{2}(\mathbb{R}^{3}) }\right)  ||b_{n}-b_{1,n}-b_{2,n} ||_{L^{2}(\mathbb{R}^{3}) },
\end{align*}\\
and $\limsup_{n\to\infty}|2E[b_{1,n}]-\alpha|\leq C \varepsilon$ by (5.5), letting $n\to\infty$ and $\varepsilon\to0$ imply 

\begin{align*}
I_{h}\geq \frac{\alpha}{2}+I_{\bar{h}_2}>I_{\bar{h}_2}.
\end{align*}\\
Since $0<h=\bar{h}_1+\bar{h}_2\leq \bar{h}_2$, this contradicts the monotonicity (4.13). We conclude that dichotomy does not occur.\\

\noindent
\textit{Case 2}: vanishing. \\

By H\"older's inequality and (3.2), 

\begin{align*}
|H[b_n]|=4\pi\left|\int_{\mathbb{R}^{2}_{+}}(\phi_n-\phi_\infty)_{+}G_n\frac{1}{r}\dd z\dd r \right|
&\leq 4\pi ||(\phi_n-\phi_\infty)_{+}||_{L^{2}(\mathbb{R}^{2}_{+};r^{-1}) } ||G_n||_{L^{2}(\mathbb{R}^{2}_{+};r^{-1}) } \\
&\leq 4\pi ||(\phi_n- r^{2} )_{+}||_{L^{2}(\mathbb{R}^{2}_{+};r^{-1}) } \left(\sup_{n} ||b_n||_{L^{2}(\mathbb{R}^{3}) }\right).
\end{align*}\\
For arbitrary $R>0$, we estimate 

\begin{align*}
\int_{\mathbb{R}^{2}_{+}}\left(\phi_n-r^{2}\right)_{+}^{2}\frac{1}{r}\dd z\dd r
=\int_{\{r<R\}}\left(\phi_n-r^{2}\right)_{+}^{2}\frac{1}{r}\dd z\dd r+\int_{\{r\geq R\}}\left(\phi_n-r^{2}\right)_{+}^{2}\frac{1}{r}\dd z\dd r.
\end{align*}\\
By the weighted Sobolev inequality (3.8), 

\begin{align*}
&\int_{\{r\geq R\}}\left(\phi_n-r^{2}\right)_{+}^{2}\frac{1}{r}\dd z\dd r
\leq \frac{1}{R} \int_{\{r\geq R\}}\phi_n^{4}\frac{1}{r^{4}}\dd z\dd r
\leq \frac{C}{R}\sup_{n}||\nabla \phi_n||_{L^{2}(\mathbb{R}^{2}_{+};r^{-1}) }^{4},\\
&\int_{\{r<R\}}\left(\phi_n-r^{2}\right)_{+}^{2}\frac{1}{r}\dd z\dd r
\leq \int_{\{r<R\}}\phi_n^{10/3}\frac{1}{r^{11/3}}\dd z\dd r.
\end{align*}\\
For $z_0\in \mathbb{R}$ and $R'>R$, we apply the weighted Sobolev inequality in $D(z_0,R')$ (3.14) to estimate 

\begin{align*}
\int_{D(z_0,R')}|\phi_n|^{10/3}\frac{1}{r^{11/3}}\dd z\dd r
\leq C \left(\int_{D(z_0,R')}|\nabla \phi_n|^{2}\frac{1}{r}\dd z\dd r \right)^{5/3}
=C ||\nabla \phi_n||_{L^{2}(D(z_0,R');r^{-1}) }^{10/3}.
\end{align*}\\
By summing up $D(z_0,R')$ for countable points $z_0\in \mathbb{R}$, 

\begin{align*}
\int_{\{r<R\}}|\phi_n|^{10/3}\frac{1}{r^{11/3}}\dd z\dd r
\leq C\left(\sup_n||\nabla \phi_n||^{2}_{L^{2}(\mathbb{R}^{2}_{+};r^{-1}) } \right)\left(\sup_{z_0\in \mathbb{R}}||\nabla \phi_n||^{4/3}_{L^{2}(D(z_0,R')  ;r^{-1}) } \right).
\end{align*}\\
By (3.2), 

\begin{align*}
\int_{\{r<R\}}|\phi_n|^{10/3}\frac{1}{r^{11/3}}\dd z\dd r
\leq C\left(\sup_n||b_n||^{2}_{L^{2}(\mathbb{R}^{3}) } \right)\left(\sup_{z_0\in \mathbb{R}}||b_n||^{4/3}_{L^{2}(B(z_0e_z,R'))} \right).
\end{align*}\\
The right-hand side vanishes as $n\to\infty$ by (5.3). Thus

\begin{align*}
\limsup_{n\to\infty}\int_{\mathbb{R}^{2}_{+}}\left(\phi_n-r^{2}\right)_{+}^{2}\frac{1}{r}\dd z\dd r\leq \frac{C}{R} \left(\sup_{n}||b_n||_{L^{2}(\mathbb{R}^{3} ) }^{4}\right).
\end{align*}\\
Letting $R\to\infty$ implies that $\lim_{n\to\infty}H[b_n]=0$. This contradicts $H[b_n]\to h>0$. We conclude that vanishing does not occur.\\

\noindent
\textit{Case 3}: compactness. \\

It remains to show the compactness of the sequence $\{b_n\}$ satisfying (5.2). By translation, we may assume that (5.2) holds for $z_n=0$. We may assume that for all $n$,

\begin{align}
2\pi \int_{\mathbb{R}^{2}_{+}\backslash D(0,R_{\varepsilon})}\left(|\nabla \phi_n|^{2}+|G_n|^{2}\right)\frac{1}{r}\dd z\dd r\leq \varepsilon.
\end{align}\\
By Rellich--Kondrakov theorem in the weighted space (3.9), there exists a subsequence and $b=\nabla \times ( \phi\nabla \theta)+G\nabla \theta$ such that $b_n\rightharpoonup b$ in $L^{2}(\mathbb{R}^{3})$ and $\phi_n\to \phi$ in $L^{2}_{\textrm{loc}}(\overline{\mathbb{R}^{2}_{+}}; r^{-1})$. For $D=D(0, R_{\varepsilon} )$,

\begin{align*}
H[b_n]=4\pi \int_{D}(\phi_n-\phi_\infty)_+G_n\frac{1}{r}\dd z\dd r+4\pi \int_{\mathbb{R}^{2}_{+}\backslash D}(\phi_n-\phi_\infty)_+G_n\frac{1}{r}\dd z\dd r.
\end{align*}\\
In a similar way as we proved Theorem 4.9, 

\begin{align*}
\lim_{n\to\infty}\int_{D}(\phi_n-\phi_\infty)_+G_n\frac{1}{r}\dd z\dd r= \int_{D}(\phi-\phi_\infty)_+G\frac{1}{r}\dd z\dd r.
\end{align*}\\
By H\"older's inequality, (3.18), (3.2) and (5.11), 

\begin{align*}
\left|\int_{\mathbb{R}^{2}_{+}\backslash \overline{D}}  (\phi_n-\phi_\infty)_{+}G_n\frac{1}{r}\dd z\dd r\right| 
\leq ||(\phi_n-\phi_\infty)_{+}||_{L^{2}(\mathbb{R}^{2}_{+};r^{-1}) }  ||G_n||_{L^{2}(\mathbb{R}^{2}_{+}\backslash \overline{D}; r^{-1}) }    
 \lesssim \left(\sup_{n }||b_n||^{5/3}_{L^{2}(\mathbb{R}^{3}) }\right)\varepsilon^{1/2}.
\end{align*}\\
Thus 

\begin{align*}
|H[b]-h|=\lim_{n\to\infty}\left|H[b]-H[b_n] \right|\leq \left(\sup_{n }||b_n||^{5/3}_{L^{2}(\mathbb{R}^{3}) }+||b||^{5/3}_{L^{2}(\mathbb{R}^{3}) }   \right)\varepsilon^{1/2}.
\end{align*}\\
By letting $\varepsilon\to 0$, $H[b]=h$ and 

\begin{align*}
I_h \leq E[b]\leq \liminf_{n\to\infty}E[b_n]=I_h.
\end{align*}\\
Thus $b$ is a minimizer of $I_h$. By $\lim_{n\to \infty}||b_n||_{L^{2}(\mathbb{R}^{3}) }= ||b||_{L^{2}(\mathbb{R}^{3}) }$, $b_n\to b$ in $L^{2}(\mathbb{R}^{3})$. The proof is now complete.
\end{proof}

\vspace{5pt}

\begin{thm}
Let $h\in \mathbb{R}$. Let $\{(v_n,b_n)\}\subset L^{2}_{\sigma,\textrm{axi}}(\mathbb{R}^{3})$ be a sequence such that ${\mathcal{E}}[v_n,b_n]\to I_h$ and $H[b_n]\to h$. There exist $\{n_k\}\subset \mathbb{N}$, $\{z_k\}\subset \mathbb{R}$ and $b \in S_h$ such that $(v_{n_k},  b_{n_k}(\cdot+z_ke_z))\to (0,b)$ in $L^{2}(\mathbb{R}^{3})$.
\end{thm}

\vspace{5pt}

\begin{proof}
For $h_n=H[b_n]$, 

\begin{align*}
I_{h_n}\leq E[b_n]\leq {\mathcal{E}}[v_n,b_n].
\end{align*}\\
By the lower semi-continuity of the minimum (4.17), letting $n\to\infty$ implies that $E[b_n]\to I_h$ and $E[v_n]\to 0$. We apply Theorem 5.5 and conclude. 
\end{proof}

\vspace{5pt}

\begin{rem}
The variational problem (1.13) is also well-defined for the function $\phi_\infty=Wr^{2}/2+\gamma$ for $W>0$ and $\gamma\geq 0$ as noted in Remark 3.14. The same compactness theorem as Theorem 5.6 holds for minimizing sequences to (1.13). Namely, for given $h\in \mathbb{R}$, $W>0$, $\gamma\geq 0$, and a sequence $\{(v_n,b_n)\}\subset L^{2}_{\sigma,\textrm{axi}}(\mathbb{R}^{3})$ satisfying 

\begin{align*}
&\frac{1}{2}\int_{\mathbb{R}^{3}}\left(|v_n|^{2}+|b_n|^{2} \right)\dd x\to  I_{h,W,\gamma},\\
&2\int_{\mathbb{R}^{3}}(\phi_n-\phi_\infty )_+\frac{G_n}{r^{2}}\dd x\to  h,
\end{align*}\\
there exist $\{n_k\}$, $\{z_k\}$ and a minimizer $b\in S_{h,W,\gamma}$ of $I_{h,W,\gamma}$ such that $(v_{n_k},b_{n_k}(\cdot+z_ke_z) )\to (0,b)$ in $L^{2}(\mathbb{R}^{3})$. In fact, for $\tilde{h}=(W/2)^{-2}h$ and $\tilde{\gamma}=(W/2)^{-1}\gamma$, $b\in L^{2}_{\sigma,\textrm{axi}}(\mathbb{R}^{3})$ and $\tilde{b}=(W/2)^{-1}b$ satisfy

\begin{align*}
\int_{\mathbb{R}^{3}}|\tilde{b}|^{2}\dd x
&=\left(\frac{W}{2}\right)^{-2}\int_{\mathbb{R}^{3}}|b|^{2}\dd x,\\
\int_{\mathbb{R}^{3}}\left(\tilde{\phi}-r^{2}-\tilde{\gamma} \right)_+\frac{\tilde{G}}{r^{2}}\dd x
&=\left(\frac{W}{2}\right)^{-2}\int_{\mathbb{R}^{3}}\left(\phi-\frac{W}{2}r^{2}-\gamma \right)_+\frac{G}{r^{2}}\dd x.
\end{align*}\\
Minimizers of $I_{h,W,\gamma}$ are those of $I_{\tilde{h},2,\tilde{\gamma}}$ and vice versa. The scaling of the minimum is 

\begin{align*}
I_{h,W,\gamma}=\left(\frac{W}{2}\right)^{-2}I_{\tilde{h},2,\tilde{\gamma}}.
\end{align*}\\
The compactness of minimizing sequences to $I_{h,W,\gamma}$ is derived from that for $I_{\tilde{h},2,\tilde{\gamma}}$ by Theorem 5.5. 
\end{rem}

\section{Leray--Hopf solutions}

We show that axisymmetric Leray--Hopf solutions to (1.9) satisfy the equalities of generalized magnetic helicity (1.17) and generalized magnetic mean-square potential (1.18) by using the vector potential equations. We provide proof of the existence of axisymmetric Leray--Hopf solutions in Appendix A.

\subsection{Existence}
Leray--Hopf solutions to viscous and resistive MHD in $\mathbb{R}^{3}$ can be defined similarly as those for the Navier--Stokes equations, e.g., \cite[p.718]{SSS96}, cf. \cite{Masuda1984}, \cite{Sohr}, \cite{Lemarie}. We define Leray--Hopf solutions to (1.9) with $u_{\infty}, B_{\infty}\in \mathbb{R}^{3}$ and their weak ideal limits by adapting the definitions for the case of bounded domains in Definitions 2.7 and 2.9.

\begin{defn}[Leray--Hopf solutions]
Let $u_{\infty}, B_{\infty}\in \mathbb{R}^{3}$. Let $v_0,b_0\in L^{2}_{\sigma}(\mathbb{R}^{3})$. Let

\begin{align*}
(v, b)\in C_{w}([0,T]; L^{2}_{\sigma}(\mathbb{R}^{3}) )\cap L^{2}(0,T; H^{1}(\mathbb{R}^{3}) ).
\end{align*}\\
Suppose that $(v_t, b_t) \in L^{1}(0,T; (L^{2}_{\sigma}\cap H^{1})(\mathbb{R}^{3})^{*}) )$ and 

\begin{align*}
&<v_t,\xi>+\int_{\mathbb{R}^{3}}((v+u_{\infty})\cdot \nabla v-(b+B_{\infty})\cdot \nabla b)\cdot \xi\dd x+\nu \int_{\mathbb{R}^{3}}\nabla v:\nabla \xi\dd x=0,\\
&<b_t,\zeta>+\int_{\Omega}(b\times v+B_{\infty}\times v+b\times u_{\infty} )\cdot \nabla \times \zeta\dd x +\mu \int_{\Omega}\nabla \times b\cdot \nabla \times \zeta \dd x=0,
\end{align*}\\
for a.e. $t\in [0,T]$ and every $\xi, \zeta\in L^{2}_{\sigma}\cap H^{1} (\mathbb{R}^{3})$. Suppose furthermore that $(v(\cdot,0),b(\cdot,0))=(v_0,b_0)$ and 

\begin{align}
\frac{1}{2}\int_{\mathbb{R}^{3}}\left(|v|^{2}+|b|^{2}\right) \dd x
+\int_{0}^{t}\int_{\mathbb{R}^{3}}\left(\nu |\nabla v|^{2}+\mu |\nabla b|^{2}\right) \dd x\dd s
\leq \frac{1}{2}\int_{\mathbb{R}^{3}}\left( |v_0|^{2}+|b_0|^{2} \right) \dd x, 
\end{align}\\
for all $t\in [0,T]$. Then, we call $(v,b)$ Leray--Hopf solutions to (1.9).
\end{defn}

\begin{defn}[Weak ideal limits]
Let $(v_j,b_j)$ be a Leray--Hopf solution to (1.9) for $\nu_j,\mu_j>0$ and $(v_{0,j},b_{0,j})$ such that $(v_{0,j},b_{0,j})\rightharpoonup (v_{0},b_{0})$ in $L^{2}(\mathbb{R}^{3})$ as $(v_j,b_j)\to (0,0)$. Assume that

\begin{align}
(v_j,b_j)\overset{\ast}{\rightharpoonup}  (v,b) \quad \textrm{in}\ L^{\infty}(0,T; L^{2}(\mathbb{R}^{3}) ). 
\end{align}\\
Then, we call $(v,b)$ a weak ideal limit of $(v_j,b_j)$. If instead $\nu_j=\nu>0$ for every $j$ and $\mu_j\to 0$, we call $(v,b)$ a weak nonresistive limit of $(v_j,b_j)$.
\end{defn}

Using Leray's method in Appendix A, we construct Leray--Hopf solutions to (1.9). They are axisymmetric for constants $u_{\infty},B_{\infty}$ parallel to $e_z$ and axisymmetric data $v_0,b_0\in L^{2}_{\sigma,\textrm{axi}}(\mathbb{R}^{3})$.

\begin{thm}
Let $B_{\infty}=-2e_{z}$. Let $u_{\infty}$ be a constant parallel to $e_z$. For $v_0,b_0\in L^{2}_{\sigma,\textrm{axi}}(\mathbb{R}^{3})$, there exists an axisymmetric Leray--Hopf solution to (1.9). 
\end{thm}

\subsection{Vector potential equations}

For a smooth solution $(v,b)$ to (1.9), $(u,B)=(v+u_\infty,b+B_{\infty})$ is a solution to (1.1) and (1.7). We apply the Clebsch representation (3.16) for $B$ and its vector potential $A$ with the function $\phi_{\infty}=r^{2}+\gamma$ and $\gamma\geq 0$. By the equation $(1.1)_2$ and some potential $Q$, the vector potential equations can be expressed as 

\begin{equation}
\begin{aligned}
A_{t}+B\times u+\nabla Q&=-\mu \nabla \times B.
\end{aligned}
\end{equation}\\
The $\theta$-component of this equation is the drift-diffusion equation for $\Phi=\phi-\phi_{\infty}$ and the flux function $\phi$ of $b$, 

\begin{align}
\Phi_{t}+u\cdot \nabla \Phi&=\mu\left(\Delta -\frac{2}{r}\partial_r\right)\Phi.
\end{align}\\
The quantities 

\begin{align}
\int_{\mathbb{R}^{3}}g(\Phi)\frac{G}{r^{2}}\dd x,\quad \int_{\mathbb{R}^{3}}g(\Phi)\dd x,
\end{align}\\
are conserved at the zero resistivity limit $\mu=0$ for arbitrary regular functions $g=g(s)$ by the vector potential equation (6.3). We apply the identity for axisymmetric magnetic fields 

\begin{align}
A_t\cdot B\dot{g}(\Phi)=\partial_t \left(g(\Phi)\frac{G}{r^{2}}\right)+\nabla \cdot (g(\Phi)\nabla \theta\times  A_t).
\end{align}\\
This identity follows from (3.16) and $B\dot{g}(\Phi)=\nabla \times (g(\Phi)\nabla \theta)+G\dot{g}(\Phi)\nabla \theta$. We can observe generalized magnetic helicity conservation from (6.6) by multiplying the solenoidal $B\dot{g}(\Phi)$ by (6.3) and integration by parts. We take $g(s)=2s_+$ for generalized magnetic helicity and $g(s)=s_{+}^{2}$ for generalized magnetic mean-square potential in (6.5). Their equalities with resistivity $\mu>0$ are of the form

\begin{align}
\int_{\mathbb{R}^{3}}\Phi_{+}\frac{G}{r^{2}}\dd x
+\mu \int_{0}^{t}\int_{\mathbb{R}^{3}}\nabla \times B\cdot B1_{(0,\infty)}(\Phi) \dd x\dd s
&=\int_{\mathbb{R}^{3}}\Phi_{0,+}\frac{G_0}{r^{2}}\dd x, \\
\int_{\mathbb{R}^{3}}\Phi^{2}_{+}\dd x
+2\mu \int_{0}^{t}\int_{\mathbb{R}^{3}} |\nabla \Phi_{+}|^{2} \dd x\dd s
&= \int_{\mathbb{R}^{3}}\Phi^{2}_{0,+} \dd x. 
\end{align}\\
We demonstrate (6.7) and (6.8) for axisymmetric Leray--Hopf solutions by showing that the equations (6.3) hold on $L^{2}(\mathbb{R}^{3})$ for a.e. $t\in [0,T]$. In terms of the projection operator $\mathbb{P}$, associated with (2.1) in $\Omega=\mathbb{R}^{3}$, the equations (6.3) can be expressed as 

\begin{align*}
A_t+\mathbb{P}(B\times u)=-\mu\nabla \times B\quad \textrm{on}\ L^{2}_{\sigma}(\mathbb{R}^{3}).
\end{align*}\\
At the heuristic level, the conditions $B\times u, \nabla \times B=\nabla \times b\in L^{2}(\mathbb{R}^{3})$ imply that $\nabla Q=-(1-\mathbb{P})(B\times u)\in L^{2}(\mathbb{R}^{3})$. Hence a distributional solution $A$ of (6.3) satisfies $A_t\in L^{2}(\mathbb{R}^{3})$ and (6.3) holds on $L^{2}(\mathbb{R}^{3})$. We approximate $A$ for the time variable and derive (6.3) from the definition of Leray--Hopf solutions.

\begin{prop}
For axisymmetric Leray--Hopf solutions $(v,b)$ to (1.9) for $B_{\infty}=-2e_z$ and $u_{\infty}$ parallel to $e_z$, $(u,B)=(v+u_{\infty},b+B_{\infty})$ satisfies 

\begin{align}
\begin{aligned}
B\times u\in L^{4/3}(0,T; L^{2}(\mathbb{R}^{3})),  \\
\nabla \times B\in L^{2}(0,T; L^{2}(\mathbb{R}^{3})).
\end{aligned}
\end{align}\\
Moreover, for $A$ defined by (3.16), there exists some $Q$ such that 

\begin{align}
A_t, \nabla Q\in L^{4/3}(0,T; L^{2}(\mathbb{R}^{3})),  
\end{align}\\
and (6.3) holds on $L^{2}(\mathbb{R}^{3})$ for a.e. $t\in [0,T]$.
\end{prop}

\begin{proof}
For the vector potential $a$ of $b$ defined by (3.3), 

\begin{align*}
A_t&=a_t,\\
B\times u&=b\times v+b\times u_{\infty}+B_{\infty}\times v, \\
\nabla \times B&=\nabla \times b.
\end{align*}\\
By $L^{\infty}(0,T; L^{2}(\mathbb{R}^{3}))\cap L^{2}(0,T; H^{1}(\mathbb{R}^{3}) )\subset L^{8/3}(0,T; L^{4} (\mathbb{R}^{3}))$, (6.9) holds. By the definition of Leray--Hopf solutions, 

\begin{align*}
b_t+\nabla \times (B\times u)+\mu \nabla \times (\nabla \times b)=0\quad \textrm{on}\ (L^{2}_{\sigma}\cap H^{1}) (\mathbb{R}^{3})^{*},
\end{align*}\\
for a.e. $t\in [0,T]$. Thus 

\begin{align*}
\nabla \times (a_{t}+(B\times u)+\mu \nabla \times b)=0\quad \textrm{on}\ (L^{2}_{\sigma}\cap H^{1}) (\mathbb{R}^{3})^{*}.
\end{align*}\\
We take arbitrary $0<\delta<T/2$ and $\varepsilon<\delta$. By the mollifier in Lemma 2.13, 

\begin{align*}
\nabla \times (a_{t}^{\varepsilon}+(B\times u)^{\varepsilon}+\mu \nabla \times b^{\varepsilon})=0\quad \textrm{on}\ (L^{2}_{\sigma}\cap H^{1}) (\mathbb{R}^{3})^{*},
\end{align*}\\
for $t\in (\delta, T-\delta)$. By Lemma 3.5, $a\in L^{\infty}(0,T; L^{6}(\mathbb{R}^{3}))$ and $a_{t}^{\varepsilon}\in C^{\infty}(\delta,T-\delta; L^{6}(\mathbb{R}^{3}))$. For an arbitrary $\zeta\in L^{2}_{\sigma} \cap H^{1}(\mathbb{R}^{3})$ satisfying $\nabla \times \zeta \in L^{6/5}(\mathbb{R}^{3})$, by integration by parts,  

\begin{align*}
\int_{\mathbb{R}^{3}}(a_{t}^{\varepsilon}+(B\times u)^{\varepsilon}+\mu \nabla \times b^{\varepsilon} ) \cdot \nabla \times \zeta \dd x=0.
\end{align*}\\
We take an arbitrary $\xi \in L^{6/5}\cap L^{2}(\mathbb{R}^{3})$ and apply Proposition 3.8 to take $\zeta\in L^{2}_{\sigma}\cap L^{6}(\mathbb{R}^{3})$ such that $\nabla \zeta\in L^{6/5}\cap L^{2}(\mathbb{R}^{3})$ and $\nabla \times \zeta=\mathbb{P}\xi$. Then, for $\nabla Q^{\varepsilon}=-(I-\mathbb{P})(B\times u)^{\varepsilon}$, 

\begin{align*}
0=\int_{\mathbb{R}^{3}}(a_{t}^{\varepsilon}+(B\times u)^{\varepsilon}+\mu \nabla \times b^{\varepsilon} ) \cdot \mathbb{P}\xi \dd x 
&=\int_{\mathbb{R}^{3}}(a_{t}^{\varepsilon}+\mathbb{P}(B\times u)^{\varepsilon}+\mu \nabla \times b^{\varepsilon} ) \cdot \xi \dd x \\
&=\int_{\mathbb{R}^{3}}(a_{t}^{\varepsilon}+(B\times u)^{\varepsilon}+\nabla Q^{\varepsilon}+\mu \nabla \times b^{\varepsilon} ) \cdot \xi \dd x.
\end{align*}\\
Since $(B\times u)^{\varepsilon}\in L^{2}(\mathbb{R}^{3})$ for $t\in (\delta,T-\delta)$ and $\xi \in L^{6/5}\cap L^{2}(\mathbb{R}^{3})$ is arbitrary, $\nabla Q^{\varepsilon}, a^{\varepsilon}_{t}\in L^{2}(\mathbb{R}^{3})$ and 

\begin{align*}
a^{\varepsilon}_{t}+(B\times u)^{\varepsilon}+\nabla Q^{\varepsilon}=-\mu\nabla \times b^{\varepsilon}\quad \textrm{on}\ L^{2}(\mathbb{R}^{3}).
\end{align*} \\
By Lemma 2.13, letting $\varepsilon\to0$ implies that 

\begin{align*}
(B\times u)^{\varepsilon}&\to B\times u\quad \textrm{in}\ L^{4/3}(\delta,T-\delta; L^{2}(\mathbb{R}^{3}) ),\\
\nabla \times b^{\varepsilon}&\to \nabla \times b \quad \textrm{in}\ L^{2}(\delta,T-\delta; L^{2}(\mathbb{R}^{3}) ).
\end{align*}\\
Thus $\nabla Q=-(1-\mathbb{P})(B\times u)$ and $a_{t}=A_t$ satisfy (6.10). Since $0<\delta<T/2$ is arbitrary, (6.3) holds for a.e. $t\in [0,T]$.
\end{proof}

\begin{prop}
The equation (6.4) holds on $L^{2}_{\textrm{loc}}(\mathbb{R}^{3})$ for a.e. $t\in [0,T]$ for axisymmetric Leray--Hopf solutions to (1.9) for $B_{\infty}=-2e_z$ and $u_\infty$ parallel to $e_z$. Moreover, 

\begin{align}
\partial_t\Phi_{+}^{2}+u\cdot \nabla \Phi_{+}^{2}=\mu\left(\Delta-\frac{2}{r}\partial_r \right)\Phi_{+}^{2}-2\mu |\nabla \Phi_{+}|^{2} \quad \textrm{on}\ L^{1}_{\textrm{loc}}(\mathbb{R}^{3}).  
\end{align}
\end{prop}

\begin{proof}
By (3.16), 

\begin{align*}
A_t\cdot r\nabla \theta&=\frac{1}{r} \Phi_t,\\
(B\times u)\cdot r\nabla \theta&=\frac{1}{r} u\cdot \nabla \Phi, \\
\nabla \times B\cdot r\nabla \theta&=-\frac{1}{r}\left(\Delta-\frac{2}{r}\partial_r \right)\Phi.
\end{align*}\\
Each term belongs to $L^{2}(\mathbb{R}^{3})$ by (6.9) and (6.10). By multiplying $r\nabla \theta$ by (6.3), 

\begin{align*}
\frac{1}{r}\left(\Phi_t+u\cdot \nabla \Phi-\mu \left(\Delta -\frac{2}{r}\partial_r\right)\Phi\right)=0\quad \textrm{on}\ L^{2}(\mathbb{R}^{3}). 
\end{align*}\\
Thus (6.4) holds on $L^{2}_{\textrm{loc}}(\mathbb{R}^{3})$. Since $\Phi_{+}\in L^{\infty}(0,T; L^{2}(\mathbb{R}^{3}))$ by (3.19) and (3.2), by multiplying $2\Phi_{+}$ by the above equation, (6.11) follows. 
\end{proof}

\subsection{Equalities}

We show the equalities (6.7) and (6.8) for axisymmetric Leray--Hopf solutions to (1.9). 

\begin{prop}
For axisymmetric Leray--Hopf solutions to (1.9) for $B_{\infty}=-2e_z$ and $u_\infty$ parallel to $e_z$,  

\begin{align}
\frac{\dd }{\dd t}\int_{\mathbb{R}^{3}}\Phi_+\frac{G}{r^{2}}\dd x
=-\mu \int_{\mathbb{R}^{3}} \nabla \times B\cdot B 1_{(0,\infty)}(\Phi)\dd x,
\end{align}\\
in the sense of distribution.
\end{prop}

\begin{proof}
The field $B1_{(0,\infty)}(\Phi)=\nabla \times (\Phi_+\nabla \theta)+G1_{(0,\infty)}(\Phi)\nabla \theta$ is solenoidal. By (3.17) and (3.2), 

\begin{align}
||B1_{(0,\infty)}(\Phi)||_{L^{2}(\mathbb{R}^{3})}=||(b+B_\infty)1_{(0,\infty)}(\Phi)||_{L^{2}(\mathbb{R}^{3})}\lesssim ||b||_{L^{2}(\mathbb{R}^{3})}.
\end{align}\\
By multiplying $B1_{(0,\infty)}(\Phi)\in L^{\infty}(0,T; L^{2}_{\sigma}(\mathbb{R}^{3}))$ by (6.3), 

\begin{align*}
A_{t}\cdot B1_{(0,\infty)}(\Phi)+\nabla Q\cdot B1_{(0,\infty)}(\Phi)=-\mu \nabla \times B\cdot B1_{(0,\infty)}(\Phi)\quad \textrm{on}\ L^{1}(\mathbb{R}^{3}),
\end{align*}\\
for a.e. $t\in [0,T]$. Thus

\begin{align*}
\int_{\mathbb{R}^{3}}A_{t}\cdot B1_{(0,\infty)}(\Phi)\dd x
=-\mu \int_{\mathbb{R}^{3}}\nabla \times B\cdot B1_{(0,\infty)}(\Phi) \dd x. 
\end{align*}\\
For $\chi\in C^{\infty}_{c}[0,\infty)$ satisfying $\chi=1$ in $[0,1]$ and $\chi=0$ in $[2,\infty)$, we set $\chi_R(x)=\chi(R^{-1}|x|)$. For arbitrary $\rho\in C^{\infty}_{c}(0,T)$, by multiplying $\chi_R\rho$ by (6.6) and integration by parts,

\begin{align*}
\int_{0}^{T}\int_{\mathbb{R}^{3}}A_t \cdot B1_{(0,\infty)}(\Phi)\chi_R\rho\dd x\dd t
=-\int_{0}^{T}\int_{\mathbb{R}^{3}}\Phi_+ \frac{G}{r^{2}}\chi_R\dot{\rho} \dd x\dd t-\int_{0}^{T}\int_{\mathbb{R}^{3}}( \Phi_+\nabla \theta \times A_t)\cdot \nabla \chi_R\rho \dd x\dd t.
\end{align*}\\
Since $\Phi_+\nabla \theta\in L^{\infty}(0,T; L^{2}(\mathbb{R}^{3}))$ and $A_t\in L^{4/3}(0,T; L^{2}(\mathbb{R}^{3}) )$ by (3.18), (3.2) and (6.10), $\Phi_+\nabla\theta \times  A_t \in L^{4/3}(0,T; L^{1}(\mathbb{R}^{3}) )$. The last term vanishes as $R\to\infty$. Since $\rho\in C^{\infty}_{c}(0,T)$ is arbitrary, 

\begin{align*}
\frac{\dd }{\dd t}\int_{\mathbb{R}^{3}}\Phi_+\frac{G}{r^{2}}\dd x
=-\mu \int_{\mathbb{R}^{3}}\nabla \times B\cdot B1_{(0,\infty)}(\Phi) \dd x,
\end{align*}\\
in the sense of distribution. Thus (6.12) holds. 
\end{proof}

\begin{lem}
The equality (6.7) holds for axisymmetric Leray--Hopf solutions to (1.9) for $B_\infty=-2e_z$ and $u_\infty$ parallel to $e_z$ and all $t\in [0,T]$. Moreover, 

\begin{equation}
\begin{aligned}
\left\|\frac{\dd }{\dd t}\int_{\mathbb{R}^{3}}\Phi_+\frac{G}{r^{2}}\dd x\right\|_{L^{2}(0,T)}
\leq C\mu^{1/2}\left(||v_0||_{L^{2}(\mathbb{R}^{3}) }^{2}+||b_0||_{L^{2}(\mathbb{R}^{3}) }^{2}  \right). 
\end{aligned}
\end{equation}
\end{lem}

\begin{proof}
By (6.13), 

\begin{align*}
\left|\int_{\mathbb{R}^{3}} \nabla \times B\cdot B 1_{(0,\infty)}(\Phi)\dd x\right|
=\left|\int_{\mathbb{R}^{3}} \nabla \times b\cdot B 1_{(0,\infty)}(\Phi)\dd x\right| 
\lesssim ||\nabla \times b||_{L^{2}} || b||_{L^{2}}. 
\end{align*}\\
By (6.12) and (6.1), 

\begin{align*}
\left\|\frac{\dd }{\dd t}\int_{\mathbb{R}^{3}}\Phi_{+}\frac{G}{r^{2}}\dd x\right\|_{L^{2}(0,T) }
\lesssim 
\mu ||\nabla \times b||_{L^{2}(0,T; L^{2}(\mathbb{R}^{3}))} || b||_{L^{\infty}(0,T; L^{2}(\mathbb{R}^{3}))} 
\lesssim \mu^{1/2}\left(||v_0||_{L^{2}(\mathbb{R}^{3}) }^{2}+||b_0||_{L^{2}(\mathbb{R}^{3}) }^{2} \right).
\end{align*}\\
Thus (6.14) holds. The equality (6.7) follows by integrating (6.12) in time.
\end{proof}

\begin{lem}
The equality (6.8) holds for axisymmetric Leray--Hopf solutions to (1.9) for $B_\infty=-2e_z$ and $u_\infty$ parallel to $e_z$ and all $t\in [0,T]$. 
\end{lem}

\begin{proof}
We take non-increasing $\chi\in C^{\infty}_{c}[0,\infty)$ such that $\chi=1$ in $[0,1]$ and $\chi=0$ in $[2,\infty)$ and set $\chi_R(x)=\chi(R^{-1}|x|)$. Since $\Phi_+=(-\gamma)_+=0$ on $\{r=0\}$, by multiplying $\chi_R$ by (6.11) and integration by parts, 

\begin{align*}
&\int_{\mathbb{R}^{3}}\Phi_{+}^{2}\chi_R\dd x-\int_{\mathbb{R}^{3}}\Phi_{0,+}^{2}\chi_R\dd x
+2\mu \int_{0}^{t}\int_{\mathbb{R}^{3}} |\nabla \Phi_{+}|^{2}\chi_R \dd x\dd s \\
&=\mu\int_{0}^{t}\int_{\mathbb{R}^{3}} \Phi_{+}^{2}\left(\Delta +\frac{2}{r}\partial_r\right)\chi_R\dd x\dd s+\int_{0}^{t}\int_{\mathbb{R}^{3}}u\cdot \nabla \chi_R \Phi_{+}^{2}\dd x\dd s. 
\end{align*}\\
Since $\Phi_{+}\in L^{\infty}(0,T; L^{q}(\mathbb{R}^{3}) )$ for $1\leq q<\infty$ by (3.19) and $u=v+u_{\infty}$ for $v\in L^{\infty}(0,T; L^{2}(\mathbb{R}^{3}) )$, the right-hand side vanishes as $R\to\infty$. The equality (6.8) follows from the monotone convergence theorem.
\end{proof}

\section{Weak ideal limits}

We demonstrate generalized magnetic helicity conservation at weak ideal limits of axisymmetric Leray--Hopf solutions to (1.9) for fixed initial data in Theorems 7.5. We show local flux convergence and the generalized magnetic mean-square potential conservation at weak ideal limits by the vector potential equations (6.3) and the Aubin--Lions lemma. The equality of generalized magnetic mean-square potential for axisymmetric Leray--Hopf solutions (6.8) strengthens the local flux convergence to the global. By taking a limit to the equality of generalized magnetic helicity for axisymmetric Leray--Hopf solutions (6.7), we show the desired generalized magnetic helicity conservation at weak ideal limits.

\subsection{Local convergence}

We estimate the time derivative of the vector potential by the equations (6.3) and apply the Aubin--Lions lemma. 

\begin{prop}
Let $(v_j,b_j)$ be an axisymmetric Leray--Hopf solution to (1.9) for fixed  $v_0,b_0\in L^{2}_{\sigma,\textrm{axi}}(\mathbb{R}^{3})$ with $(\nu_j,\mu_j)$, $B_{\infty}=-2e_z$, and $u_{\infty}$ parallel to $e_z$. Let $(v,b)$ be a weak ideal limit of $(v_j,b_j)$ as $(\nu_j,\mu_j)\to (0,0)$. Let $A_j$ and $A$ be vector potentials of $B_j=b_j+B_{\infty}$ and $B=b+B_{\infty}$ defined by (3.16). There exists a subsequence such that 

\begin{align}
A_j\to A\quad \textrm{in}\ L^{2}(0,T; L^{2}_{\textrm{loc}}(\mathbb{R}^{3}) ).
\end{align}\\
In particular, for flux functions $\phi_j$ of $b_j$ and $\phi$ of $b$,  

\begin{align}
\frac{1}{r}\phi_j\to \frac{1}{r}\phi\quad \textrm{in}\ L^{2}(0,T; L^{2}_{\textrm{loc}}(\mathbb{R}^{3}) ).  
\end{align}
\end{prop}

\begin{proof}
By Proposition 6.4, the vector potential $a_j$ of $b_j$ defined by (3.3) and $u_j=v_j+u_{\infty}$ satisfy

\begin{align*}
\partial_t a_j+\mathbb{P}(B_j\times u_j)=-\mu_j\nabla \times b_j\quad \textrm{on}\ L^{2}_{\sigma}(\mathbb{R}^{3}),
\end{align*}\\
for a.e. $t\in [0,T]$, where $\mathbb{P}$ is the projection operator associated with (2.1). We take an open  ball $B(0,R)\subset \mathbb{R}^{3}$ with radius $R>0$ and an arbitrary $\xi\in W^{1,4}_{0}(B(0,R))$. We extend $\xi$ to $\mathbb{R}^{3}$ by the zero extension and denote it by the same symbol. By multiplying $\xi$ by the above equation and integration by parts,

\begin{align*}
\int_{B(0,R)}\partial_t a_{j}\cdot \xi\dd x+\int_{B(0,R)}B_j\times u_j\cdot \mathbb{P} \xi\dd x
=-\mu_j\int_{B(0,R)}b_j\cdot \nabla \times \xi\dd x.
\end{align*}\\
By the Sobolev embedding and boundedness of $\mathbb{P}$ on $L^{4}(\mathbb{R}^{3})$,

\begin{align*}
||\mathbb{P} \xi||_{L^{\infty}(\mathbb{R}^{3})}\lesssim ||\mathbb{P} \xi||_{W^{1,4}(\mathbb{R}^{3})}\lesssim || \xi||_{W^{1,4}(\mathbb{R}^{3})}=|| \xi||_{W^{1,4}(B(0,R) )}.
\end{align*}\\
Thus 

\begin{align*}
\left|\int_{B(0,R)}\partial_t a_{j}\cdot  \xi\dd x\right|
&\lesssim ||B_j\times u_j||_{L^{1}(B(0,R)) }|| \xi||_{W^{1,4}(B(0,R))}
+\mu_j ||b_j||_{L^{2}(\mathbb{R}^{3}) }||\nabla  \xi||_{L^{2}(B(0,R))} \\
&\lesssim (||B_j\times u_j||_{L^{1}(B(0,R)) }
+\mu_j ||b_j||_{L^{2}(\mathbb{R}^{3}) })|| \xi||_{W^{1,4}(B(0,R))}.
\end{align*}\\
By taking the supremum for $ \xi$,

\begin{align*}
||\partial_t a_j||_{W^{1,4}_{0}(B(0,R))^{*} }\lesssim ||B_j\times u_j||_{L^{1}(B(0,R)) }
+\mu_j ||b_j||_{L^{2}(\mathbb{R}^{3}) }.
\end{align*}\\
Since $B_j\times u_j=b_j\times v_j+B_{\infty}\times v_j+b_j\times u_{\infty}$ is uniformly bounded in $L^{\infty}(0,T; L^{1}(B(0,R)) )$ by (6.1), so is $\partial_t a_j$ in $L^{\infty}(0,T; W^{1,4}_0(B(0,R))^{*} )$. By Lemma 2.14, there exists a subsequence such that for the vector potential $a$ of $b$ defined by (3.3),

\begin{align*}
a_j\to a\quad  \textrm{in}\ L^{2}(0,T; L^{2}(B(0,R))).
\end{align*}\\
Since $R>0$ is arbitrary, by a diagonal argument

\begin{align*}
a_j\to a\quad  \textrm{in}\ L^{2}(0,T; L^{2}_{\textrm{loc}}(\mathbb{R}^{3})).
\end{align*}\\
Thus (7.1) holds for $A_j=a_j-\phi_{\infty}\nabla \theta$, $\phi_{\infty}=r^{2}+\gamma$ and $\gamma\geq 0$. By taking the $\theta$-component $a_j\cdot r\nabla \theta=r^{-1} \phi_j$ and (7.2) follows.
\end{proof}

\subsection{Mean-square potential conservation}

We show that the convergence (7.2) implies that for $u=v+u_{\infty}$, $\Phi=\phi-\phi_{\infty}$ is a distributional solution to 

\begin{align}
\Phi_{t}+u\cdot \nabla \Phi&=0\quad \textrm{in}\ \mathbb{R}^{3}\times (0,T).
\end{align}\\ 
Furthermore, we show that $\Phi_{+}^{2}$ is a distributional solution to  

\begin{align}
\partial_t \Phi^{2}_{+}+u\cdot \nabla \Phi^{2}_{+}&=0\quad \textrm{in}\ \mathbb{R}^{3}\times (0,T),
\end{align}\\ 
and conserves generalized magnetic mean-square potential.

\begin{prop}
Let $\phi_j, \phi, v$ and $u_{\infty}$ be as in Proposition 7.1. Let $\phi_{\infty}=r^{2}+\gamma$ and $\gamma\geq 0$. For $\Phi_j=\phi_j-\phi_{\infty}$ and $\Phi=\phi-\phi_{\infty}$, 

\begin{align}
\Phi_j\to \Phi \quad \textrm{in}\ L^{2} (0,T; L^{2}_{\textrm{loc}}(\mathbb{R}^{3}) ).
\end{align}\\
For $u=v+u_{\infty}$, $\Phi$ is a distributional solution to (7.3). Moreover,  

\begin{align}
\Phi_{j,+}&\to \Phi_{+}\quad \textrm{in}\ L^{2}(0,T; L^{2}_{\textrm{loc}}(\mathbb{R}^{3})), \\
\Phi_{j,+}^{2}&\to \Phi_{+}^{2}\quad \textrm{in}\ L^{2}(0,T; L^{2}_{\textrm{loc}}(\mathbb{R}^{3})).
\end{align}\\
The function $\Phi_{+}^{2}$ is a distributional solution to (7.4) and satisfies  

\begin{align}
\int_{\mathbb{R}^{3}}\Phi_{+}^{2}\dd x=\int_{\mathbb{R}^{3}}\Phi_{0,+}^{2}\dd x,  
\end{align}\\
for a.e. $t\in [0,T]$.
\end{prop}

\begin{proof}
The convergence (7.2) implies (7.5). For an arbitrary $\varphi\in C^{\infty}_{c}(\mathbb{R}^{3}\times [0,T))$, we take $R>0$ such that $\textrm{spt}\ \varphi\subset B(0,R)\times [0,T]$. The function $\Phi_j$ satisfies (6.4) on $L^{2}_{\textrm{loc}}(\mathbb{R}^{3})$ for a.e. $t\in [0,T]$ by Proposition 6.5.  By multiplying $\varphi$ by (6.4) and integration by parts, 

\begin{align*}
\int_0^{T}\int_{\mathbb{R}^{3}}\Phi_j(\partial_t \varphi  +(v_j+u_{\infty})\cdot \nabla  \varphi) \dd x\dd s
+\int_{\mathbb{R}^{3}} \Phi_0\varphi_0\dd x
=-\mu_j\int_0^{T}\int_{\mathbb{R}^{3}} \left( \Phi_j \Delta \varphi-\frac{2}{r}\partial_r \Phi_j\varphi\right)\dd x\dd s.
\end{align*}\\
The function $\Phi_j$ is uniformly bounded in $L^{2}(0,T; L^{2}(B(0,R)) )$. By (3.2) and (6.1), $r^{-1}\nabla \Phi_j$ is uniformly bounded in $L^{\infty}(0,T; L^{2}(B(0,R)) )$. Thus, the right-hand side vanishes as $j\to\infty$. The field $v_j$ is uniformly bounded in $L^{\infty}(0,T; L^{2}(\mathbb{R}^{3}) )$ by (6.1). By 

\begin{align*}
\left|\int_{0}^{T}\int_{\mathbb{R}^{3}}(\Phi_j v_{j}-\Phi v)  \cdot \nabla \varphi \dd x\dd s\right|
&\leq \sup_j ||v_{j}||_{L^{\infty}(0,T; L^{2}(\mathbb{R}^{3}) ) }
||\Phi_j-\Phi||_{L^{1}(0,T; L^{2}(B) ) }||\nabla \varphi||_{L^{\infty}(\mathbb{R}^{3}\times (0,T)) } \\
&+\left|\int_{0}^{T}\int_{\mathbb{R}^{3}}( v_{j}- v) \Phi\cdot \nabla \varphi \dd x\dd s\right|\to 0,
\end{align*}\\
the limit $\Phi$ satisfies 

\begin{align*}
\int_0^{T}\int_{\mathbb{R}^{3}}\Phi(\partial_t \varphi  +(v+u_{\infty})\cdot \nabla  \varphi) \dd x\dd s
=-\int_{\mathbb{R}^{3}} \Phi_0\varphi_0\dd x,
\end{align*}\\
and (7.3) holds in the distributional sense.

The convergence (7.5) implies (7.6) since $|\tau_+-s_{+}|\leq |\tau-s|$ for $\tau,s\in \mathbb{R}$. The function $\Phi_{j,+}$ is uniformly bounded in $L^{\infty}(0,T; L^{q}(\mathbb{R}^{3}) )$ for $1\leq q<\infty$ by (3.19) and (6.1). By H\"older's inequality, for $l>4$ and $\theta\in (0,1)$ satisfying  $1/4=\theta/2+(1-\theta)/l$,

\begin{align*}
||\Phi_{j,+}-\Phi_{+}  ||_{L^{4}(B(0,R)) }\leq||\Phi_{j,+}-\Phi_{+}  ||_{L^{2}(B(0,R)) }^{\theta}||\Phi_{j,+}-\Phi_{+}  ||_{L^{l}(B(0,R)) }^{1-\theta}. 
\end{align*}\\
This implies that $\Phi_{j,+}\to \Phi_{+}$ in $L^{2} (0,T; L^{4}_{\textrm{loc}}(\mathbb{R}^{3}) )$ since $2\theta<1$. By applying H\"older's inequality for $\Phi_{j,+}^{2}-\Phi_{+}^{2}=(\Phi_{j,+} -\Phi_{+}  )( \Phi_{j,+} +\Phi_{+})$, 

\begin{align*}
||  \Phi_{j,+}^{2}-\Phi_{+}^{2}||_{L^{2}(B(0,R)) }
\leq ||  \Phi_{j,+}-\Phi_{+}||_{L^{4}(B(0,R)) }||  \Phi_{j,+}+\Phi_{+}||_{L^{4}(B(0,R)) }.
\end{align*}\\
Thus (7.7) holds. By Proposition 6.5, $\Phi_{j,+}^{2}$ satisfies (6.11). For arbitrary $\varphi\in C^{\infty}_{c}(\mathbb{R}^{3}\times [0,T))$, 

\begin{align*}
&\int_0^{T}\int_{\mathbb{R}^{3}}\Phi_{j,+}^{2}(\partial_t \varphi  +(v_j+u_{\infty})\cdot \nabla  \varphi) \dd x\dd s
+\int_{\mathbb{R}^{3}} \Phi_{0,+}^{2}\varphi_0\dd x \\
&=-\mu_j\int_0^{T}\int_{\mathbb{R}^{3}} \left( \Phi_{j,+}^{2}  \Delta \varphi+\frac{2}{r}\partial_r \Phi_{j,+}^{2}\varphi  -2|\nabla \Phi_{j,+}|^{2}  \varphi\right)\dd x\dd s.
\end{align*}\\
We take $R>0$ such that $\textrm{spt}\ \varphi\subset B(0,R)\times [0,T]$. The functions $\Phi_{j,+}$ and $r^{-1}\partial_r \Phi_j$ are uniformly bounded in $L^{\infty}(0,T; L^{2}(B(0,R)))$. Hence $r^{-1}\partial_r \Phi_{j,+}^{2}=2r^{-1}\partial_r \Phi_{j} \Phi_{j,+}$ is uniformly bounded in $L^{\infty}(0,T; L^{1}( B(0,R)))$. By $|\nabla \Phi_{j,+}|=\nabla \Phi_j1_{(0,\infty)}(\Phi_j)\leq |\nabla \Phi_j|$, $|\nabla \Phi_{j,+}|^{2}\in L^{\infty}(0,T; L^{1}(B(0,R)))$ is uniformly bounded. Thus, the right-hand side vanishes as $j\to\infty$. By (7.7) and (6.2), in a similar way as we showed (7.3), letting $j\to\infty$ implies that 

\begin{align*}
\int_0^{T}\int_{\mathbb{R}^{3}}\Phi_{+}^{2}(\partial_t \varphi  +(v+u_{\infty})\cdot \nabla  \varphi) \dd x\dd s=
-\int_{\mathbb{R}^{3}} \Phi_{0,+}^{2}\varphi_0\dd x. 
\end{align*}\\
Thus, $\Phi_{+}^{2}$ is a distributional solution to (7.4). 

We take $\chi\in C^{\infty}_{c}[0,\infty)$ such that $\chi=1$ in $[0,1]$ and $\chi=0$ on $[2,\infty)$ and set $\chi_R(x)=\chi(R^{-1}|x|)$. For arbitrary $\rho\in C^{\infty}_{c}[0,T)$, substituting $\varphi=\chi_{R} \rho$ into the above and letting $R\to\infty$ imply that 

\begin{align*}
\int_{0}^{T}\dot{\rho}(t)\left(\int_{\mathbb{R}^{3}}\Phi_{+}^{2}\dd x\right)\dd t&=-\rho(0) \int_{\mathbb{R}^{3}}\Phi_{0,+}^{2}\dd x.
\end{align*}\\
Thus (7.8) holds for a.e. $t\in [0,T]$.
\end{proof}

\subsection{Global convergence}

We strengthen the local convergence (7.6) to the following global convergence (7.9) by using the equalities (6.8) and (7.8). We then adjust it to the desired form by the isometries (3.12) to demonstrate generalized magnetic helicity conservation at weak ideal limits.

\begin{prop}
\begin{align}
\Phi_{j,+}\to \Phi_{+}\quad \textrm{in}\ L^{2}(0,T; L^{2}(\mathbb{R}^{3})). 
\end{align}
\end{prop}

\begin{proof}
By (3.19), (3.2), and (6.1), $\Phi_{j,+}$ is uniformly bounded in $L^{\infty}(0,T; L^{2}(\mathbb{R}^{3}) )$. By choosing a subsequence, $\Phi_{j,+}\rightharpoonup \Phi_{+}$ in $L^{2}(0,T; L^{2}(\mathbb{R}^{3}) )$. By (6.8) and (7.8),  

\begin{align*}
T\int_{\mathbb{R}^{3}}\Phi_{0,+}^{2}\dd x=\int_{0}^{T}\int_{\mathbb{R}^{3}}\Phi_+^{2}\dd x\dd t \leq \liminf_{j\to\infty}\int_{0}^{T}\int_{\mathbb{R}^{3}}\Phi_{j,+}^{2}\dd x\dd t\leq T\int_{\mathbb{R}^{3}}\Phi_{0,+}^{2}\dd x.
\end{align*}\\
Thus $\lim_{j\to \infty}||\Phi_{j,+}||_{L^{2}(0,T; L^{2}(\mathbb{R}^{3}) )}=||\Phi_{+}||_{L^{2}(0,T; L^{2}(\mathbb{R}^{3}) )}$ and (7.9) holds.
\end{proof}

\begin{prop}
\begin{align}
\Phi_{j,+}\to \Phi_{+}\quad \textrm{in}\ L^{7}(0,T;  L^{2}(\mathbb{R}^{2}_{+};r^{-1})).  
\end{align}
\end{prop}

\begin{proof}
The function $\Phi_{j,+}-\Phi_{+}$ is supported in $\{x\in \mathbb{R}^{3}\ |\ \Phi_j(x)>0\}\cup \{x\in \mathbb{R}^{3}\ |\ \Phi(x)>0\}$, and  

\begin{align*}
|\Phi_{j,+}-\Phi_{+}|
&= |\Phi_{j,+}-\Phi_{+}|(1_{(0,\infty)}(\Phi_j)+1_{(0,\infty)}(\Phi)-1_{(0,\infty)}(\Phi_j)1_{(0,\infty)}(\Phi) ) \\
&\leq  |\Phi_{j,+}-\Phi_{+}|(1_{(0,\infty)}(\Phi_j)+1_{(0,\infty)}(\Phi) ).
\end{align*}\\
By (3.17), (3.2) and (6.1), $1_{(0,\infty)}(\Phi_j)$ is uniformly bounded in $L^{\infty}(0,T; L^{2}(\mathbb{R}^{3}) )$. The convergence (7.9) implies that 

\begin{align*}
\Phi_{j,+}\to \Phi_{+}\quad  \textrm{in}\ L^{2}(0,T; L^{1}(\mathbb{R}^{3})).
\end{align*}\\
By the isometry $(3.12)_1$ for $\varphi_j=\phi_j/r^{2}$ and $\varphi=\phi/r^{2}$,

\begin{align*}
\left(\varphi_j-1-\frac{\gamma}{r^2}\right)_{+}\to \left(\varphi-1-\frac{\gamma}{r^2}\right)_{+} \quad  \textrm{in}\ L^{2}(0,T; L^{1}(\mathbb{R}^{5})).
\end{align*}\\
By (3.11), (3.2), and (6.1), $\varphi_j$ is uniformly bounded in $L^{\infty}(0,T; \dot{H}^{1}(\mathbb{R}^{5}))$. By $(\varphi_j-1-\gamma/r^{2})_{+}\leq |\varphi_j|$ and $\dot{H}^{1}(\mathbb{R}^{5})\subset L^{10/3}(\mathbb{R}^{5})$, $(\varphi_j-1-\gamma/r^{2})_{+}$ is uniformly bounded in $L^{\infty}(0,T; L^{10/3}(\mathbb{R}^{5}) )$. By H\"older's inequality,  

\begin{align*}
||\varsigma||_{L^{7}(0,T; L^{2}(\mathbb{R}^{5}) ) } 
\leq ||\varsigma||_{L^{\infty}(0,T; L^{10/3}(\mathbb{R}^{5}) ) } ^{5/7}||\varsigma||_{L^{2}(0,T; L^{1}(\mathbb{R}^{5}) ) } ^{2/7}, 
\end{align*}\\
$\varsigma=(\varphi_j-1-\gamma/r^{2})_{+}-(\varphi-1-\gamma/r^{2})_{+}$ converges to zero in $L^{7}(0,T; L^{2}(\mathbb{R}^{5}) )$. By the isometry $(3.12)_2$, (7.10) holds. 
\end{proof}

\subsection{Helicity conservation}

We now demonstrate generalized magnetic helicity conservation at weak ideal limits for axisymmetric Leray--Hopf solutions.

\begin{thm}
Let $B_{\infty}=-2e_z$. Let $u_{\infty}$ be a constant parallel to $e_z$. Let $\phi_{\infty}=r^{2}+\gamma$ and $\gamma\geq 0$. Suppose that $(v,b)$ is a weak ideal limit of axisymmetric Leray--Hopf solutions to (1.9) for fixed $(v_0,b_0)$. Then, for the Clebsch potentials $\phi,G$ of $b$ and $\Phi=\phi-\phi_{\infty}$, 

\begin{align}
\int_{\mathbb{R}^{3}}\Phi_{+}\frac{G}{r^{2}}\dd x
=\int_{\mathbb{R}^{3}}\Phi_{0,+}\frac{G_0}{r^{2}}\dd x,     
\end{align}\\
for a.e. $t\in [0,T]$.
\end{thm}

\begin{proof}
By the uniform bound (6.14), 

\begin{align*}
&\sup_{0\leq t\leq T}\left|\int_{\mathbb{R}^{3}}\Phi_{j,+}\frac{G_j}{r^{2}}\dd x-\int_{\mathbb{R}^{3}}\Phi_{0,+}\frac{G_0}{r^{2}}\dd x  \right|
\leq \int_{0}^{T}\left|\frac{\dd}{\dd t} \left(\int_{\mathbb{R}^{3}}\Phi_{j,+} \frac{G_j}{r^{2}} \dd x\right)\right| \dd t \\
&\lesssim T^{1/2}\mu^{1/2}_{j}\left(||v_0||_{L^{2}(\mathbb{R}^{3}) }^{2}+||b_0||_{L^{2}(\mathbb{R}^{3}) }^{2}  \right).
\end{align*}\\
Thus 

\begin{align}
\lim_{j\to\infty}\int_{\mathbb{R}^{3}}\Phi_{j,+}\frac{G_j}{r^{2}}\dd x=\int_{\mathbb{R}^{3}}\Phi_{0,+}\frac{G_0}{r^{2}}\dd x\quad \textrm{uniformly in}\ [0,T].  
\end{align}\\
For an arbitrary $\rho\in C_{c}^{\infty}[0,T)$,

\begin{align*}
\int_{0}^{T}\dot{\rho}(t)\left(\int_{\mathbb{R}^{3}}\Phi_{j,+} \frac{G_j}{r^{2}} \dd x \right)\dd t-\int_{0}^{T}\dot{\rho}(t)\left(\int_{\mathbb{R}^{3}}\Phi_{+} \frac{G}{r^{2}} \dd x \right)\dd t
&=\int_{0}^{T}\dot{\rho}(t)\left(\int_{\mathbb{R}^{3}}( \Phi_{j,+}-\Phi_{+}) \frac{G_j}{r^{2}} \dd x \right)\dd t \\
&+\int_{0}^{T}\dot{\rho}(t)\left(\int_{\mathbb{R}^{3}}\Phi_{+}(G_j-G) \frac{1}{r^{2}} \dd x \right)\dd t.
\end{align*}\\
The last term converges to zero as $j\to\infty$ since $G_j\rightharpoonup G$ weakly-star in $L^{\infty}(0,T;L^{2}(\mathbb{R}^{2}_{+};r^{-1}))$ by (6.2) and $\Phi_{+}\in L^{\infty}(0,T; L^{2}(\mathbb{R}^{2}_{+};r^{-1}) )$ by (3.18). By (7.10), 

\begin{align*}
&\left| \int_{0}^{T}\dot{\rho}(t)\left(\int_{\mathbb{R}^{3}}( \Phi_{j,+}-\Phi_{+}) \frac{G_j}{r^{2}} \dd x \right)\dd t \right| \\
&=2\pi \left| \int_{0}^{T}\dot{\rho}(t)\left(\int_{\mathbb{R}^{2}_{+}}( \Phi_{j,+}-\Phi_{+}) \frac{G_j}{r} \dd z\dd r \right)\dd t \right| \\
&\leq 2\pi ||\dot{\rho}||_{L^{7/6}(0,T) }||\Phi_{j,+}-\Phi_{+}||_{L^{7}(0,T; L^{2}(\mathbb{R}^{2}_{+};r^{-1} ) ) }\left(\sup_{j}||G_{j}||_{L^{\infty}(0,T; L^{2}(\mathbb{R}^{2}_{+};r^{-1}) ) }\right)\to 0.
\end{align*}\\
By the uniform convergence (7.12), 

\begin{align*}
\int_{0}^{T}\dot{\rho}(t)\left(\int_{\mathbb{R}^{3}}\Phi_{+}\frac{G}{r^{2}}\dd x\right) \dd t=-\rho(0)\int_{\mathbb{R}^{3}}\Phi_{0,+}\frac{G_0}{r^{2}}\dd x.
\end{align*}\\
Thus (7.11) holds for a.e. $t\in [0,T]$.
\end{proof}

\begin{thm}
Let $B_{\infty}=-2e_z$. Let $u_{\infty}$ be a constant parallel to $e_z$. Let $\phi_{\infty}=r^{2}+\gamma$ and $\gamma\geq 0$. There exists a weak ideal limit $(v,b)$ of axisymmetric Leray--Hopf solutions to (1.9) for $v_0,b_0\in L^{2}_{\sigma,\textrm{axi}}(\mathbb{R}^{3})$ satisfying

\begin{align}
&\int_{\mathbb{R}^{3}}\left(|v|^{2}+|b|^{2}\right) \dd x
\leq \int_{\mathbb{R}^{3}}\left(|v_0|^{2}+|b_0|^{2}\right) \dd x,    \\
&\int_{\mathbb{R}^{3}}(\phi-\phi_{\infty})_{+}\frac{G}{r^{2}}\dd x
=\int_{\mathbb{R}^{3}}(\phi_0-\phi_{\infty})_{+}\frac{G_0}{r^{2}}\dd x,  
\end{align}\\ 
for a.e. $t\in [0,T]$, where $\phi$, $G$ are the Clebsch potentials of $b$.
\end{thm}

\begin{proof}
We take $(\nu_j,\mu_j)$ such that $(\nu_j,\mu_j)\to (0,0)$. By Theorem 6.3, there exists an axisymmetric Leray--Hopf weak solution $(v_j,b_j)$ to (1.9) for $(v_0,b_0)$. By (6.1), there exists a subsequence and $(v,b)$ such that (6.2) holds. The limit $(v,b)$ satisfies (7.13) by the lower semicontinuity of the norm for the weak-star convergence (6.2). The conservation (7.14) follows from Theorem 7.5.
\end{proof}

\begin{rem}
The results in Sections 6 and 7 hold also for $B_{\infty}=-We_{z}$, $\phi_{\infty}=Wr^{2}/2+\gamma$, $W>0$, and $\gamma\geq 0$.
\end{rem}

\section{Orbital stability}
We complete the proof of Theorem 1.4. We first show the stability of a set of minimizers to the variational problem (4.1) with parameters $(h,2,\gamma)$ in weak ideal limits of axisymmetric Leray--Hopf solutions by the compactness result in Theorem 5.6 and the existence result for weak ideal limits in Theorem 7.6. We then extend the result for general parameters $(h, W,\gamma)$. We derive Theorem 1.4 from the uniqueness of the Grad--Shafranov equation (4.3) for $\gamma=0$ and the explicit form of the constant $h_C$ in (1.11).

\subsection{Stability of minimizers}

We denote the suppressed constant $\gamma\geq 0$ explicitly for $H[\cdot]=H_{\gamma}[\cdot]$, $I_{h}=I_{h,\gamma}$ and $S_{h}=S_{h,\gamma}$ in (3.26), (4.1), and (4.2).

\begin{prop} Let $h\in \mathbb{R}$ and $\gamma\geq 0$. Let $\phi_{\infty}=r^{2}+\gamma$ and $B_{\infty}=-2e_{z}$. Let $u_{\infty}$ be a constant parallel to $e_z$. The set $S_{h,\gamma}$ in (4.2) is orbitally stable in weak ideal limits of axisymmetric Leray--Hopf solutions to (1.9) in the sense that for arbitrary $\varepsilon >0$, there exists $\delta >0$ such that for $v_0,b_0\in L^{2}_{\sigma,\textrm{axi}} (\mathbb{R}^{3})$ satisfying 

\begin{align*}
||v_0||_{L^{2}(\mathbb{R}^{3})}
+\inf\left\{  ||b_0-\tilde{b}||_{L^{2}(\mathbb{R}^{3}) }\ \big|\   \tilde{b}\in S_{h,\gamma}\  \right\}+\left|H_{\gamma}[b_0]-h\right| \leq \delta, 
\end{align*}\\
there exists a weak ideal limit $(v,b)$ of axisymmetric Leray--Hopf solutions to (1.9) for $(v_0,b_0)$ such that 

\begin{align*}
\textrm{ess sup}_{t>0} \left(||v||_{L^{2}(\mathbb{R}^{3})}
+\inf\left\{  ||b-\tilde{b}||_{L^{2}(\mathbb{R}^{3}) }\ \big|\   \tilde{b}\in S_{h,\gamma}\  \right\} \right)\leq \varepsilon. 
\end{align*}
\end{prop}

\vspace{5pt}

\begin{proof}
Suppose that the assertion was false. Then, there exists $\varepsilon_0>0$ such that for any $n\geq 1$, there exists $(v_{0,n},b_{0,n})$ satisfying

\begin{align*}
||v_{0,n}||_{L^{2}(\mathbb{R}^{3})}+\inf\left\{  ||b_{0,n}-\tilde{b}||_{L^{2}(\mathbb{R}^{3}) }\ \big|\   \tilde{b}\in S_{h,\gamma}\  \right\}+|H_{\gamma}[b_{0,n}]-h|
\leq \frac{1}{n},
\end{align*}\\
and the weak ideal limit $(v_n,b_n)$ in Theorem 7.6 satisfying 

\begin{align*}
\textrm{ess sup}_{t>0}\left( ||v_n||_{L^{2}(\mathbb{R}^{3})}
+\inf\left\{  ||b_n-\tilde{b}||_{L^{2}(\mathbb{R}^{3}) }\ \big|\   \tilde{b}\in S_{h,\gamma}\  \right\}\right)\geq \varepsilon_0.
\end{align*}\\
We denote by $F_n$ the set of all points $t\in (0,\infty)$ such that 

\begin{align*}
&{\mathcal{E}}[v_n,b_n]\leq {\mathcal{E}}[v_{0,n},b_{0,n}], \\
&H_{\gamma}[b_n]=H_{\gamma}[b_{0,n}].
\end{align*}\\
The set $F_n^{c}$ has measure zero. For $F=\cap_{n=1}^{\infty} F_n$, $F^{c}$ has measure zero, and the above inequality and equality hold for all  $t\in F$ and $n\geq 1$. We take a point $t_n\in F$ such that 

\begin{align*}
||v_n||_{L^{2}(\mathbb{R}^{3})}(t_n)
+\inf\left\{  ||b_n-\tilde{b}||_{L^{2}(\mathbb{R}^{3}) }\ \big|\   \tilde{b}\in S_{h,\gamma} \right\}(t_n)
\geq \frac{\varepsilon_0}{2}>0,
\end{align*}\\
and write $(v_n,b_n)=(v_n,b_n)(\cdot,t_n)$ by suppressing $t_n$. 

For $h_n=H_{\gamma}[b_{0,n}]$,

\begin{align*}
I_{h_n,\gamma}\leq \frac{1}{2}\int_{\mathbb{R}^{3}}|b_{0,n}|^{2}\dd x\leq \frac{1}{2}\left(\inf\left\{  ||b_{0,n}-\tilde{b}||_{L^{2}(\mathbb{R}^{3}) }\ \big|\   \tilde{b}\in S_{h,\gamma} \right\}+\sqrt{2}I_{h,\gamma}^{1/2}  \right)^{2}.
\end{align*}\\
By the lower semi-continuity (4.17), letting $n\to\infty$ implies that 

\begin{align*}
H_{\gamma}[b_{0,n}]\to h,\quad {\mathcal{E}}[v_{0,n},b_{0,n}]\to I_{h,\gamma}.
\end{align*}\\
By helicity conservation and non-increasing total energy,

\begin{align*}
&H_{\gamma}[b_n]=H_{\gamma}[b_{0,n}]=h_n,\\
&I_{h_n,\gamma}\leq {\mathcal{E}}[v_n,b_n]\leq {\mathcal{E}}[v_{0,n},b_{0,n}]. 
\end{align*}\\
Letting $n\to\infty$ implies that

\begin{align*}
H_{\gamma}[b_n]\to h,\quad {\mathcal{E}}[v_n,b_n]\to I_{h,\gamma}.
\end{align*}\\
By Theorem 5.6, there exists $\{n_k\}$, $\{z_{k}\}$ and some $b\in S_{h,\gamma}$ such that 

\begin{align*}
(v_{n_k},b_{n_k}(\cdot +z_{k}e_z))\to (0,b) \quad \textrm{in}\ L^{2}(\mathbb{R}^{3}).
\end{align*}\\
Thus 

\begin{align*}
0&=\lim_{k\to\infty}\left\{ ||v_{n_k}||_{L^{2}(\mathbb{R}^{3})}+||b_{n_k}(\cdot+z_{k}e_z)-b||_{L^{2}(\mathbb{R}^{3})}   \right\} \\
&\geq\liminf_{k\to\infty}\left( ||v_{n_k}||_{L^{2}(\mathbb{R}^{3})}+\inf\left\{\ ||b_{n_k}-\tilde{b}||_{L^{2}(\mathbb{R}^{3})}  \ \middle|\ \tilde{b}\in S_{h,\gamma}  \right\}\right)
\geq \frac{\varepsilon_0}{2}>0.
\end{align*}\\
We obtained a contradiction. The proof is complete.
\end{proof}

\begin{thm}[Stability of nonlinear force-free fields with discontinuous factors] Let $h\in \mathbb{R}$, $W>0$ and $\gamma\geq 0$. Let $\phi_{\infty}=Wr^{2}/2+\gamma$ and $B_{\infty}=-We_{z}$. Let $u_{\infty}$ be a constant parallel to $e_z$. The set of minimizers $S_{h,W,\gamma}$ to $I_{h,W,\gamma}$ in (1.13) is orbitally stable in weak ideal limits of axisymmetric Leray--Hopf solutions to (1.9) in the sense that for arbitrary $\varepsilon >0$ there exists $\delta >0$ such that for $v_0,b_0\in L^{2}_{\sigma,\textrm{axi}} (\mathbb{R}^{3})$ satisfying 

\begin{align*}
||v_0||_{L^{2}(\mathbb{R}^{3})}
+\inf\left\{  ||b_0-\tilde{b}||_{L^{2}(\mathbb{R}^{3}) }\ \big|\   \tilde{b}\in S_{h,W,\gamma}\  \right\}+\left|2\int_{\mathbb{R}^{3}}(\phi_0-\phi_{\infty})_+\frac{G_0}{r^{2}}\dd x -h\right| \leq \delta, 
\end{align*}\\
for the Clebsch potentials $\phi_0,G_0$ of $b_0$, there exists a weak ideal limit $(v,b)$ of axisymmetric Leray--Hopf solutions to (1.9) for $(v_0,b_0)$ such that 

\begin{align*}
\textrm{ess sup}_{t>0}\left( ||v||_{L^{2}(\mathbb{R}^{3})}
+\inf\left\{  ||b-\tilde{b}||_{L^{2}(\mathbb{R}^{3}) }\ \big|\   \tilde{b}\in S_{h,W,\gamma}\  \right\}\right) \leq \varepsilon. 
\end{align*}
\end{thm}

\vspace{5pt}

\begin{proof}
The assertion follows from the same argument as the proof of Proposition 8.1 for $W=2$ by the compactness of minimizing sequences to $I_{h, W,\gamma}$ (Remark 5.7) and the existence of weak ideal limits conserving generalized magnetic helicity (Remark 7.7).   
\end{proof}

\subsection{Uniqueness}

For given constants $(W,\lambda)$ and the constant $h_C$ in (1.11), we show that the set of minimizers $S_{h_C,W,0}$ is translations of $U_C-B_{\infty}$ for the explicit solution $U_C$ in (1.6). The following uniqueness result is due to Fraenkel \cite[Theorem 4]{Fra92}, \cite[Exercise 4.23]{Fra00}.

\begin{thm}
Let $W>0$, $\kappa\in \mathbb{R}$ and $\phi_{\infty}=Wr^{2}/2$. Let $\phi\in \dot{H}^{1}_{0}(\mathbb{R}^{2}_{+};r^{-1})$ be a weak solution to (4.3). There exists $z_0\in \mathbb{R}$ such that $\phi-\phi_{\infty}=\Phi_{C}(\cdot+ z_0e_{z})$ for $\Phi_C$ with $(W,\kappa^{2})$ in (1.6).\\ 
\end{thm}

The proof of Theorem 8.3 is based on the moving plane method. As proved in the proof of Proposition 4.4, $\varphi=\phi/r^{2}$ is a continuous decaying solution to 

\begin{align}
-\Delta_{y} \varphi=\kappa^{2}\left(\varphi-\frac{W}{2} \right)_{+}\quad \textrm{in}\ \mathbb{R}^{5}.
\end{align}\\
The decay implies that the set $\{y\in \mathbb{R}^{5}\ |\ \varphi>W/2 \}$ is compact in $\mathbb{R}^{5}$. Therefore, $\varphi$ is expressed in terms of the Newton potential,  

\begin{align*}
\varphi(y)=\frac{1}{8\pi^{2}}\int_{\{\varphi>W/2\}}\frac{\kappa^{2}}{|y-w|^{3}}\left(\varphi-\frac{W}{2} \right)_{+}\dd w.
\end{align*}\\
This potential representation implies that $\varphi$ is a positive solution to (8.1) satisfying an admissible asymptotic behavior as $|y|\to\infty$ which enables one to apply the moving plane method \cite[Theorem 4.2]{Fra00} to deduce that a translation of $\varphi$ in $y_1$-direction is radially symmetric and decreasing.

The symmetry of $\varphi$ implies the explicit form (1.6) for $\lambda=\kappa^{2}$ \cite{Moffatt}, cf. \cite{Fra92}, \cite{Fra00}. The equation of $\varphi=\varphi(\rho)$ for $\rho=\sqrt{z^{2}+r^{2}}$ is 

\begin{align*}
-\left(\partial_{\rho}^{2}+\frac{4}{\rho}\partial_{\rho} \right)\varphi=\kappa^{2}\left(\varphi-\frac{W}{2}\right)_{+},\quad \rho>0. 
\end{align*}\\
Since $\varphi$ is decreasing, there exists a unique $R>0$ such that $\varphi(R)=W/2$. The function $\tilde{\varphi}=\varphi-W/2$ satisfies 

\begin{align*}
-\left(\partial_{\rho}^{2}+\frac{4}{\rho}\partial_{\rho} \right)\tilde{\varphi}=\kappa^{2}\tilde{\varphi}_{+},\quad \rho>0. 
\end{align*}\\
The scaled $\hat{\varphi}=\tilde{\varphi}(\rho/|\kappa|)$ satisfies $-(\partial_{\rho}^{2}+4\rho^{-1}\partial_{\rho} )\hat{\varphi}=\hat{\varphi}$ and $\hat{\varphi}>0$ for $0<\rho<R_0=|\kappa| R$ and $\hat{\varphi}(R_0)=0$. By the transform $\iota=\rho^{3/2}\hat{\varphi}$, this equation is reduced to the Bessel's differential equation, 

\begin{align*}
\iota''+\frac{1}{\rho}\iota'+\left(1-\frac{(3/2)^{2} }{\rho^{2}}\right)\iota=0,\quad \iota&>0, \quad 0<\rho<R_0,\\
\iota(R_0)&=0.
\end{align*}\\
By the boundedness of $\iota$ at $\rho=0$, $\iota$ is the $3/2$-th order Bessel function of the first kind with some constant $C_1$, i.e., 

\begin{align*}
\iota&=C_1J_{3/2}(\rho),\\
R_0&=c_{3/2}.
\end{align*}\\
The function $\tilde{\varphi}$ is harmonic for $\rho>R$ and expressed as $\tilde{\varphi}=C_2+C_3/\rho^{3}$ for $\rho>R$ with $C_2=-W/2$ and  $C_3=WR^{2}/2$ by the boundary conditions $\tilde{\varphi }\to -W/2$ as $\rho\to\infty$ and $\tilde{\varphi}(R)=0$. By continuity of $\partial_{\rho}\tilde{\varphi}$ at $\rho=R$ and $\dot{J}_{3/2}(c_{3/2})=-J_{5/2}(c_{3/2})$, 

\begin{align*}
C_1=\frac{3}{2}W\frac{c_{3/2}^{1/2} }{|\kappa|^{3}J_{5/2}(c_{3/2})   }.
\end{align*}\\
This implies the explicit form (1.6) with $(W,\kappa^{2})$, i.e., $\Phi=\Phi_{C}$, since $\tilde{\varphi}=\Phi/r^{2}$ for $\Phi=\phi-\phi_{\infty}$.\\

Moffatt \cite[p.128]{Moffatt} computed the constant $h_C$ in (1.11). We write it with our symbols.

\begin{prop}
The constant $h_C$ in (1.11) is  

\begin{align}
h_C=\left(\frac{W}{\lambda}\right)^{2}\frac{12\pi c_{3/2}}{J_{5/2}^{2}(c_{3/2}) }\int_{0}^{c_{3/2}}\rho J_{3/2}^{2}(\rho)\dd \rho.
\end{align}
\end{prop}

\begin{proof}
By (1.11), for $\phi=\Phi_C+Wr^{2}/2$ and $\varphi=\phi/r^{2}$, 

\begin{align*}
h_C
=2\lambda^{1/2}\int_{\mathbb{R}^{3}}\Phi_{C,+}^{2}\frac{1}{r^{2}}\dd x
=2\lambda^{1/2}\int_{\mathbb{R}^{3}}\left(\phi-\frac{W}{2}r^{2} \right)_{+}^{2}\frac{1}{r^{2}}\dd x
&=\frac{2\lambda^{1/2}}{\pi}\int_{\mathbb{R}^{5}}\left(\varphi-\frac{W}{2} \right)_{+}^{2}\dd y\\
&=\frac{2}{\pi\lambda^{2}}\int_{\mathbb{R}^{5}}\left(\varphi\left(\frac{w}{\lambda^{1/2}}\right)-\frac{W}{2} \right)_{+}^{2}\dd w.
\end{align*}\\
The function $\varphi_1(w)=\varphi(w/\lambda^{1/2})$ is a solution to (8.1) for $\kappa=1$. By $\varphi_1(w)-W/2=C_1J_{3/2}(\rho)\rho^{-3/2}$ for $\rho=|w|<c_{3/2}$ with $C_1=3Wc_{3/2}^{1/2}/ (2J_{5/2}(c_{3/2}))$ and 
 
\begin{align*}
\int_{\mathbb{R}^{5}}\left(\varphi_1(w)-\frac{W}{2} \right)_{+}^{2}\dd w
=\frac{8\pi^{2}}{3}C_1^{2}\int_{0}^{c_{3/2}}\rho J_{3/2}^{2}(\rho)\dd \rho,
\end{align*}\\
the constant $h_C$ is given by (8.2).
\end{proof}

\begin{thm}
Let $W>0$ and $\lambda>0$. Let $U_C$ be as in (1.6). Let $h_C>0$ be as in (1.11). Let $B_{\infty}=-We_z$. Let $S_{h_C,W,0}$ be as in Theorem 8.2. Then, 

\begin{align}
S_{h_C,W,0}=\left\{ U_{C}(\cdot +ze_z)-B_{\infty}\ \big|\ z\in \mathbb{R}\  \right\}.   
\end{align}
\end{thm}

\vspace{5pt}

\begin{proof}
We take an arbitrary $b=\nabla \times (\phi\nabla \theta)+G\nabla \theta\in S_{h_C,W,0}$. By Proposition 4.2, $G=\kappa(\phi-\phi_{\infty})_{+}$ for $\phi_{\infty}=Wr^{2}/2$ and some $\kappa>0$, and $\phi$ is a weak solution to (4.3). By Theorem 8.3, $\phi-\phi_{\infty}$ is a translation of the explicit solution (1.6) with parameters $W>0$ and $\tilde{\lambda}=\kappa^{2}$. By (8.2), generalized magnetic helicity of $b$ is $C(W/\tilde{\lambda})^{2}$ for some exact constant $C$. Since $b\in S_{h_C,W,0}$, generalized magnetic helicity of $b$ is $h_C=C(W/\lambda)^{2}$. Thus $\tilde{\lambda}=\lambda$ and $\phi-\phi_{\infty}=\Phi_C(\cdot+z_0e_z)$, $G=G_{C}(\cdot +z_0e_z)$ for $\Phi_C$ and $G_C$ in (1.11) and some $z_0\in \mathbb{R}$. By $B_{\infty}=-\nabla \times (\phi_{\infty}\nabla \theta)$, 

\begin{align*}
U_C=\nabla \times (\Phi_C\nabla \theta)+G_C\nabla \theta
=b(\cdot -z_0e_z)+B_{\infty}.
\end{align*}\\
We proved that $S_{h_C,W,0}\subset \left\{ U_{C}(\cdot +ze_z)-B_{\infty}\ \big|\ z\in \mathbb{R}\  \right\}$. The translations of $b=U_C(\cdot+z_0e_z)-B_{\infty}$ are also elements of $S_{h_C,W,0}$. Thus, the equality (8.3) holds.
\end{proof}

\begin{proof}[Proof of Theorem 1.4]
By Theorem 8.5 for $h=h_{C}$ and $\gamma=0$, 

\begin{align*}
\inf\left\{||b_0-\tilde{b}||_{L^{2}(\mathbb{R}^{3})}\ |\ \tilde{b}\in S_{h_C,W,0}\  \right\}
&=\inf\left\{||b_0-\tilde{b}||_{L^{2}(\mathbb{R}^{3})}\ \middle|\ \tilde{b}=U_C(\cdot +ze_z)-B_{\infty},\  z\in \mathbb{R} \right\} \\
&=\inf\left\{||b_0+B_{\infty}-U_C(\cdot +ze_z)||_{L^{2}(\mathbb{R}^{3})}\ \middle|\   z\in \mathbb{R} \right\}.
\end{align*}\\
The assertion follows from Theorem 8.2.
\end{proof}

\begin{rem}
The existence and the stability of axisymmetric nonlinear force-free fields with continuous factors $f\in C(\mathbb{R}^{3})$ may be studied by replacing $g(s)=2s_{+}$ of generalized magnetic helicity in (6.5) with sufficiently regular $g(s)$, e.g., $g(s)=2s_+^{\alpha}$, $\alpha>1$, cf. \cite{A8}.
\end{rem}

\appendix

\section{Existence of axisymmetric Leray--Hopf solutions}

We demonstrate the existence of axisymmetric Leray--Hopf solutions to (1.9) stated in Theorem 6.3. The existence of weak solutions to viscous and resistive MHD in $\mathbb{R}^{3}$ constructed by Leray's method can be found in \cite{FJ21}. We construct Leray--Hopf solutions to the system (1.9) with constants $u_{\infty}, B_{\infty}\in \mathbb{R}^{3}$.  

\begin{thm}
Let $u_{\infty}, B_{\infty}\in \mathbb{R}^{3}$. For $v_0,b_0\in L^{2}_{\sigma}(\mathbb{R}^{3})$, there exists a Leray--Hopf solution to (1.9). 
\end{thm}

We take $\chi\in C^{\infty}_{c}(\mathbb{R}^{3})$ such that $\int_{\mathbb{R}^{3}}\chi \dd x=1$ and set $\chi_{\varepsilon}(x)=\varepsilon^{-3}\chi(\varepsilon^{-1}x)$ for $\varepsilon>0$. We look at the approximated system of (1.9): 

\begin{equation}
\begin{aligned}
v_t+(w+u_{\infty})\cdot \nabla v+\nabla p&=(d+B_{\infty})\cdot \nabla b+\nu \Delta v,\\
b_t+(w+u_{\infty})\cdot \nabla b+\nabla q&=(d+B_{\infty})\cdot \nabla v+\mu \Delta b,\\
\nabla \cdot v=\nabla \cdot b&=0, \\
w=\chi_{\varepsilon}*v,\ d&=\chi_{\varepsilon}*b.
\end{aligned}
\end{equation}\\
For solutions to (A.1), the energy equality 

\begin{align}
\frac{1}{2}\int_{\mathbb{R}^{3}}\left( |v|^{2}+|b|^{2}\right)\dd x
+\int_{0}^{t}\int_{\mathbb{R}^{3}}\left(\nu |\nabla v|^{2}+\mu |\nabla b|^{2}\right)\dd x\dd s
= \frac{1}{2}\int_{\mathbb{R}^{3}}\left( |v_0|^{2}+|b_0|^{2}\right)\dd x,
\end{align}\\
holds for all $t\in [0,T]$. For simplicity of notation, we apply block matrices

\begin{align}
\varphi=
\begin{bmatrix}
v \\
b \\
\end{bmatrix},\quad
F=
\begin{bmatrix}
(d+B_{\infty})\otimes b-(w+u_{\infty})\otimes v \\
(d+B_{\infty})\otimes v-(w+u_{\infty})\otimes b \\
\end{bmatrix},\quad
{\mathcal{L}}=
\begin{bmatrix}
\nu & 0 \\
0 & \mu \\
\end{bmatrix}
\Delta.
\end{align}\\  
We denote $v,b\in L^{2}_{\sigma}(\mathbb{R}^{3})$ by $\varphi={}^{t}[v,b]\in L^{2}_{\sigma}(\mathbb{R}^{3})$. In terms of the projection operator $\mathbb{P}$, associated with (2.1), the  system (A.1) can be expressed as 

\begin{equation}
\begin{aligned}
\varphi_t&={\mathcal{L}}\varphi +\mathbb{P}\nabla \cdot  F\quad \textrm{on}\ L^{2}_{\sigma}(\mathbb{R}^{3}).\\
\end{aligned}
\end{equation}\\
We denote the heat semigroup by 

\begin{align*}
S(t)=
\begin{bmatrix}
e^{t\nu\Delta} & 0 \\
0 & e^{t\mu\Delta} \\
\end{bmatrix}
.
\end{align*}\\  
The integral form of (A.4) is 

\begin{equation}
\begin{aligned}
\varphi&=S(t)\varphi_0+\int_{0}^{t}S(t-s)\mathbb{P}\nabla \cdot F(s)\dd s.
\end{aligned}
\end{equation}\\
We construct solutions of (A.5) in the energy space $C([0,T]; L^{2}(\mathbb{R}^{3}))\cap L^{2}(0,T; H^{1}(\mathbb{R}^{3}) )$ equipped with the norm $||\varphi||_{T}=||\varphi||_{L^{\infty}(0,T; L^{2}(\mathbb{R}^{3}) )}+||\varphi||_{L^{2}(0,T; \dot{H}^{1}(\mathbb{R}^{3}) )}$. By integration by parts, $S(t)$ satisfies the linear estimates 

\begin{equation}
\begin{aligned}
||S(t)\varphi_0||_{T}&\leq C||\varphi_0||_{L^{2}(\mathbb{R}^{3}) },\\
\left\|\int_{0}^{t}S(t-s)\mathbb{P}\nabla \cdot F_0(s)\dd s\right\|_T&\leq C'||F_0||_{L^{2}(0,T; L^{2}(\mathbb{R}^{3})) }.
\end{aligned}
\end{equation}\\
The constants $C$ and $C'$ are independent of $T$. For $F_0$ satisfying $\nabla \cdot F_0\in L^{2}(0,T; L^{2}(\mathbb{R}^{3}))$, 

\begin{align}
\nabla^{2}\int_{0}^{t}S(t-s)\mathbb{P}\nabla \cdot F_0\dd s \in L^{2}(0,T; L^{2}(\mathbb{R}^{3})  ),
\end{align}\\
by the maximal regularity, e.g., \cite[1.6.2 Lemma]{Sohr}. Solutions of the integral equation (A.5) in the energy space satisfies (A.4) for a.e. $t\in [0,T]$ since $\nabla \cdot F \in L^{2}(0,T; L^{2}(\mathbb{R}^{3}))$ by the estimate of the mollifier,

\begin{align}
||\chi_{\varepsilon}*\varphi ||_{L^{\infty}(\mathbb{R}^{3}) }\leq \frac{C}{\varepsilon^{3/2}} ||\varphi||_{L^{2}(\mathbb{R}^{3})}.
\end{align}

\vspace{5pt}

\begin{prop}
Let $T>0$ and $\varepsilon>0$. For $\varphi_0\in L^{2}_{\sigma}(\mathbb{R}^{3})$, there exists a unique solution $\varphi\in C([0,T]; L^{2}_{\sigma}(\mathbb{R}^{3}) )\cap L^{2}(0,T; H^{1}(\mathbb{R}^{3}))$ of (A.5). The solution $\varphi={}^{t}[v,b]$ satisfies 

\begin{align}
\varphi_t, \nabla^{2}\varphi \in L^{2}_{\textrm{loc}}(0,T; L^{2}(\mathbb{R}^{2})),
\end{align}\\
and (A.2) holds for all $t\in [0,T]$.
\end{prop}

\begin{proof}
We set $\varphi_1=S(t)\varphi_0$ and $\varphi_{j+1}$ for $j\geq 1$ by 

\begin{equation*}
\begin{aligned}
\varphi_{j+1}&=S(t)\varphi_0+\int_{0}^{t}S(t-s)\mathbb{P}\nabla \cdot F_j(s)\dd s.
\end{aligned}
\end{equation*}\\
for $F_j$ defined by $\varphi_j$ and (A.3). By (A.8), 

\begin{align*}
||F_j||_{L^{2}(0,T; L^{2}(\mathbb{R}^{3}) ) }\leq CT^{1/2}\left(\frac{1}{\varepsilon^{3/2}}K_j+1\right)K_j,
\end{align*}
for $K_j=||\varphi_j||_{T}$. By (A.6),

\begin{align*}
K_{j+1}\leq  K_1+CT^{1/2}\left(\frac{1}{\varepsilon^{3/2}}K_j+1\right)K_j.
\end{align*}\\
The constant $C$ is independent of $T$ and $j$. Thus for some $T_0<T$, $\varphi_j\in C([0,T_0]; L^{2}_{\sigma}(\mathbb{R}^{3}) )\cap L^{2}(0,T_0; H^{1}(\mathbb{R}^{3}))$ is uniformly bounded. By estimating $\varphi_{j+1}-\varphi_j$ in a similar way, for some $T_1<T_0$, 

\begin{align*}
\varphi_j\to \varphi\quad  \textrm{in}\ C([0,T_1]; L^{2}(\mathbb{R}^{3}) )\cap L^{2}(0,T_1; H^{1}(\mathbb{R}^{3})),
\end{align*}\\
and $\varphi$ is a unique solution to (A.5). By (A.7), $\varphi$ satisfies (A.9) and (A.4) for a.e. $t\in [0,T_1]$. By multiplying $\varphi$ by (A.4) and integration by parts, for $\delta \leq t\leq T_1$,

\begin{align*}
\frac{1}{2}\int_{\mathbb{R}^{3}}\left( |v|^{2}+|b|^{2} \right)\dd x
+\int_{\delta }^{t}\int_{\mathbb{R}^{3}}\left(\nu |\nabla v|^{2}+\mu |\nabla b|^{2}\right)\dd x\dd s
= \frac{1}{2}\int_{\mathbb{R}^{3}}\left( |v|^{2}+|b|^{2}\right)(\delta)\dd x.
\end{align*}\\
The solution $\varphi$ is strongly continuous at $t=0$. Letting $\delta\to0$ implies (A.2) for $t\in [0,T_1]$. By the energy equality (A.2), the local-in-time solution is extendable, and (A.2) holds for all $t\in [0,T]$. 
\end{proof}

\begin{prop}
The solution $\varphi$ in Proposition A.2 satisfies 

\begin{align}
\sup_{\varepsilon>0}||\varphi_t||_{L^{4/3}(0,T; H^{1}(\mathbb{R}^{3})^{*} ) }<\infty. 
\end{align}
\end{prop}

\vspace{5pt}

\begin{proof}
By $L^{\infty}(0,T; L^{2}(\mathbb{R}^{3}) )\cap L^{2}(0,T; \dot{H}^{1}(\mathbb{R}^{3}) )\subset L^{8/3}(0,T; L^{4}(\mathbb{R}^{3}) )$ and (A.2), $F$ in (A.3) is uniformly bounded in $L^{4/3}(0,T; L^{2}(\mathbb{R}^{3}))$. By multiplying functions in $H^{1}(\mathbb{R}^{3})$ by (A.4) and integration by parts,

\begin{align*}
||\varphi_t||_{H^{1}(\mathbb{R}^{3})^{*} }\leq C( ||\nabla \varphi||_{L^{2}(\mathbb{R}^{3}) }+||F||_{L^{2}(\mathbb{R}^{3}) }).
\end{align*}\\
Thus (A.10) holds. 
\end{proof}

\begin{proof}[Proof of Theorem A.1]
We take $\{\varepsilon_j\}$ such that $\varepsilon_j \to0$ as $j\to\infty$ and denote solutions in Proposition A.2 by $\varphi_j$. By $H^{1}_0(B(0,R))^{*}\supset H^{1}(\mathbb{R}^{3})^{*}$ for $R>0$, (A.2), (A.10), Lemma 2.16 and a diagonal argument, there exists a subsequence and $\varphi={}^{t}[v,b]$ such that 

\begin{align*} 
\nabla \varphi_j&\rightharpoonup  \nabla \varphi \quad \textrm{in}\ L^{2}(0,T; L^{2}(\mathbb{R}^{3})), \\
\varphi_j&\to  \varphi \quad \textrm{in}\ L^{2}(0,T; L^{2}_{\textrm{loc}}(\mathbb{R}^{3})), \\
\partial_t \varphi_j&\rightharpoonup  \partial_t \varphi \quad \textrm{in}\ L^{4/3}(0,T; H^{1}(\mathbb{R}^{3})^{*}),\\ 
\varphi_j&\stackrel{\ast}{\rightharpoonup}  \varphi \quad \textrm{in}\ L^{\infty}(0,T; L^{2}(\mathbb{R}^{3})).
\end{align*}\\
By $\varphi_t\in L^{4/3}(0,T; H^{1}(\mathbb{R}^{3})^{*} )$ and $\varphi\in L^{\infty}(0,T; L^{2}(\mathbb{R}^{3}) )$, $\varphi\in C_w([0,T]; L^{2}(\mathbb{R}^{3}) )$ follows \cite[Ch.III, Lemma 1.4]{Te}. By the lower semi-continuity for the weak-star convergence in $L^{\infty}(0,T; L^{2}(\mathbb{R}^{3}) ) $ and the weak continuity of the limit, the energy inequality (6.1) holds for all $t\in [0,T]$. Since $(L^{2}_{\sigma}\cap H^{1})(\mathbb{R}^{3})^{*}\supset H^{1}(\mathbb{R}^{3})^{*}$, $\varphi_t \in L^{1}(0,T; L^{2}_{\sigma}\cap H^{1}(\mathbb{R}^{3})^{*})$ follows.

It remains to show that $\varphi$ satisfies $\varphi(\cdot,0)={}^{t}[v_0,b_0]$ and (1.9) in the sense of Definition 6.1. We take arbitrary $\rho\in C^{\infty}_{c}[0,T)$ and $\xi,\zeta\in C_{c,\sigma}^{\infty}(\mathbb{R}^{3})$. Since $\varphi_j$ satisfies (A.4) for a.e. $t\in [0,T]$, by multiplying ${}^{t}[\rho\xi,  \rho \zeta]$ by $(A.4)$ and integration by parts, 

\begin{align*}
-\int_{0}^{T}\int_{\mathbb{R}^{3}}v_j\cdot \xi \dot{\rho}\dd x\dd t-\rho(0)\int_{\mathbb{R}^{3}}v_0\cdot \xi \dd x&+\int_{0}^{T}\int_{\mathbb{R}^{3}}((w_j+u_{\infty} )\cdot \nabla  v_j-(d_j+B_{\infty})\cdot \nabla b_j )\cdot \xi \rho\dd x\dd t \\
&+\nu \int_{0}^{T}\int_{\mathbb{R}^{3}}\nabla v_j: \nabla \xi\rho\dd x\dd t=0, \\
-\int_{0}^{T}\int_{\mathbb{R}^{3}}b_j\cdot \zeta \dot{\rho}\dd x\dd t-\rho(0)\int_{\mathbb{R}^{3}}b_0\cdot \zeta \dd x&+\int_{0}^{T}\int_{\mathbb{R}^{3}}((w_j+u_{\infty})\cdot \nabla b_j - (d_j+B_{\infty})\cdot \nabla  v_j)\cdot \zeta \rho\dd x\dd t \\
&+\mu \int_{0}^{T}\int_{\mathbb{R}^{3}}\nabla\times b_j\cdot \nabla\times  \zeta\rho\dd x\dd t=0.
\end{align*}\\
By using the strong convergence of $\varphi_j$ in $L^{2}(0,T; L^{2}_{\textrm{loc}}(\mathbb{R}^{3}))$ and letting $j\to\infty$,

\begin{align*}
-\int_{0}^{T}\int_{\mathbb{R}^{3}}v\cdot \xi \dot{\rho}\dd x\dd t-\rho(0)\int_{\mathbb{R}^{3}}v_0\cdot \xi \dd x&+\int_{0}^{T}\int_{\mathbb{R}^{3}}((v+u_{\infty} )\cdot \nabla  v-(b+B_{\infty})\cdot \nabla b )\cdot \xi \rho\dd x\dd t \\
&+\nu \int_{0}^{T}\int_{\mathbb{R}^{3}}\nabla v: \nabla \xi\rho\dd x\dd t=0, \\
-\int_{0}^{T}\int_{\mathbb{R}^{3}}b\cdot \zeta \dot{\rho}\dd x\dd t-\rho(0)\int_{\mathbb{R}^{3}}b_0\cdot \zeta \dd x&+\int_{0}^{T}\int_{\mathbb{R}^{3}}((v+u_{\infty})\cdot \nabla b - (b+B_{\infty})\cdot \nabla  v)\cdot \zeta \rho\dd x\dd t \\
&+\mu \int_{0}^{T}\int_{\mathbb{R}^{3}}\nabla\times b\cdot \nabla\times  \zeta\rho\dd x\dd t=0.
\end{align*}\\
We extend $\xi,\zeta$ to functions in $L^{2}_{\sigma}\cap H^{1}(\mathbb{R}^{3})$ by density of $C_{c,\sigma}^{\infty}(\mathbb{R}^{3})$ in $L^{2}_{\sigma}\cap H^{1}(\mathbb{R}^{3})$. Since $\rho\in C^{\infty}_{c}[0,T)$ is arbitrary, $\varphi (\cdot,t)\rightharpoonup \varphi(\cdot,0)={}^{t}[v_0,b_0]$ in $L^{2}(\mathbb{R}^{3})$ as $t\to0$. Thus $\varphi (\cdot,t)\to \varphi(\cdot,0)$ in $L^{2}(\mathbb{R}^{3})$ by (6.1). By integration by parts, $(v,b)$ satisfies (1.9) in the sense of Definition 6.1.
\end{proof}

\begin{proof}[Proof of Theorem 6.3]
The system (1.1) is invariant under the rotation around the $z$-axis. For constants $B_{\infty}$, $u_{\infty}$ parallel to $e_z$ and axisymmetric $\chi$, the system (1.9) as well as (A.1) are also invariant under the rotation. By the uniqueness, solutions in Proposition A.2 for $\varepsilon_j>0$ are axisymmetric for axisymmetric $v_0,b_0\in L^{2}_{\sigma,\textrm{axi}}(\mathbb{R}^{3})$. Their limits as $\varepsilon_j\to0$ are also axisymmetric.   
\end{proof}

\section{Taylor states in a ball}

We demonstrate that Taylor states in a ball are three. We show that all linear force-free fields are expressed by eigenfunctions of the Laplace operator subject to the Dirichlet boundary condition with spherical mean-zero by using the spherical coordinates and the poloidal-toroidal decomposition. The spherical mean-zero condition implies that the linear force-free fields with the smallest $f^{2}$ (Taylor states) are given by three eigenfunctions associated with the second smallest eigenvalue of the Dirichlet Laplacian.

\subsection{The spherical coordinates}

We use the spherical coordinates 

\begin{align*}
x_1&=\rho\sin \theta \cos\phi,  \\
x_2&=\rho\sin \theta \sin\phi, \\
x_3&=\rho\cos\theta,
\end{align*}\\
and the associated orthogonal frame 
 
\begin{equation*}
e_{\rho}=\left(
\begin{array}{c}
\cos\phi\sin \theta \\
\sin\phi\sin \theta \\
\cos\theta
\end{array}
\right),\quad 
e_\theta=\left(
\begin{array}{c}
\cos\phi\cos \theta \\
\sin\phi\cos \theta \\
-\sin\theta
\end{array}
\right),\quad 
e_{\phi}=\left(
\begin{array}{c}
-\sin\phi \\
\cos\phi \\
0
\end{array}
\right).
\end{equation*}\\
The basis $e_{\theta}$ and $e_{\phi}$ are the orthogonal frame on $\mathbb{S}^{2}$. We denote the vector field $U$ in $\mathbb{R}^{3}$ by 

\begin{align*}
U=U^{\rho}e_{\rho}+V,\quad V=U^{\theta}e_{\theta}+U^{\phi}e_{\phi}.
\end{align*}\\
We decompose the gradient in $\mathbb{R}^{3}$ with the surface gradient on $\mathbb{S}^{2}$

\begin{align*}
\nabla=e_{\rho}\partial_{\rho}+\frac{1}{\rho}\nabla_{\mathbb{S}^{2}},\quad 
\nabla_{\mathbb{S}^{2}}=e_{\theta}\partial_{\theta}+\frac{1}{\sin\theta}e_{\phi}\partial_{\phi}.
\end{align*}\\
We denote the $\pi/2$ counter-clockwise rotation of $V$ by $V^{\perp}=U^{\theta}e_{\phi}-U^{\phi}e_{\theta}$. We also set the operator $\nabla^{\perp}_{\mathbb{S}^{2}}=e_{\phi}\partial_{\theta}-(\sin\theta)^{-1}e_{\theta}\partial_{\phi}$. The divergence and the rotation of $U=U^{\rho}e_{\rho}+V$ are expressed as 

\begin{equation}
\begin{aligned}
\nabla \cdot U&=\frac{1}{\rho^{2}}\partial_{\rho}(\rho^{2}U^{\rho})+\frac{1}{\rho}\nabla_{\mathbb{S}^{2}} \cdot V,\\
\nabla \times U&=\frac{1}{\rho}(\nabla^{\perp}_{\mathbb{S}^{2}}\cdot V)e_{\rho}+\left(\partial_{\rho}+\frac{1}{\rho}\right)V^{\perp}-\frac{1}{\rho}\nabla^{\perp}_{\mathbb{S}^{2}}U^{\rho}.
\end{aligned}
\end{equation}\\
We express the Laplace operator in $\mathbb{R}^{3}$ by using the Laplace--Beltrami operator on $\mathbb{S}^{2}$ 

\begin{align*}
\Delta =\partial_{\rho}^{2}+\frac{2}{\rho}\partial_{\rho}+\frac{1}{\rho^{2}}\Delta_{\mathbb{S}^{2}},\quad \Delta_{\mathbb{S}^{2}}\varphi=\frac{1}{\sin\theta}\partial_{\theta}(\partial_{\theta}\varphi \sin\theta)+\frac{1}{\sin^{2}\theta}\partial_{\phi}^{2}\varphi.
\end{align*}\\
We use the identities 

\begin{equation}
\begin{aligned}
 \nabla \times (\nabla \times (x\varphi))&=-x\Delta \varphi+\nabla ( \varphi+x\cdot \nabla \varphi),\\
-\Delta_{\mathbb{S}^{2}}\varphi&=x\cdot \nabla \times (\nabla \times (x\varphi)).
\end{aligned}
\end{equation}\\
The first identity follows an expression of the Laplace operator in $\mathbb{R}^{3}$ using the rotation and the divergence. The second identity follows the multiplication of $x$ by the first identity.

The eigenvalues of the Laplace--Beltrami operator are $-n(n+1)$ for $n\in \mathbb{N}\cup \{0\}$ and its multiplicity is $2n+1$. The normalized spherical harmonics $\{Y^{m}_{n}\ |\ n\in \mathbb{N}\cup\{0\},\ -n\leq m\leq n\}$ make the complete orthonormal basis on $L^{2}(\mathbb{S}^{2})$. 

We define the Sobolev space $H^{s}(\mathbb{S}^{2})$ for $s\geq 0$ by the space of all functions $f\in L^{2}(\mathbb{S}^{2})$ such that its homogeneous extension to $\mathbb{R}^{3}\backslash \{0\}$ belongs to $H^{s}_{\textrm{loc}}(\mathbb{R}^{3}\backslash\{0\})$, i.e., $\bar{f}(x)=f(\theta,\phi)\in H^{s}_{\textrm{loc}}(\mathbb{R}^{3}\backslash\{0\})$ for $x=\rho e_{\rho}$. By the regularity theory in $\mathbb{R}^{3}$, the condition $\Delta_{\mathbb{S}^{2}}^{k}f\in L^{2} (\mathbb{S}^{2})$ for $k\geq 0$ is equivalent to $f\in H^{2k}(\mathbb{S}^{2})$.

We expand the function $f\in L^{2}(\mathbb{S}^{2})$ as

\begin{align*}
f=\sum_{n\geq 0,\ |m|\leq n}f^{m}_{n}Y^{m}_{n}.
\end{align*}\\
The $0$-th component $f^{0}_0$ vanishes for mean-zero functions on $\mathbb{S}^{2}$ by $Y^{0}_{0}=1$. We define the inverse operator for mean-zero functions by 

\begin{align*}
(-\Delta_{\mathbb{S}^{2}})^{-1} f=\sum_{n\geq 1,\ |n|\leq m}\frac{1}{n(n+1)}f^{m}_{n}Y^{m}_{n}.
\end{align*}\\
The function $\varphi=(-\Delta_{\mathbb{S}^{2}})^{-1} f$ satisfies $-\Delta_{\mathbb{S}^{2}}\varphi=f$ and $\varphi\in H^{2}(\mathbb{S}^{2})$. For $f\in H^{2k}(\mathbb{S}^{2})$ and $k\geq 0$, we have $\varphi \in H^{2k+2}(\mathbb{S}^{2})$.

\subsection{The poloidal-toroidal decomposition}
According to \cite[2.3.2]{Moffatt19}, we consider divergence-free vector fields in the ball $\Omega=\{x\in \mathbb{R}^{3}\ |\ |x|<R \}$ with radius $R>0$ expressed as

\begin{equation}
U=\nabla \times (\nabla \times (x P))+\nabla \times (xT), 
\end{equation}\\
with functions $P$ and $T$. The toroidal vector field $\nabla \times (xT)=-x\times \nabla T$ is a tangential vector field on the sphere $\{|x|=\rho\}$ for all $0<\rho<R$. The poloidal vector field has a non-zero radial component. By the identity (B.2$)_2$, 

\begin{align*}
-\Delta_{\mathbb{S}^{2}}P=x\cdot \nabla \times (\nabla \times (x P)).
\end{align*}\\
Adding radial functions to the functions $P$ and $T$ provides the same vector field $U$. We normalize potential functions and choose mean-zero functions 

\begin{equation}
\begin{aligned}
\int_{|x|=\rho}Pd H=0,\quad \int_{|x|=\rho}Td H=0,\quad 0<\rho<R.
\end{aligned}
\end{equation}\\
By taking curl to (B.3) and using the identity (B.2$)_1$, 

\begin{equation}
\nabla \times U=\nabla \times (\nabla \times (x T))-\nabla \times (x\Delta P). 
\end{equation}\\
The functions $P$ and $T$ satisfy the following equations on the sphere $\{|x|=\rho\}$ for each $0<\rho<R$:

\begin{equation}
\begin{aligned}
-\Delta_{\mathbb{S}^{2}}P=U\cdot x,\quad 
-\Delta_{\mathbb{S}^{2}}T=\nabla \times U\cdot x.
\end{aligned}
\end{equation}\\
Conversely, for a given smooth divergence-free vector field $U$ in $\Omega$, a vector potential $A$ exists such that $U=\nabla \times A$. The surface integrals of $U\cdot x= \nabla \times A\cdot x$ and $\nabla \times B\cdot x$ on the sphere $\{|x|=\rho\}$ vanish by Stokes's theorem. Thus the  functions 

\begin{align}
P=(-\Delta_{\mathbb{S}^{2}})^{-1}(U\cdot x),\quad T=(-\Delta_{\mathbb{S}^{2}})^{-1}(\nabla \times U\cdot x),
\end{align}\\
are smooth and satisfy the equations (B.6) and the condition (B.4). We show the representation (B.3) by using the non-existence of harmonic vector fields on the sphere.\\

\begin{prop}
Let $U$ be a smooth divergence-free vector field in $\Omega$ such that $U\cdot x=0$ and $\nabla \times U\cdot x=0$. Then $U=0$.
\end{prop}

\vspace{5pt}

\begin{proof}
We write $U=U^{\rho}e_{\rho}+V$ by the spherical coordinates. The identities (B.1) and the conditions $U\cdot x=0$ and $\nabla \times U\cdot x=0$ imply that $U^{\rho}=0$, $\nabla_{\mathbb{S}^{2}}\cdot V=0$ and $\nabla_{\mathbb{S}^{2}}^{\perp}\cdot V=0$. Thus, the tangential vector field $V$ is a harmonic vector field on the sphere $\{ |x|=\rho\}$ and $V=0$ follows.
\end{proof}

\begin{lem}[Poloidal-toroidal decomposition]
For a solenoidal vector field $U$ expressed as (B.3) with smooth functions $P$ and $T$ satisfying (B.4), $P$ and $T$ are uniquely determined by (B.7).
Conversely, for a smooth solenoidal vector field $U$ in $\Omega$, (B.3) holds with functions $P$ and $T$ defined by (B.7).    
\end{lem}

\vspace{5pt}

\begin{proof}
The first assertion follows from the fact that the Laplace--Beltrami operator is invertible for mean-zero functions. To demonstrate the latter assertion, we observe that the vector field $\tilde{U}=\nabla \times (\nabla \times (xP))+\nabla \times (x T)$ satisfies 

\begin{align*}
\tilde{U}\cdot x&=-\Delta_{\mathbb{S}^{2}}P=U\cdot x,\\
\nabla \times \tilde{U}\cdot x&=-\Delta_{\mathbb{S}^{2}}T=\nabla \times U\cdot x.
\end{align*}\\
The vector field $U-\tilde{U}$ is divergence-free and satisfies $(U-\tilde{U})\cdot x=\nabla \times (U-\tilde{U})\cdot x=0$. We apply Proposition B.1 and conclude that $\tilde{U}=U$.
\end{proof}

\subsection{Taylor states}
We recall the Laplace operator eigenvalue problem in a ball:  

\begin{equation}
\begin{aligned}
-\Delta \varphi&=\lambda \varphi \quad \textrm{in}\ \Omega,\\
\varphi&=0 \quad \textrm{on}\ \partial\Omega.
\end{aligned}
\end{equation}\\
All eigenvalues are countable positive constants; we denote them as $0<\lambda_1<\lambda_2<\cdots$. We expand the eigenfunction $\varphi$ associated with $\lambda>0$ by spherical harmonics 

\begin{align*}
\varphi(x)=\sum_{n\geq 0,\ |m|\leq n}\varphi^{m}_{n}(\rho)Y^{m}_{n}(\theta,\phi).
\end{align*}\\
By the equation (B.8),

\begin{align*}
0&=(\Delta +\lambda)\varphi=\sum_{n\geq 0,\ |m|\leq n}\left(\partial_{\rho}^{2}\varphi^{m}_{n}+\frac{2}{\rho}\partial_{\rho}\varphi^{m}_{n}+\left(\lambda-\frac{n(n+1)}{\rho^{2}}\right)\varphi^{m}_{n}  \right)Y^{m}_{n},\\
0&=\sum_{n\geq 0,\ |m|\leq n}\varphi^{m}_{n}(R)Y^{m}_{n}.
\end{align*}\\
Each component vanishes since $\{Y^{m}_{n}\}$ is an orthonormal basis on $L^{2}(\mathbb{S}^{2})$. By changing the variable, $\iota(s)=\sqrt{s}\varphi^{m}_{n}(s/\sqrt{\lambda})$ satisfies the Bessel's differential equation

\begin{align*}
\iota''+\frac{1}{s}\iota'+\left(1-\frac{\left(n+\frac{1}{2}\right)^{2} }{s^{2}}\right)\iota&=0,\quad 0<s<\sqrt{\lambda}R,\\
\iota(\sqrt{\lambda}R)&=0,
\end{align*}\\
and hence agrees with a constant multiple of $J_{n+1/2}(s)$. Thus, all eigenvalues to (B.8) are expressed as 

\begin{align*}
\lambda=\left(\frac{j_{n+1/2,k}}{R}\right)^{2},\quad n\in \mathbb{N}\cup\{0\},\ k\in \mathbb{N},
\end{align*}\\
with the $k$-th zero point $j_{n+1/2,k}$ of $J_{n+1/2}$ and the associated eigenfunctions form 

\begin{align*}
\varphi=\frac{1}{\sqrt{\rho}}J_{n+1/2}(\sqrt{\lambda}\rho)Y^{m}_{n},\quad |m|\leq n.
\end{align*}\\
By the zero points order, e.g., \cite[9.5.2]{AS64}, 

\begin{align*}
j_{\nu,1}<j_{\nu+1,1}<j_{\nu,2}<j_{\nu+1,2}<j_{\nu,3}<\cdots,
\end{align*}\\
the smallest zero point $j_{1/2,1}$ and the second smallest zero point $j_{3/2,1}$ provide the principal eigenvalue $\lambda_1=(j_{1/2,1}/R)^{2}$ and the second least eigenvalue $\lambda_{2}=(j_{3/2,1}/R)^{2}$. The principal eigenvalue $\lambda_1$ is simple, and the associated eigenfunction is radially symmetric. The second least eigenvalue $\lambda_{2}$ has a multiplicity of three. The associated eigenfunctions can be expressed as 

\begin{align}
\varphi_i=\frac{1}{\sqrt{\rho}}J_{3/2}(\sqrt{\lambda_2}\rho)\frac{x_i}{\rho},\quad i=1,2,3.
\end{align}\\

\begin{lem}
All linear force-free fields in $L^{2}_{\sigma}(\Omega)$ are expressed as 

\begin{equation}
\begin{aligned}
U=\nabla \times (\nabla \times (x \varphi))+\nabla \times (xf\varphi), 
\end{aligned}
\end{equation}\\
by eigenfunctions to (B.8) for $\lambda=f^{2}$ satisfying 

\begin{align}
\int_{|x|=\rho}\varphi d H=0,\quad 0<\rho<R.
\end{align}\\
In particular, linear force-free fields with the smallest $f^{2}$ are expressed by the eigenfunctions (B.9) and $f^{2}=(j_{3/2,1}/R)^{2}$.
\end{lem}

\vspace{5pt}

\begin{proof}
Linear force-free fields in $L^{2}_{\sigma}(\Omega)$ are smooth in $\Omega$ by the Helmholtz equation $-\Delta U=f^{2}U$. We apply Lemma B.2 and express $U$ by the poloidal-toroidal decomposition (B.3) with the functions $P$ and $T$ satisfying the conditions (B.4). By the representation (B.5) and the equation (1.3$)_1$,

\begin{align*}
0=\nabla \times U-fU=\nabla \times (\nabla \times (x(T-fP)))-\nabla \times (x(\Delta P+fT)).
\end{align*}\\
By the uniqueness of the potential functions in Lemma B.2,

\begin{align*}
T-fP&=0,\\
\Delta P+fT&=0.
\end{align*}\\
By the boundary condition (1.3$)_2$ and (B.6), $-\Delta_{\mathbb{S}^{2}}P=0$ on the sphere $\partial\Omega=\{|x|=R\}$. The mean-zero condition (B.4) implies $P=0$ on $\partial\Omega$.  Thus, $\varphi=P$ is an eigenfunction to the problem (B.8) for $\lambda=f^{2}$ satisfying (B.11). 

The smallest $f^{2}$ is not the principal eigenvalue $\lambda_1$ to the problem (B.8) since the associated radially symmetric eigenfunction does not satisfy the condition (B.11). The eigenfunctions (B.9) associated with the second smallest eigenvalue $\lambda_2$ fulfill the condition (B.11). Thus the linear-force free fields with the smallest $f^{2}=\lambda_2=(j_{3/2,1}/R)^{2}$ are expressed by (B.10) and (B.9). 
\end{proof}

\vspace{5pt}

\begin{thm}
The Taylor states with helicity $h>0$ (resp. $h<0$) in $\Omega$ are three: they have the proportionality factor $f^{+}_{1}=j_{3/2,1}/R$ (resp. $f^{-}_{1}=-f^{+}_{1}$) and are expressed as (B.10) with the functions (B.9).
\end{thm}

\vspace{5pt}

\begin{proof}
Taylor states with helicity $h>0$ are linear force-free fields with the smallest $f=f_{1}^{+}>0$ by Lemma 2.3. They are expressed as (B.10) with the eigenfunction (B.9) and $f_{1}^{+}=j_{3/2,1}/R$ by Lemma B.3.   
\end{proof}

\vspace{15pt}

\bibliographystyle{alpha}
\bibliography{ref}

\end{document}